%% file: CycleClassesEO.tex
\documentclass[leqno,10pt]{article}
\usepackage{a4wide,amsmath,amssymb,amsfonts,mathrsfs,exscale,graphics,amsthm,colonequals,comment,ifthen,url,breakurl,enumerate,etoolbox}

\usepackage[curve,matrix,arrow,cmtip]{xy}
\newdir^{ (}{{}*!/-3pt/\dir^{(}}   
\newdir_{ (}{{}*!/-3pt/\dir_{(}}   
\NoComputerModernTips

\newtoggle{thesis}
\togglefalse{thesis}

\numberwithin{equation}{section}

\theoremstyle{plain}
\newtheorem{theorem}[equation]{Theorem}

\newtheorem{lemma}[equation]{Lemma}
\newtheorem{corollary}[equation]{Corollary}
\newtheorem{proposition}[equation]{Proposition}

\newtheorem{intro-theorem}{Theorem}
\newtheorem{intro-corollary}[intro-theorem]{Corollary}
\newtheorem{intro-proposition}[intro-theorem]{Proposition}

\theoremstyle{definition}
\newtheorem{definition}[equation]{Definition}

\newtheorem{construction}[equation]{Construction}
\newtheorem{remark}[equation]{Remark}
\newtheorem{example}[equation]{Example}

\newtheorem{intro-definition}[intro-theorem]{Definition}
\newtheorem{intro-conjecture}[intro-theorem]{Conjecture}

\newcounter{listcounter}
\newcounter{deflistcounter}
\newcounter{equivcounter}

\newskip{\itemsepamount}
\newskip{\topsepamount}
\newenvironment{assertionlist}{%
  \begin{list}
    {\upshape (\arabic{listcounter})}
    {\setlength{\leftmargin}{18pt}
     \setlength{\rightmargin}{0pt}
     \setlength{\itemindent}{0pt}
     \setlength{\labelsep}{5pt}
     \setlength{\labelwidth}{13pt}
     \setlength{\listparindent}{\parindent}
     \setlength{\parsep}{0pt}
     \setlength{\itemsep}{\itemsepamount}
     \setlength{\topsep}{\topsepamount}
     \usecounter{listcounter}}}
  {\end{list}}

\newenvironment{definitionlist}{%
  \begin{list}
    {\upshape (\alph{deflistcounter})}
    {\setlength{\leftmargin}{18pt}
     \setlength{\rightmargin}{0pt}
     \setlength{\itemindent}{0pt}
     \setlength{\labelsep}{5pt}
     \setlength{\labelwidth}{13pt}
     \setlength{\listparindent}{\parindent}
     \setlength{\parsep}{0pt}
     \setlength{\itemsep}{\itemsepamount}
     \setlength{\topsep}{\topsepamount}
     \usecounter{deflistcounter}}}
  {\end{list}}

\newenvironment{equivlist}{%
  \begin{list}
    {\upshape (\roman{equivcounter})}
    {\setlength{\leftmargin}{18pt}
     \setlength{\rightmargin}{0pt}
     \setlength{\itemindent}{0pt}
     \setlength{\labelsep}{5pt}
     \setlength{\labelwidth}{13pt}
     \setlength{\listparindent}{\parindent}
     \setlength{\parsep}{0pt}
     \setlength{\itemsep}{\itemsepamount}
     \setlength{\topsep}{\topsepamount}
     \usecounter{equivcounter}}}
  {\end{list}}

% ...

\def\FilLF{\mathop{\tt FilLF}\nolimits}

\DeclareMathOperator{\GL}{GL}
\DeclareMathOperator{\GSp}{GSp}
\DeclareMathOperator{\GSpin}{GSpin}
\DeclareMathOperator{\PGL}{PGL}
\DeclareMathOperator{\SO}{SO}
\DeclareMathOperator{\SL}{SL}
\DeclareMathOperator{\Sp}{Sp}

 \DeclareMathOperator{\Spec}{Spec}
 
\DeclareMathOperator{\Lie}{Lie}
\DeclareMathOperator{\Gal}{Gal}
\DeclareMathOperator{\Ker}{Ker}

\DeclareMathOperator{\tor}{tor}

\let\inno\hookrightarrow

\newcommand{\defeq}{\colonequals}

\newcommand{\Gm}[1][\empty]{
  \ifthenelse{\equal{#1}{\empty}}
    {\mathbb{G}_m}
    {\mathbb{G}_{m,#1}}}

 \newcommand{\Gred}[1][\empty]{
  \ifthenelse{\equal{#1}{\empty}}
    {G^{\text{red}}}
    {G^{\text{red},#1}}}
    
\newcommand{\GZip}[1][\empty]{%
 \ifthenelse{\equal{#1}{\empty}}%
 {\mathop{\text{$G$-{\tt Zip}}}\nolimits}%
 {\mathop{\text{$#1$-{\tt Zip}}}\nolimits}%
}

\newcommand{\GZipFlag}[1][\empty]{%
 \ifthenelse{\equal{#1}{\empty}}%
 {\mathop{\text{$G$-{\tt ZipFlag}}}\nolimits}%
 {\mathop{\text{$#1$-{\tt ZipFlag}}}\nolimits}%
}

\newcommand\addots{\mathinner{\mkern1mu\raise0pt\vbox{\kern7pt\hbox{.}}\mkern2mu\raise3pt\hbox{.}\mkern2mu\raise6pt\hbox{.}\mkern1mu}}
\newcommand{\B}[1]{[#1\bs *]}
\newcommand{\Brh}{\mathtt{Brh}}
\newcommand{\Hdg}{\mathtt{Hdg}}
\newcommand{\set}[2]{\{\,#1\ \mid \  #2\,\}}
\newcommand{\sett}[2]{\{\,#1\ \mid \,\text{#2}\}}

\newcommand{\vdual}{{}^{\vee}}

 \newcommand{\Rep}[1][\empty]{
  \ifthenelse{\equal{#1}{\empty}}
    {\mathop{\text{\tt Rep}}\nolimits}
    {\mathop{\text{$#1$-{\tt Rep}}}\nolimits}}

\DeclareMathOperator{\Norm}{Norm}
\DeclareMathOperator{\rad}{rad}

\newcommand\lto{\longrightarrow}
\newcommand\ltoover[1]{\mathrel{\smash{\overset{#1}{\lto}}}}

\newcommand\varto[1]{\mathrel{\hbox to #1pt{\rightarrowfill}}}
\newcommand\vartoover[2]{\mathrel{\smash{\overset{#2}{\varto{#1}}}}}

\newcommand{\bijective}{\leftrightarrow}

\newcommand{\epi}{\twoheadrightarrow}

\newcommand{\mono}{\hookrightarrow}

\newcommand{\sends}{\mapsto}
\newcommand{\lsends}{\longmapsto}

\newcommand{\iso}{\overset{\sim}{\to}}
\newcommand{\liso}{\overset{\sim}{\lto}}

\newcommand{\BA}{{\mathbb{A}}}

\newcommand{\BC}{{\mathbb{C}}}

\newcommand{\BF}{{\mathbb{F}}}
\newcommand{\BG}{{\mathbb{G}}}

\newcommand{\BN}{{\mathbb{N}}}

\newcommand{\BP}{{\mathbb{P}}}
\newcommand{\BQ}{{\mathbb{Q}}}
\newcommand{\BR}{{\mathbb{R}}}
\newcommand{\BS}{{\mathbb{S}}}

\newcommand{\BZ}{{\mathbb{Z}}}

\newcommand{\QQbreve}{{\breve{\BQ}}}

\newcommand{\CA}{{\mathcal A}}

\newcommand{\CE}{{\mathcal E}}
\newcommand{\CF}{{\mathcal F}}

\newcommand{\CI}{{\mathcal I}}

\newcommand{\CP}{{\mathcal P}}

\newcommand{\CT}{{\mathcal T}}

\newcommand{\CW}{{\mathcal W}}
\newcommand{\CX}{{\mathcal X}}
\newcommand{\CY}{{\mathcal Y}}

\newcommand{\Btilde}{\tilde{B}}

\newcommand{\Etilde}{\tilde{E}}
\newcommand{\ftilde}{\tilde{f}}

\newcommand{\Gtilde}{\tilde{G}}

\newcommand{\Ptilde}{\tilde{P}}

\newcommand{\Qtilde}{\tilde{Q}}

\newcommand{\Stilde}{\tilde{S}}

\newcommand{\Ttilde}{\tilde{T}}

\newcommand{\Wtilde}{\tilde{W}}

\newcommand{\ztilde}{\tilde{z}}
\newcommand{\Ztilde}{\tilde{Z}}

\newcommand{\kgtilde}{\tilde{\kappa}}
\newcommand{\mgtilde}{\tilde{\mu}}

\newcommand{\Bscr}{{\mathscr B}}

\newcommand{\Escr}{{\mathscr E}}

\newcommand{\Gscr}{{\mathscr G}}

\newcommand{\Lscr}{{\mathscr L}}

\newcommand{\Oscr}{{\mathscr O}}
\newcommand{\Pscr}{{\mathscr P}}

\newcommand{\Sscr}{{\mathscr S}}
\newcommand{\Tscr}{{\mathscr T}}

\newcommand{\Vscr}{{\mathscr V}}

\newcommand{\Xscr}{{\mathscr X}}
\newcommand{\Yscr}{{\mathscr Y}}

\newcommand{\Ebar}{\bar{E}}

\renewcommand{\hbar}{\bar{h}}

\newcommand{\kbar}{\bar{k}}

\newcommand{\sbar}{\bar{s}}

\newcommand{\xbar}{\bar{x}}
\newcommand{\ybar}{\bar{y}}
\newcommand{\zbar}{\bar{z}}

\newcommand{\Gbf}{{\bf G}}

\newcommand{\Lbf}{{\bf L}}
\newcommand{\Pbf}{{\bf P}}

\newcommand{\Sbf}{{\bf S}}
\newcommand{\Tbf}{{\bf T}}

\newcommand{\Xbf}{{\bf X}}

\newcommand{\bs}{\backslash}

\def\UIsom{\mathop{\underline{\rm Isom}}\nolimits}

\newcommand{\leftexp}[2]{{\vphantom{#2}}^{#1}{#2}}

\let\phi\varphi

\DeclareMathOperator{\inn}{int}

\DeclareMathOperator{\Par}{Par}
\DeclareMathOperator{\Cent}{Cent}

\DeclareMathOperator{\Res}{Res}

\DeclareMathOperator{\id}{id}

\DeclareMathOperator{\Sym}{Sym}
\DeclareMathOperator{\codim}{codim}

\newcommand{\PhiQ}[2]{[#1 {}_\varphi\!\!\bs #2]}
\let\into\hookrightarrow

\newcommand{\twomatrix}[4]{\begin{pmatrix} #1 & #2 \\ #3 & #4 \end{pmatrix}}
\newcommand{\twosmallmatrix}[4]{\left(\begin{smallmatrix} #1 & #2 \\ #3 & #4 \end{smallmatrix}\right)}

\begin{document}
\title{Tautological rings of Shimura varieties and cycle classes of Ekedahl-Oort strata}
\author{Torsten Wedhorn \and Paul Ziegler}


\maketitle


\noindent{\scshape Abstract.\ }
We define the tautological ring as the subring of the Chow ring of a Shimura variety generated by all Chern classes of all automorphic bundles. We explain its structure for the special fiber of a good reduction of a Shimura variety of Hodge type and show that it is generated by the cycle classes of the Ekedahl-Oort strata as a vector space. We compute these cycle classes. As applications we get the triviality of $\ell$-adic Chern classes of flat automorphic bundles in characterstic $0$, an isomorphism of the tautological ring of smooth toroidal compactifications in positive characteristic with the rational cohomology ring of the compact dual of the hermitian domain given by the Shimura datum, and a new proof of Hirzebruch-Mumford proportionality for Shimura varieties of Hodge type.

\medskip



\input{Ch0CycleEO.tex}
\input{Ch1CycleEO.tex}
\input{Ch2CycleEO.tex}

\input{Ch3CycleEO.tex}

\input{Ch4CycleEO.tex}

\input{Ch5CycleEO.tex}

\input{Ch7CycleEO.tex}

\input{Ch8CycleEO.tex}


\bibliographystyle{amsalpha}
\bibliography{references}

\end{document}

%% file: Ch0CycleEO.tex

\section*{Introduction}

\subsubsection*{Tautological rings}

The Chow ring $A^{\bullet}(\Sbf_K)$ (always with rational coefficients) of a Shimura variety $\Sbf_K$ is still a very mysterious object. Here we study the subring generated by all Chern classes of all automorphic bundles on the Shimura variety or on a smooth toroidal compactification of the Shimura variety. In the Siegel case this subring was already studied by van der Geer and Ekedahl (\cite{VanDerGeer_CyclesAbVar}, \cite{EkedahlVanDerGeer_CycleAbVar}). Following their terminology, we call this subring the \emph{tautological ring}\footnote{One could argue against this terminology: By analogy to the notion of tautological rings of moduli spaces of curves, the tautological ring should be the subring ``generated by all interesting classes''. But with our definition there are many interesting classes, for instance those of special subvarieties, that are not contained in the tautological ring.} of the Shimura variety or of some toroidal compactification.

More precisely, let $(\Gbf,\Xbf)$ be a Shimura datum with attached Shimura variety $\Sbf_K = Sh_K(\Gbf,\Xbf)$ for some sufficiently small open compact subgroup $K$ of $\Gbf(\BA_f)$. It defines a conjugacy class of cocharacters $\mu$ of $\Gbf$ whose field of definition is the reflex field. To simplify the notation here in the introduction let us assume that $\Gbf$ does not contain an $\BQ$-anisotropic and $\BR$-split torus in its center. This condition is automatic if $(\Gbf,\Xbf)$ is of Hodge type. The Borel embedding of the hermitian space $\Xbf$ into its compact dual $\Xbf\vdual$ induces a morphism
\[
\sigma\colon \Sbf_K \to \Hdg \defeq [\Gbf \bs \Xbf\vdual]
\]
of algebraic stacks over the reflex field $E$ of the Shimura variety (\cite[III]{Milne_CanonicalModel}). By definition, a vector bundle on $\Sbf_K$ is automorphic if it is isomorphic to the pullback of a vector bundle on $\Hdg$. Moreover it is flat, if it is obtained by pullback from a vector bundle on the classifying stack $\B{\Gbf}$, i.e., if it is induced by a finite dimensional representation of $\Gbf$ (see Section~\ref{TAUT} for details). The theory of canonical extensions of automorphic vector bundles means that there is a canonical extension of $\sigma$ to smooth toroidal compactifications $\Sbf_K^{\tor}$.

\begin{intro-definition}[Definition~\ref{DefTautologicalRing}]\label{DefTautIntro}
The \emph{tautological ring of $\Sbf_K$} (resp.~of $\Sbf^{\tor}_K$) is the image of the Chow ring of $\Hdg$ in the Chow ring of $\Sbf_K$ (resp.~of $\Sbf^{\tor}_K$) under pull back via $\sigma$.
\end{intro-definition}

In the Siegel case, the tautological ring is the subring generated by all Chern classes of the Hodge filtration in the de Rham cohomology of the universal abelian scheme (Example~\ref{TautologicalRingSiegel}) which is the definition of van der Geer in this case.

\subsubsection*{Ekedahl-Oort strata}

From now on we assume that $(\Gbf,\Xbf)$ is of Hodge type and that $p > 2$ is a prime of good reduction for the Shimura datum. Then the reductive group $\Gbf$ has a reductive model $\Gscr$ and hence the algebraic stack $\Hdg$ has a good integral model. Denote by $G$ the special fiber $\Gscr$. Hence $G$ is a reductive group over $\BF_p$. Moreover, since the Shimura variety is of Hodge type, there are canonical smooth integral models $\Sscr_K$ and $\Sscr^{\tor}_K$ with special fibers $S_K$ and $S^{\tor}_K$ by the work of Vasiu (\cite{Vasiu_IntegralModels}), Kisin (\cite{Kisin_IntegralModels}), Kim and Madapusi Pera (\cite{KimMadapusi_2AdicIntegralModels}, \cite{Madapusi_ToroidalHodge}) such that the morphism $\sigma$ extends. Hence we also obtain in characteristic $p$ tautological rings of $S_K$ and $S^{\tor}_K$ as images under pull back maps
\begin{equation}\label{EqIntroSigma}
\sigma^*\colon A^{\bullet}(\Hdg) \to A^{\bullet}(S^{(\tor)}_K).
\end{equation}
In characteristic $p$ the work of Viehmann and the first author (\cite{VW_PEL}), of Zhang and Wortmann (\cite{Zhang_EOHodge}, \cite{Wortmann_MuOrd}), and of Goldring and Koskivirta (\cite{GoKo_HasseHeckeGalois}) shows that the morphism $\sigma$ factors into
\begin{equation}\label{EqIntroFactorSigma}
\sigma\colon S_K^{(\tor)} \vartoover{25}{\zeta^{(\tor)}} \GZip^{\mu} \vartoover{25}{\beta} \Hdg,
\end{equation}
where $\GZip^{\mu}$ is the stack of $G$-zips of type $\mu$ which was defined and studied by Pink and the authors (\cite{PWZ1}, \cite{PWZ2}). The stack $\GZip^{\mu}$ has a finite stratification by gerbes $Z_w$, where $w$ runs through a certain subset ${}^IW$ of the Weyl group $W$ of $G$ (see  Section~\ref{GZIP} for a reminder on $G$-zips). The locally closed subschemes
\[
S_w \defeq \zeta^{-1}(Z_w) \subseteq S_K, \qquad\qquad S_w^{\tor} \defeq \zeta^{\tor,-1}(Z_w) \subseteq S^{\tor}_K
\]
are by definition the Ekedahl-Oort strata of $S_K$ and $S^{\tor}_K$. As $\zeta$ is smooth by the work of Zhang (\cite{Zhang_EOHodge}), many results proved for $Z_w \subseteq \GZip^{\mu}$ in \cite{PWZ1} such as smoothness, a formula for its codimension, or closure relation of the strata are known to hold also for the Ekedahl-Oort strata $S_w$. As was explained to us by Beno\^it Stroh, the extension $\zeta^{\tor}$ is expected to be also smooth. In this case, the above properties also hold for the Ekedahl-Oort strata $S_w^{\tor}$ in a toroidal compactification. Using a deep result on the existence of Hasse invariants (\cite{GoKo_HasseHeckeGalois}, see also \cite{Boxer_Thesis} in the PEL case) we can also prove the following connectedness result on Ekedahl-Oort strata.

\begin{intro-theorem}[Theorem~\ref{ConnectedDimStrata}, Corollary~\ref{Length1Connected}]\label{Intro-Connected}
Suppose that $\zeta^{\tor}$ is smooth.
\begin{assertionlist}
\item
For all $j \geq 1$ the union of all EO-strata of dimension $\leq j$ is geometrically connected in each geometric connected component of $S^{\tor}_K$. 
\item
For Shimura varieties of PEL type the union of all EO-strata of dimension $\leq 1$ is geometrically connected in each geometric connected component of $S_K$.
\end{assertionlist}
\end{intro-theorem}

The first assertion seems to be new even in the Siegel case. Assertion~(2) was proved in the Siegel case by Oort (\cite{Oort_EO}).


\subsubsection*{The tautological ring and the Chow ring of the stack of $G$-zips}

By \eqref{EqIntroFactorSigma} the pullback $\sigma^*$ is a composition
\begin{equation}\label{EqIntroPullbackSigma}
\sigma^*\colon A^{\bullet}(\Hdg) \vartoover{25}{\beta^*} A^{\bullet}(\GZip^{\mu}) \vartoover{25}{\zeta^{(\tor),*}} A^\bullet(S^{(\tor)}_K).
\end{equation}
Brokemper has given in \cite{Brokemper_ChowZip} two descriptions for $A^{\bullet}(\GZip^{\mu})$. From his multiplicative description (recalled in Proposition~\ref{ChowGZip} below) we deduce:

\begin{intro-theorem}[Theorem~\ref{PullbackKey}, Lemma~\ref{RationalCI}, Corollary~\ref{ChowGZipCohomology}]\label{IntroThm1}
\begin{assertionlist}
\item
The map $\beta^*$ is surjective and its kernel is generated by all Chern classes in degree $> 0$ of vector bundles attached to representations of the group $G$. In particular, the tautological ring of $S_K$ (resp.~$S_K^{\tor}$) is equal to the image of $\zeta^*$ (resp.~$\zeta^{\tor,*}$).
\item
The graded $\BQ$-algebra $A^{\bullet}(\GZip^{\mu})$ is isomorphic to the rational cohomology ring $H^{2\bullet}(\Xbf\vdual,\BQ)$ of the complex manifold $\Xbf\vdual$.
\end{assertionlist}
\end{intro-theorem} 

As a consequence we obtain:

\begin{intro-corollary}[Theorem~\ref{TrivialFlatBundle}, Corollary~\ref{TrivialEtaleChern}]\label{IntroCor1}
In characteristic $p > 0$, Chern classes of flat automorphic bundles are zero in degree $> 0$. In characteristic $0$, the $\ell$-adic Chern classes of flat automorphic bundles are zero in degree $> 0$.
\end{intro-corollary}

The characteristic $0$ variant for compact Shimura varieties (not necessarily of Hodge type) is one of the main results in \cite{EsnaultHarris_AutomorphicVB} by Esnault and Harris.

One particular important line bundle is the Hodge line bundle $\omega^{\flat}(\iota) \in {\rm Pic}(\GZip^{\mu})$ associated to an embedding $\iota$ of $(\Gbf,\Xbf)$ in the Siegel Shimura datum. Its pullback to the Shimura variety is the determinant of the Hodge filtration of the ``universal'' abelian scheme attached to $\iota$. Combining Corollary~\ref{IntroCor1} with a result of Goldring and Koskivirta (\cite{GoKo_Quasiconstant}) one gets:

\begin{intro-corollary}[Proposition~\ref{HodgeLineBundlesIndep}]\label{IntroCor2}
Suppose that the adjoint group of $\Gbf$ is $\BQ$-simple. Then $c_1(\omega^{\flat}(\iota)) \in A^1(\GZip^{\mu})$ does not depend on $\iota$ up to multiplication with positive rational numbers.
\end{intro-corollary}

The second description of $A^{\bullet}(\GZip^{\mu})$ by Brokemper (recalled in Proposition~\ref{ChowGZipNaive}) shows that the classes $[\overline{Z}_w]$ of closed Ekedahl-Oort strata form a $\BQ$-basis of $A^{\bullet}(\GZip^{\mu})$. Hence the tautological rings in characteristic $p$ are generated as a $\BQ$-vector space by the classes of the closed Ekedahl-Oort strata which are indexed by by the subset ${}^IW$ of the geometric Weyl group $W$ of $\Gbf$. 

In fact, it is also possible to define classes in $A^{\bullet}(\GZip^{\mu})$ whose pullbacks to the Shimura variety $S_K$ are the classes of the closures of the Newton strata or of central leaves in $S_K$. In particular these classes are also contained in the tautological ring of $S_K$. This will be pursued in an other paper. 

The technical heart of the paper is to relate both descriptions of Brokemper:

\begin{intro-theorem}[Subsection~\ref{HowCompute}]\label{IntroThm2}
Let $G$ be a reductive group over $\BF_p$, where $p$ is any prime ($p = 2$ included), and let $\mu$ be a cocharacter of $G$. There is a concrete algorithm how to express for $w \in {}^IW$ the cycle class $[\overline{Z}_w] \in A^{\bullet}(\GZip^{\mu})$ of the closure of an Ekedahl-Oort strata as a polynomial in Chern classes of vector bundles on $\Hdg$.
\end{intro-theorem}

By pulling back to the Shimura variety or to a toroidal compactification (for $p > 2$) we get the same descriptions of cycle classes of Ekedahl-Oort strata in the Chow rings of $S_K$ and $S^{\tor}$.

To obtain a description as in Theorem~\ref{IntroThm2} we follow a strategy already used by Ekedahl and van der Geer in \cite{EkedahlVanDerGeer_CycleAbVar} in the Siegel case, albeit using a somewhat different language. Following \cite{GoKo_HasseHeckeGalois} we construct a commutative diagram
\begin{equation}\label{EqKeyIntro}
\begin{aligned}\xymatrix{
\GZipFlag^{\mu} \ar[rr]^-{\psi} \ar[d]_{\pi} & & \Brh_G \ar[d]^{\gamma} \\
\GZip^{\mu} \ar[rr]^{\beta} & & \Hdg,
}\end{aligned}
\end{equation}
where $\GZipFlag^{\mu}$ is the algebraic stack of flagged $G$-zips defined by Goldring and Koskivirta (\cite{GoKo_ZipFunctoriality}, \cite{GoKo_HasseHeckeGalois}) and where $\Brh_G = [B \bs G/B] = \B{B} \times_{\B{G}} \B{B}$ is the Bruhat stack. Here $B \subseteq G$ is a Borel subgroup. Then we proceed in three steps.

\paragraph*{1. Calculation of cycles of Schubert varieties.}
In $A^{\bullet}(\Brh_G)$ there are the classes $[\overline{\Brh}_w]$ of Schubert varieties for $w \in W$. They can be computed as follows. The cycle class of the smallest Schubert variety $[\Brh_e]$ is the class of the diagonal and can be computed by a result of Graham. Then one defines explicit operators $\delta_w$ such that $[\overline{\Brh}_w] = \delta_w([\Brh_e])$. This is certainly well known but to our surprise we found this only explained in the literature for classical groups (and sometimes only over $\BC$). Hence we explain this for arbitrary split reductive group over an arbitrary field in Section~\ref{BRUHAT}.

\paragraph*{2. Pullback to $\GZipFlag^{\mu}$.}
One describes the pullback via $\psi$ explicitly and obtains a description for the cycle classes in $A^{\bullet}(\GZipFlag^{\mu})$ of the closures of $Z^{\emptyset}_w \defeq \psi^{-1}({\Brh}_w)$ (Subsections~\ref{ChowGZipFlag} and~\ref{PullBackOfKey}).

\paragraph*{3. Push down to $\GZip^{\mu}$.}
By a result of Koskivirta (\cite{Koskivirta_NormalEO}), $\pi$ induces for $w \in {}^IW$ a finite \'etale map $Z_w^{\emptyset} \to Z_w$. If $\gamma(w)$ is its degree, we obtain
\[
[\overline{Z}_w] = \gamma(w) \pi_*([\overline{Z}^\emptyset_w]).
\]
Using a result of Brion (\cite{BrionGysin}) one can describe $\pi_*$ explicitly (Theorem~\ref{GysinFormula}). Moreover, we explain how to compute $\gamma(w)$ as the number of $\BF_p$-rational points of the flag variety of an explicitly given form of a Levi subgroup of $G$ (Subsection~\ref{Gammaw}). 

\bigskip

We also introduce the flag space over the Shimura variety (Subsection~\ref{FlagSpace}) that classifies -- roughly spoken -- refinements of the Hodge filtration. This generalizes a construction of Ekedahl and van der Geer in \cite{EkedahlVanDerGeer_CycleAbVar}. It carries a stratification obtained by pullback from the stratification of the Bruhat stack. From the analog properties of Schubert varieties we deduce that the closure of these strata are normal, Cohen-Macaulay, and with only rational singularities. This generalizes results from loc. cit.

\subsubsection*{Structure of the tautological ring}

By definition the tautological ring is a quotient of $A^{\bullet}(\Hdg)$, and $A^{\bullet}(\Hdg)$ can be described explicitly (Remark~\ref{TautScalarExt}). There is the following conjecture about the tautological ring.

\begin{intro-conjecture}\label{IntroConj}
The tautological ring of a smooth toroidal compactification $\Sbf_K^{\tor}$ (considered as a scheme over some splitting field of $\Gbf$) is isomorphic to the rational cohomology ring of the compact dual $\Xbf\vdual$.
\end{intro-conjecture}

We show that this conjecture is equivalent to the property that all Chern classes of positive degree of flat automorphic bundles vanish in the Chow ring of $\Sbf^{\tor}_K$ (Proposition~\ref{DescribeTautChar0}). This has also been shown in \cite[1.11]{EsnaultHarris_AutomorphicVB} if the Shimura variety is compact. We show that the conjecture always holds in  characteristic $p$.


\begin{intro-theorem}[Theorem~\ref{ZetaTorInj}]
The map of graded $\BQ$-algebras $H^{2\bullet}(\Xbf\vdual) \cong A^{\bullet}(\GZip^{\mu}) \to A^{\bullet}(S_K^{\tor})$ is injective.
\end{intro-theorem}

Finally, as an immediate application we obtain a very strong form of Hirzebruch-Mumford proportionality in positive characteristic (Theorem~\ref{HMCharp}). From this we deduce a new and purely algebraic proof of Hirzebruch-Mumford proportionality for Shimura varieties of Hodge type over $\BC$ (Corollary~\ref{HMChar0}). 




\subsubsection*{Structure of the paper}

The paper starts with a preliminary section in which we recall the notion of Chow groups of quotient stacks and some basic properties of these groups. Then the main body of the paper consists of two parts.

\medskip

The first part (Sections~\ref{BRUHAT}--\ref{CHOW}) explains how to compute cycle classes of Ekedahl-Oort strata in the Chow ring of the stack of $G$-zips of type $\mu$. This is a purely group theoretic part and everything is done for arbitrary reductive groups, arbitrary cocharacters, and in arbitrary positive characterstic $p \geq 2$.

In Section~\ref{BRUHAT} we explain how to calculate the cycle classes of Schubert varieties in the Bruhat stack of a split reductive group. All of this is well documented in the literature for classical groups. 

Section~\ref{GZIP} recalls the stack of $G$-zips and of $G$-zips ``endowed with a refinement of the Hodge filtration'' and defines the commutative diagram \eqref{EqKeyIntro}. In Section~\ref{CHOW} we explain what maps are induced from this diagram on Chow rings. This allows us to prove Theorem~\ref{IntroThm1} and to give an algorithm for the determination of cycle classes of Ekedahl-Oort strata in $A^{\bullet}(\GZip^{\mu})$ (Subsection~\ref{HowCompute}). The section closes with stating some easy functoriality properties for maps of reductive groups inducing an isomorphism on adjoint groups.

\medskip

In the second part of the paper (Sections~\ref{TAUT}--\ref{APP}) we apply the results from the first part to Shimura varieties of Hodge type. Here we have to make the assumption $p > 2$.

In Section~\ref{TAUT} we define the tautological ring for arbitrary Shimura varieties in characteristic $0$ and for Shimura varieties of Hodge type in characteristic $p$, where $p$ is a prime of good reduction.

In Section~\ref{EOTAUT} we recall the definition of Ekedahl-Oort strata and prove Theorem~\ref{Intro-Connected}. Here we also give the definition of the flag space over the Shimura variety and its stratification.

Then we prove in Section~\ref{APP} the triviality of Chern classes of flat automorphic bundles, the uniqueness of the class of a Hodge line bundle (up to positive scalar), our results on the structure of the tautological ring and on Hirzebruch-Mumford proportionality.

\medskip

In the final Section~\ref{EXAMPLE} we illustrate our results in the special cases of the Siegel Shimura variety, the Hilbert-Blumenthal variety, and Shimura varieties of Spin type.


\bigskip

\noindent{\scshape Acknowledgements.\ }
We are very grateful to Wushi Goldring for fruitful discussions with the first author, for pointing out the paper \cite{EsnaultHarris_AutomorphicVB} to us, and for pointing towards the Conjecture~\ref{IntroConj}. We are also grateful to Beno\^it Stroh for his explanation on the smoothness of the map $\zeta^{\tor}$. The first author is partially supported by the LOEWE Reseach Unit USAG. The second author was supported by the Swiss National Science Foundation.


\tableofcontents

\bigskip\bigskip

\subsection*{Notation}

All algebraic spaces and algebraic stacks are assumed to be quasi-separated and of finite type over their respective base.

\subsubsection*{Notation on reductive groups}

Let $k$ be a field and let $k^s$ be a separable closure. Suppose that $G$ is a reductive\footnote{A reductive group is always assumed to be connected.} group over $k$ and that $T \subseteq G$ is a maximal torus, defined over $k$. Then we denote by $W = (N_G(T)/T)(k^s)$ the Weyl group of $(G,T)$. It carries an action of $\Gamma = \Gal(k^s/k)$ by group automorphisms.

Now suppose that $G$ is quasi-split over $k$. Then we can choose a Borel pair $T \subseteq B \subseteq G$ defined over $k$. The choice of $B$ defines a subset $\Sigma \subset W$ of simple reflection and $\Gamma$ acts by automorphisms of the Coxeter system $(W,\Sigma)$. We denote by $\ell(\cdot)$ the length function and by $\leq$ the Bruhat order on the Coxeter system $(W,\Sigma)$. We choose representatives $\dot{w} \in G(k^s)$ of $w \in W$ such that $\dot(w_1w_2) = \dot{w}_1\dot{w}_2$ if $\ell(w_1w_2) = \ell(w_1) + \ell(w_2)$. We denote by $w_0 \in W$ the element of maximal length and by $e \in W$ the identity.

For any subset $K \subseteq \Sigma$ we denote by $W_K$ the subgroup of $W$ generated by $K$ and we set
\[
{}^KW \defeq \sett{w \in W}{$\ell(sw) > \ell(w)$ for all $s\in K$}
\]
which is a system of representatives of $W_K/W$. Let $w_{0,K}$ be the element of maximal length in $W_K$.

We denote by $\Phi \subset X^*(T_{k^s})$ (resp.~$\Phi\vdual \subset X_*(T_{k^s})$) the set of roots (resp.~of coroots) of $(G,T)_{k^s}$ and by $\Phi^+ \subset \Phi$ the set of positive roots given by $B$. The based root datum $(X^*(T_{k^s}),\Phi,X_*(T_{k^s}),\Phi\vdual,\Phi^+)$ and the Coxeter system $(W,\Sigma)$ does not depend on the choice of $(T,B)$ up to unique isomorphism and it is called ``the'' based root datum of $G$ and ``the'' Weyl group of $G$. For a set of simple reflections $K \subset \Sigma$ we denote by $\Phi_K \subset \Phi$ the set of roots that are in the $\BZ$-span of the simple roots corresponding to $K$ and let $\Phi^+_K \defeq \Phi^+ \cap \Phi_K$.

Let $\mu\colon \BG_{m,k^s} \to G_{k^s}$ be a cocharacter of $G$ defined over $k^s$. It gives rise to a pair of opposite parabolic subgroups $(P_-(\mu),P_+(\mu))$ and a Levi subgroup $L \defeq L(\mu) = P_-(\mu) \cap P_+(\mu)$ defined by the condition that $\Lie(P_-(\mu))$ (resp.~$\Lie(P_+(\mu))$) is the some of the non-positive (resp.~non-negative) weight space of $\mu$ in $\Lie(G)$. On $k^s$-valued points we have
\begin{align*}
P_+(\mu) &= \sett{g \in G}{$\lim_{t\to 0}\mu(t)g\mu(t)^{-1}$ exists}, \\
P_-(\mu) &= \sett{g \in G}{$\lim_{t\to \infty}\mu(t)g\mu(t)^{-1}$ exists}, 
\end{align*}
and $L = \Cent_G(\mu)$. 

We will also need to consider reductive groups over more general rings than a field. Hence let $S$ be a scheme. To simplify the notation we assume that $S$ is connected. Let $G$ be a reductive group scheme over $S$, i.e., a smooth affine group scheme over $S$ whose fibers are reductive groups. The map that attaches to $s \in S$ the isomorphism class of the based root datum of the geometric fiber $G_{\sbar}$ is locally constant (\cite[Exp.\ XXII, Prop. 2.8]{SGA3III}) and hence constant because we assumed $S$ to be connected. Hence we may again speak of ``the'' based root datum of $G$. Let $(W,\Sigma)$ be the Weyl group together with its set of simple reflections of this based root datum. Fix $I \subseteq \Sigma$ and let $\Par_I$ be the scheme of parabolic subgroups of $G$ of type $I$. It is defined \'etale locally on $S$ because $G$ is split \'etale locally on $S$ (\cite[Exp.\ XXII, Cor.\ 2.3]{SGA3III}).

If $\lambda\colon \BG_{m,S'} \to G_{S'}$ is a cocharacter of $G$ defined over some covering $S' \to S$ for the \'etale topology, then the constructions of the parabolic subgroups $P_+(\lambda)$ and $P_-(\lambda)$ over a field generalize to arbitrary schemes (\cite[4.1.7]{Conrad_Reductive}) and yield parabolic subgroups of $G_{S'}$. If $I$ is the type of $P_+(\lambda)$, we also write $\Par_{\lambda}$ instead of $\Par_I$.

In other words, we say that a parabolic subgroup $P$ of $G_{S'}$ is of type $\lambda$ if it is locally for the \'etale topology conjugate to $P_+(\lambda)$. In fact, then $P$ is already locally for the Zariski topology conjugate to $P_+(\lambda)$ by \cite[Exp.\ XXVI, Cor.\ 5.5]{SGA3III}. If $I = \set{s_{\alpha} \in \Sigma}{\langle \lambda,\alpha \rangle = 0}$, then the conjugacy class of $P_+(\lambda)$ depends only on $I$ and we also say that $P$ is of type $I$. We denote by $\Par_{G,\lambda}$ the scheme of parabolic subgroups of $G$ of type $\lambda$.


%% file: Ch1CycleEO.tex

\section{Chow groups of quotient stacks}

Let $k$ be a field. All Chow groups in the following will have $\BQ$-coefficients.

\subsection{Chow Rings of Smooth Quotient Stacks}

By a quotient stack we will mean a stack of the form $[G \bs X]$ where $X$ is a quasi-separated algebraic space of finite type over $\Spec(k)$ and $G$ is an affine group scheme of finite type over $\Spec(k)$ which acts on $X$ from the left.

For such $X$ and $G$ the equivariant Chow groups $A^G_i(X)$ are defined in \cite{EG1} as follows: Let $n=\dim X$ and $g=\dim G$. There exists a representation of $G$ on an $\ell$-dimensional $k$-vector space $V$ such that there exists an open subset $U$ of $V$ with complement of codimension strictly bigger than $n-i$ on which $G$ acts freely. For such a $U$ the quotient $G\bs (X\times U)$ by the diagonal action exists as an algebraic space and $A^G_i(X)$ is defined to be $A_{i+\ell-g}(G\bs (X\times U))$. By \cite{EG1}, this group does not depend on the choice of $U$ and in fact by \cite[Proposition 16]{EG1} the group $A_i([G\bs X])\defeq A^G_{i+g}(X)$ depends only on the stack $[G\bs X]$ and not on the chosen presentation of this stack. 

A quotient stack is smooth if it admits a presentation as above with $X$ smooth. Suppose that $X$ is in addition separated and equi-dimensional of dimension $n$. In this case for $A^i([G\bs X]) \defeq A_{n-g-i}([G \bs X])$ on the graded vector space $A^\bullet([G \bs X]) \defeq \oplus_{i\geq 0}A^i([G \bs X])$ there is a naturally defined cup product turning this group into a graded $\BQ$-algebra (\cite[4]{EG1}). This construction has been generalized to arbitrary smooth algebraic stacks of finite type over a field by Kresch \cite{Kresch_Cycle}. Here we will need only the case of smooth quotient stacks.

By \cite[Proposition 3]{EG1}, the equivariant Chow groups have the same functoriality properties as the usual Chow groups for $G$-equivariant morphisms $X\to Y$. Every representable morphisms $\CX \to \CY$ of quotient stacks arises in this way: For a presentation $\CY=[G\bs Y]$ take $X=\CX\times_\CY X$. This is a $G$-torsor over $\CX$ so that $\CX=[G\bs X]$ and by assumption it is representable by an algebraic space. This shows that $A^\bullet(\_)$ is contravariantly functorial for representable morphisms of quotient stacks and covariantly functorial for proper representable morphisms of quotient stacks. In fact, by \cite{Kresch_Cycle} it is also contravariantly functorial for flat (not necessarily representable) morphisms of smooth quotient stacks.

In this paper we will mainly use only the following types of morphisms between quotient stacks. For all of them $A^\bullet(\_)$ is contravariantly functorial.

\begin{example}\label{ExamplePullBack}
Let $G$ and $X$ be as above.
\begin{assertionlist}
\item
For a quasi-separated algebraic space $Y$ of finite type over $\Spec(k)$ every morphism $Y \to [G\bs X]$ is representable.
\item\label{ExamplePullBac3}
Let $\Xscr$ be any equi-dimensional algebraic stack over $k$. Then every morphism $\Xscr \to \B{G}$ is flat of constant dimension. In particular, if $\varphi\colon H \to G$ is a homomorphism of affine algebraic groups, the canonical morphism $\B{H} \to \B{G}$ is flat of relative dimension $\dim(G) - \dim(H)$.
\end{assertionlist}
\end{example}

Let $\varphi\colon G \to H$ be a map of algebraic groups over $k$. Let $f\colon X \to Y$ be a map of quasi-separated algebraic spaces of finite type over $k$. Suppose that $G$ acts on $X$ and that $H$ acts on $Y$ such that $f(gx) = \varphi(g)f(x)$ for $g \in G(R)$ and $x \in X(R)$, where $R$ is a $k$-algebra. Then $f$ induces a morphism of algebraic stacks $[f]\colon [G\bs X] \to [H \bs Y]$. 

\begin{lemma}\label{QuotientMapRep}
\begin{assertionlist}
\item
If $f$ is flat, then $[f]$ is flat.
\item
If $\varphi$ is a monomorphism, then $[f]$ is representable.
\end{assertionlist}
\end{lemma}

\begin{proof}
The first assertion is clear, and the second assertion is a very special case of \cite[Tag 04YY]{Stacks}.
\end{proof}

\begin{proposition}\label{GaloisInvariants}
Let $\Xscr = [G \bs X]$ be a smooth equi-dimensional quotient stack over $k$ and let $k'$ be a Galois extension of $k$ with Galois group $\Gamma$. Then the canonical homomorphism
\[
A^{\bullet}(\Xscr) \lto A^{\bullet}(\Xscr_{k'})^{\Gamma}
\]
is an isomorphism.
\end{proposition}

\begin{proof}
This is well known (e.g., \cite[1.3.6]{Brokemper_ChowZip}) if $\Xscr$ is an algebraic space. In general let $n \defeq \dim(X)$ and $g \defeq \dim(G)$. For $i \geq 0$ choose an $\ell$-dimensional representation $V$ of $G$ and an open subset $U$ of $V$ such that $G$ acts freely on $U$ and such that $V \setminus U$ has codimension $>i$. Then
\[
A^i(\Xscr) = A_{n-i+\ell-g}((X \times U)/G) \liso A_{n-i+\ell-g}((X_{k'} \times U_{k'})/G_{k'})^{\Gamma} = A^i(\Xscr_{k'})^{\Gamma}.\qedhere
\]
\end{proof}

We will also use the following result by Brokemper (\cite[1.4.7]{Brokemper_ChowZip}) that shows that $A^{\bullet}(\cdot)$ ``ignores unipotent actions''.

\begin{proposition}\label{IgnoreUnipotent}
Let $0 \to U \to G \to H \to 0$ be a split exact sequence of linear algebraic groups over $k$, where $U$ is a smooth connected unipotent group scheme over $k$. Choose a splitting. Let $X$ be a smooth quasi-projective $G$-scheme over $k$ and endow $X$ with the $H$-action via the chosen splitting. Then the pull back map
\[
A^{\bullet}([G \bs X]) \to A^{\bullet}([H \bs X])
\]
is an isomorphism of $\BQ$-algebras.
\end{proposition}


\subsection{A Variant of a Result of Leray-Hirsch}

We have the following Leray-Hirsch type result from \cite{EG2}:

\begin{proposition}\label{LH}
Let $\CX \to \CY$ be a representable morphism of smooth quotient stacks. Suppose that $\CY$ is connected. Assume that there exists a proper and smooth algebraic space $F$ over $\Spec(k)$ which admits a decomposition into locally closed algebraic subspaces isomorphic to $\BA^m$ such that $\CX\to \CY$ is a Zariski-locally trivial fibration with fiber $F$, i.e. such that $\CX\to \CY$ is Zariski-locally on $\CY$ isomorphic to $\CY\times F\to \CY$.

Then the following hold:
\begin{enumerate}[(i)]
\item For any $i\geq 0$, a family of elements of $A^i(\CX)$ restricts to a basis of $A^i(F')$ for every geometric fiber $F'$ of $\CX\to \CY$ if it does so for a single such fiber.
\item For any $i\geq 0$, there exists a family of elements of $A^i(\CX)$ which restricts to a basis of $A^i(F')$ for every geometric fiber $F'$ of $\CX\to \CY$. 
\item 
Let $(B_i \subset A^i(\CX))_{i\geq 0}$ be a collection of families as in $(ii)$. Then $\cup_{i\geq 0} B_i$ is a basis for $A^i(\CX)$ over $A^i(\CY)$.
\end{enumerate}
\end{proposition}
\begin{proof}
  By taking presentations $\CX=[X/G]$ and $\CY=[Y/G]$ as well as suitable $U\subset V$ for $G$ as above the claim reduces to an analogous claim for the morphism $(X\times U)/G \to (Y\times U)/G$ of algebraic spaces. Then the claim is given by \cite[Proposition 6]{EG2} as well as its proof and \cite[Lemma 1]{EG2}.
\end{proof}



%% file: Ch2CycleEO.tex

\section{The Bruhat Stack and Cycle Classes of Schubert Varieties}\label{BRUHAT}

From now we fix a split reductive group scheme $G$ over the field $k$, a Borel subgroup $B\subset G$ and a maximal torus $T\subset B$.

For an algebraic group $H$ over $k$, we denote the quotient stack $[H \bs \Spec(k)]$ by $\B{H}$. This is a smooth algebraic stack over $k$ of dimension $-\dim(H)$.


\subsection{Chow Rings of Classifying Stacks}

By \cite[Section 3.2]{EG1} the Chow rings $A^\bullet(\B{T})$, $A^\bullet(\B{B})$ and $A^\bullet(\B{G})$ are given as follows: Let $X^*(T)$ be the character group of $T$. Every $\chi \in X^*(T)$ induces a line bundle on $\B{T}$ and we get a morphism $X^*(T) \to A^1(\B{T})$ sending $\chi$ to the Chern class of this line bundle, c.f. \cite[Section 2.4]{EG1}. This extends to an isomorphism
\[
\Sym(X^*(T)_{\BQ}) \iso S \defeq A^{\bullet}(\B{T}).
\]
This in fact holds even with $\BZ$-coefficients. The canonical homomorphism $A^{\bullet}(\B{B}) \to S=A^{\bullet}(\B{T})$ is an isomorphism (Proposition~\ref{IgnoreUnipotent}). The action of the Weyl group $W$ on $T$ induces an action of $W$ on the abelian group $X^*(T)$ by $(w,\chi) \sends \chi \circ \inn(w^{-1})$. By functoriality we obtain an action of $W$ on the graded $\BQ$-algebra $S$. Then the natural homomorphism $A^\bullet(\B{G}) \to S$ yields an identification
\begin{equation}\label{EqDescribeChowReductive}
A^\bullet(\B{G}) \liso S^W
\end{equation}
(recall that we consider rational coefficients).

\begin{example}\label{ChowRingGLn}
Let $G = \GL_n$. Let $T \subseteq G$ be the diagonal torus identified with $\BG_m^n$. Then $X^*(T) =\BZ^n$ and $A^{\bullet}(\B{T}) = S = \BQ[t_1,\dots,t_n]$, where $(t_1,\dots,t_n)$ is the standard basis of $\BQ^n = X^*(T)_{\BQ}$. Moreover
\[
A^{\bullet}(\B{G}) = \BQ[t_1,\dots,t_n]^{S_n} = \BQ[\sigma_1,\dots,\sigma_n],
\]
where $\sigma_i$ is the elementary symmetric polynomial of degree $i$ in $t_1,\dots,t_n$.

If $\Xscr$ is a smooth quotient stack and $\Vscr$ is a vector bundle of rank $n$ on $\Xscr$, then $\Vscr$ corresponds to a flat morphism $\alpha_{\Vscr}\colon \Xscr \to \B{\GL_n}$ of algebraic stacks and the $i$-th Chern class of $\Vscr$ is given by
\[
c_i(\Vscr) = \alpha^*_{\Vscr}(\sigma_i) \in A^{\bullet}(\Xscr).
\]
The determinant $\det\colon \GL_n \to \BG_m$ induces a flat morphism $\B{\GL_n} \to \B{\BG_m}$ and hence a pullback morphism of $\BQ$-algebras
\[
\det\!{}^*\colon A^{\bullet}(\B{\BG_m}) = \BQ[t_1] \lto A^{\bullet}(\B{\GL_n}) = \BQ[\sigma_1,\dots,\sigma_n]
\]
which sends $t_1$ to $\sigma_1$. In particular
\[
c_1(\Vscr) = c_1(\det\Vscr).
\]
\end{example}

\begin{proposition}[{\cite[4.6]{Demazure}}] \label{AGB}
The homomorphism $S \to A^\bullet(G/B)$ sending $\chi \in X^*(T)$ to the Chern class of the induced line bundle on $G/B$ is surjective and its kernel is the ideal $J$ of $S$ generated by the elements of $S^W$ whose degree zero part vanishes.
\end{proposition}


\subsection{The Chow Ring of the Bruhat Stack}

We consider the \emph{Bruhat stack}
\[
\Brh \defeq \Brh_G \defeq \B{B}\times_{\B{G}} \B{B} \cong [B\bs G/B]
\]
together with its Bruhat decomposition into the locally closed substacks $\Brh_w= [ B\bs BwB/B]$. We are interested in the classes $[\overline{\Brh}_w]$ of the closures $\overline{\Brh}_w$ in $A^\bullet(\Brh)$.

\begin{proposition}\label{ChowBruhat}
\begin{assertionlist}
\item\label{ChowBruhat1}
The classes $[\overline{\Brh}_w]$ form a basis of $A^\bullet(\Brh)$ over $S$ for both natural homomorphisms $S=A^\bullet(\B{B})\to A^\bullet(\Brh)$.
\item\label{ChowBruhat2}
The natural homomorphism $S\otimes_{S^W} S \to A^\bullet(\Brh)$ is an isomorphism.
\end{assertionlist}
\end{proposition}

\begin{proof}
Consider $\Brh$ as a $G/B$-fibration over $\B{B}$ via say the first projection. Let $w_0\in W$ be the longest element. The subspace $\Brh_{w_0}\subset \Brh$ is open and given by the open Bruhat cell in $G/B$. It has natural structure as a $U^-$-torsor over $\B{B}$, where $U^-$ is the unipotent radical of the unique Borel subgroup $B^-$ of $G$ such that $B ^- \cap B = T$. The pushout along the open immersion $U^- \inno G/B$, $u \sends uB$, is isomorphic to $\Brh$. Any $U^-$-torsor is Zariski-locally trivial and hence so is $\Brh\to \B{B}$. Thus $\Brh \to \B{B}$ satisfies the conditions of Proposition \ref{LH}. Hence \ref{ChowBruhat1} follows from Proposition \ref{LH} using the fact that the closures of the Bruhat strata on $G/B$ give a basis of $A^\bullet(G/B)$ (c.f. \cite[Cor. to Proposition 1]{Demazure}). 

For \ref{ChowBruhat2} we can argue as follows: Consider $S\otimes_{S^W} S$ as an $S$ module via the first factor. The ring $S$ is free over ${S^W}$ of rank $|W|$ and hence so is $S\otimes_{S^W} S$ over $S$. The $S$-module $A^\bullet(\Brh)$ is free of rank $|W|$ by \ref{ChowBruhat1}. Thus $S\otimes_{S^W} S \to A^\bullet(\Brh)$ is a homomorphism of free $S$-modules of the same rank and it suffices to prove that it is surjective. By Proposition \ref{AGB} we can take homogenous elements $x_i$ in $S$ which map to a basis of $A^\bullet(G/B)$. Then by Proposition \ref{LH} the images of $1\otimes x_i$ in $A^\bullet(\Brh)$ form a basis of $A^\bullet(\Brh)$ over $S$. This proves surjectivity.
\end{proof}

To give a description of the class of $\overline{\Brh}_w$ in $S \otimes_{S^W} S$ we now proceed as follows. We first recall a formula for the class of the diagonal $\Brh_e$ in $\Brh$ by Graham. Then we define explicit operators $\delta_w$ on $A^{\bullet}(\Brh)$ such that $[\overline{\Brh}_w]=\delta_w[\Brh_e]$.


\subsection{The Class of the Diagonal} \label{DiagonalClass}
In \cite{GrahamDiagonal} the following formula for the class of the diagonal $\Brh_e$ in $\Brh$ is proved in case $k=\BC$. The proof given there can be readily adapted to arbitrary fields.

For $w\in W$ let $i_w\colon S\otimes_{S^W} S \to S,\; r\otimes r'\mapsto rw(r')$. The fact that $\Spec(S\otimes_{S^W} S)= \Spec(S) \times_{\Spec(S)/W} \Spec(S)$ implies that $\prod_{w\in W}i_w\colon S\otimes_{S^W} S \to \prod_{w\in W} S$ is injective.

\begin{theorem}[Graham] \label{DiagClass}
  The image of $[\Brh_e]$ under $i_e$ is $\prod_{\alpha\in \Phi^+}\alpha \in S$. The image of $[\Brh_e]$ under $i_w$ for $w\not = 1$ is zero.
\end{theorem}


\begin{example}\label{ExampleClassDiag}
We recall the results of Fulton on the class of the diagonal for classical groups. For the classical groups $\GL_{n}$, $\SO_{2n+1}$, $\Sp_{2n}$, and $\SO_{2n}$ we choose the standard maximal torus $T \cong \BG_m^n$ and Borel subgroup to obtain $S = \Sym(X^*(T)_{\BQ})$ and the roots in $X^*(T)$. We give elements $\widetilde{[\Brh_e]}$ in $S \otimes_{\BQ} S = \BQ[x_1,\dots,x_n,y_1,\dots,y_n]$, where $x_i$ and $y_i$ represent the same $\BG_m$-factor of $T$. The images of these elements in $S \otimes_{S^W} S$ are $[\Brh_e]$. As a reference we use \cite{Fulton_DetForm} where $y_j$ is denoted by $y_{n+1-j}$ and the Schubert variety corresponding to $w = e$ is denoted by $\Omega_{w_0}$.

We fix $n \in \BN$ and introduce the following polynomials:
\begin{align*}
\Phi \defeq \Phi_n &\defeq \prod_{1 \leq i<j \leq n}(x_i - y_j) \in S \otimes_{\BQ} S, \\
\Gamma_k &\defeq \det\bigl((c_{k+1+j-2i})_{1 \leq i,j \leq k}\bigr) \in \BQ[c_{-k+2}, c_{-k+3}, \dots, c_{2k+1}].
\end{align*}
For instance
\begin{equation}\label{EqGammaSmall}
\begin{aligned}
\Phi_1 = 1, &\qquad \Phi_2 = x_1 - y_2,\\
\Gamma_1 = c_1, &\qquad \Gamma_2 = c_1c_2 - c_0c_3.
\end{aligned}
\end{equation}
We also let $\sigma_1, \dots, \sigma_n$ be the elementary symmetric polynomials in $n$ variables with $\deg(\sigma_i) = i$. We also set $\sigma_0 \defeq 1$.
\begin{assertionlist}
\item[($A_{n-1}$)]
Let $n \geq 2$. Then
\begin{equation}\label{DiagonalA}
\widetilde{[\Brh_e]} = \Phi_n.
\end{equation}
\item[($B_n$)]
Let $n \geq 2$. Then
\begin{equation}\label{DiagonalB}
\begin{aligned}
\widetilde{[\Brh_e]} &= \Phi_n \Gamma_n, \\
c_i &\defeq \begin{cases}
\frac{1}{2}\bigl(\sigma_i(x_1,\dots,x_n) + \sigma_i(y_1,\dots,y_n)\bigr),&\text{if $0 \leq i \leq n$}; \\
0,&\text{otherwise.}
\end{cases}
\end{aligned}
\end{equation}
\item[($C_n$)]
Let $n \geq 2$. Then
\begin{equation}\label{DiagonalC}
\begin{aligned}
\widetilde{[\Brh_e]} &= \Phi_n \Gamma_n, \\
c_i &\defeq \begin{cases}
\sigma_i(x_1,\dots,x_n) + \sigma_i(y_1,\dots,y_n),&\text{if $0 \leq i \leq n$}; \\
0,&\text{otherwise.}
\end{cases}
\end{aligned}
\end{equation}
\item[($D_n$)]
Let $n \geq 3$. Then
\begin{equation}\label{DiagonalD}
\begin{aligned}
\widetilde{[\Brh_e]} &= \Phi_n \Gamma_{n-1}, \\
c_i &\defeq \begin{cases}
\frac{1}{2}\bigl(\sigma_i(x_1,\dots,x_n) + \sigma_i(y_1,\dots,y_n)\bigr),&\text{if $0 \leq i \leq n-1$}; \\
0,&\text{otherwise.}
\end{cases}
\end{aligned}
\end{equation}
\end{assertionlist}
\end{example}


\subsection{The Chevalley Formula}
For $(\lambda,\mu) \in X^*(T)\times X^*(T)$ we have a natural line bundle $\Lscr_{\lambda,\mu}$ on $\Brh$ with Chern class $\lambda \otimes \mu$. The following gives a version of a classical formula of Chevalley in this context:

\begin{theorem}[{\cite[Theorem 5.2.2]{GoKo_HasseHeckeGalois}}]
\begin{assertionlist}
\item
The line bundle $\Lscr_{\lambda,\mu}$ has a global section section on $\Brh_w$ if and only if $\mu=w^{-1}\lambda $.
\item
The space $H^0(\Brh_w,\Lscr_{\lambda,w^{-1}\lambda})$ has dimension $1$.
\item
For $w \in W$ set $E_w \defeq \set{\alpha \in \Phi^+}{ws_{\alpha} < w, \ell(ws_{\alpha}) = \ell(w) - 1}$. The divisor of any non-zero section of $\Lscr_{\lambda,w^{-1}\lambda}$ on $\Brh_w$ is equal to 
      \begin{equation*}
        \sum_{\alpha\in E_w} \langle \lambda,\alpha^\vee \rangle[\overline{\Brh_{w s_\alpha}}].
      \end{equation*}
\end{assertionlist}    
\end{theorem}

Note that the formula in loc.~cit.~contains an additional minus sign because there the positive roots are defined by the opposite Borel subgroup.

For $w\in W$ and $\lambda \in X^*(T)$ this implies the following relation in $A^{\bullet}(\Brh)$:
\begin{equation}
  \label{eq:Chevalley}
  (\lambda \otimes w^{-1}(\lambda))[\overline{\Brh_w}] = \sum_{\alpha \in E_w} \langle \lambda,\alpha^\vee\rangle[\overline{\Brh_{w s_\alpha}}].
\end{equation}  


\subsection{The Operators $\delta_w$} \label{Deltaw}
We use certain operators on $S$ and $S\otimes_{S^W} S$: Let $n$ be the semi-simple rank of $G$ and $\alpha_1,\hdots,\alpha_n$ be the simple roots with respect to $T$ and $B$. For $1\leq i\leq n$ let $s_i \defeq s_{\alpha_i}$ be the simple reflection in $W$ corresponding to $\alpha_i$ and $B\subset P_i=B\cup Bs_iB$ the ``$i$-th minimal parabolic'' with root system $\{\pm \alpha_i\}$. 

\begin{construction}\label{Constructds1}
  Let $X_i\defeq \B{B}\times_{\B{P_i}}\B{B}$ and $p_1,p_2\colon X_i\to \B{B}$ the two projections. Then we define $\delta_i\colon S\to S$ to be the correspondence $p_{1,*}\circ p_2^*\colon S\to S$.  
\end{construction}

\begin{construction}\label{Constructds2}
  Let $1\leq i \leq n$. Consider $S$ as the ring of polynomial functions on $X^*(T)_{\BQ}$. For $f\in S$, the element $f-s_{\alpha_i}(f)$ of $S$ vanishes on the hyperplane in $X^*(T)_{\BQ}$ given by the vanishing of to coroot $\alpha_i^\vee$. Hence we obtain an element $\tilde\delta_i(f)\defeq (f-s_{\alpha_i}(f))/\alpha_i^\vee \in S$. This defines a $\BQ$-linear homomorphism $\tilde\delta_i\colon S \to S$.
\end{construction}

\begin{theorem}
\begin{assertionlist}
\item
For each $1\leq i \leq n$ we have $\delta_i=\tilde\delta_i$.
\item
For $w\in W$, one gets a well-defined operator $\delta_w$ on $S$ by letting $\delta_w=\delta_{i_1}\cdots \delta_{i_k}$ for any decomposition $w=s_{i_1}\cdots s_{i_k}$ with $k=\ell(w)$.
\end{assertionlist} 
\end{theorem}

\begin{proof}
In case $k=\BC$ and $G$ is semi-simple and simply connected, this is proven in \cite{BGG}. See \cite[Theorem 5.7]{BGG} for $(i)$ and \cite[Theorem 3.4]{BGG} for $(ii)$.

The general case can be deduced from this as follows: First, using the functoriality of the various constructions with respect to homomorphisms of reductive groups inducing an isomorphism on adjoint groups one can reduce to the case that $G$ is semi-simple and simply connected. Now let $\tilde k$ be another algebraically closed base field, $\tilde G$ the reductive group over $\tilde k$ with the same root datum as $G$, with $\tilde T$, $\tilde B$, $\tilde \Brh := \Brh_{\tilde G}$ etc the corresponding data for $\tilde G$. We have a natural $W$-equivariant isomorphism $A^\bullet(*/T) \cong A^\bullet(*/\tilde T)$ which induces an isomorphism $A^\bullet(\Brh)\cong A^\bullet(\widetilde \Brh)$. We claim that for $w\in W$, the classes of $\overline{\Brh}_w$ and $\overline{\widetilde{\Brh}}_w$ correspond to each other under this isomorphism. For $w=e$ this follows from Theorem \ref{DiagClass}. From this one deduces the claim by induction on $\ell(w)$ using \eqref{eq:Chevalley}.

By taking $\tilde k=\BC$ this implies the claim.
\end{proof}

\begin{remark}\label{Leibniz}
Description~\ref{Constructds2} shows that one has the following Leibniz type formula
\begin{equation}\label{EqLeibniz}
\begin{aligned}
\delta_i(fg) &= \frac{fg - s_{\alpha_i}(f)s_{\alpha_i}(g)}{\alpha_i^{\vee}} = \frac{(f - s_{\alpha_i}(f))g + s_{\alpha_i}(f)(g - s_{\alpha_i}(g))}{\alpha_i^{\vee}} \\
&= \delta_i(f)g + s_{\alpha_i}(f)\delta_i(g).
\end{aligned}
\end{equation}
\end{remark}

Now, for $w\in W$, we define an operator $\delta_w$ on $A^\bullet(\Brh)=S\otimes_{S^W} S$ by letting the $\delta_w$ just defined on $S$ act on the first factor. For $1\leq i \leq n$ the operator $\delta_i= \delta_{s_i}$ on $S\otimes_{S^W} S$ can also be described as follows: Let $\Brh_i$ be the following fiber product:
\begin{equation}\label{somediag}
\begin{aligned}
  \xymatrix{
    \Brh_i \ar[r]^{q_2}  \ar[d]^{q_1} & \Brh \ar[d] \\
    \Brh \ar[r] & BP_i \times_{\B{G}} \B{B} .
  }
\end{aligned}
\end{equation}
Then $\delta_{s_i}=q_{1,*}\circ q_2^*\colon S\otimes_{S^W} S \to S \otimes_{S^W} S$.

\begin{theorem} \label{diProp}
  Let $w\in W$ and $1\leq i\leq n$. Then $\delta_{s_i}[\overline{\Brh}_w]=[\overline{\Brh}_{s_i w}]$ if $\ell(s_i w)=\ell(w)+1$ and $\delta_{s_i}[\overline{\Brh}_w]=0$ otherwise. 
\end{theorem}

\begin{proof}
We let $P_i$ act on $P_i/B$ and $G/B$ by multiplication from the left and on products of these varieties by the diagonal action. Then the $P_i$-equivariant diagram
\begin{equation*}
  \xymatrix{
    P_i/B \times P_i/B \times G/B \ar[r]^-{\pi_{13}} \ar[d]^{\pi_{23}} & P_i/B \times G/B \ar[d]^{\pi_{2}}\\
      P_i/B \times G/B \ar[r]^{\pi_2} & G/B 
}
\end{equation*}
gives a presentation of \eqref{somediag}. Here the quotient morphism $P_i/B \times G/B \to \Brh= [B\bs G/B]$ sends $(pB,gB)$ to $(Bp^{-1}gB)$ and the preimage of $\Brh_w$ is the $P_i$-orbit
\[
O_w\defeq \set{(pB,gB) \in P_i/B \times G/B}{ Bp^{-1}gB=BwB }
\]
in $P_i/B\times G/B$. We prove the claim by showing the corresponding claim for the classes of the closed subvarieties $[\overline{O}_w]$ in $A^\bullet(P_i/B \times G/B)$.

The image $\pi_{23}(\pi_{13}^{-1}(O_w))$ is contained in $P_i/B \times P_iwB/B \subset P_i/B \times (Bs_iBwB/B \cup BwB/B)$. In case $s_iw < w$ the latter set is contained in $P_i/B \times \overline{BwB}/B$ and hence of strictly smaller dimension than $\pi_{23}^{-1}(O_w)=P_i/B \times O_w$. This proves the claim in this case.

Now assume $s_i w> w$. We have $Bs_iBwB=Bs_iwB$ and
\[
\pi_{23}(\pi_{13}^{-1}(O_w))=P_i/B \times P_iwB/B = P_i/B \times (Bs_iwB/B \cup BwB/B).
\]
This is a locally closed subset of $\overline{O}_{s_iw}$ of the same dimension, hence it is open in $\overline{O}_{s_i w}$. Thus it suffices to prove that $\pi_{23}\colon \pi_{13}^{-1}(O_w)\to \pi_{23}(\pi_{13}^{-1}(O_w))$ is an isomorphism. For this it suffices to prove that every fiber of $\pi_{23}$ above a point in this image consists of a single point. For such a point $(pB,gB)$ the fiber of $\pi_{23}$ is isomorphic to $\{qB \in P_i/B\mid Bq^{-1}gB = BwB\}$. In case $gB \in BwB/B$ the identity $Bq^{-1}gB=BwB$ implies $q\in B$ and hence the fiber consists of a single point.

Now assume $gB \in Bs_iwB/B$ and let $qB\in P_i/B$ such that $Bq^{-1}gB= BwB$. Then necessarily $q\in Bs_i B$. Then (c.f. \cite[Lemma 8.3.6]{Springer_LAG}) we can write $q=us_i$ and $g=u's_i bw$ for elements $u,u'$ in the root group $U_{\alpha_i}$ associated to $\alpha_i$ and an element $b \in B$. Then $q^{-1}g=s_i u^{-1}u' s_i bw$ with $s_i u^{-1}u' s_i \in Bs_iBs_iB=B\cup B s_i B$. Since $q^{-1}g \in BwB$ we get $s_i u^{-1}u' s_i \in B \cap U_{-\alpha_i}=\{e\}$. Thus $u=u'$ which proves that the fiber of $\pi_{23}$ above $(qB,gB)$ again consists of a single point. This finishes the proof.
\end{proof}

By induction on $\ell(w)$ we get:

\begin{corollary} \label{dwFormula}
Let $w\in W$. Then $[\overline{\Brh}_w]=\delta_w[\Brh_e]$.
\end{corollary}

%% file: Ch3CycleEO.tex
\section{The Stacks of $G$-Zips and of flagged $G$-Zips}\label{GZIP}

\subsection{The Stack of $G$-Zips}\label{GZip}

From now let $k$ be an algebraic closure of $\BF_p$ and $G$ a reductive group scheme over $\BF_p$. If $X$ is an object over some $\BF_p$-algebra, then we denote by $X^{(p)}$ the pullback of $X$ under the absolute Frobenius. For a scheme $X$ over $\BF_p$ we denote by $\varphi\colon X \to X^{(p)} = X$ its relative Frobenius.

\subsubsection*{The Zip Datum}
Let $\mu\colon \BG_{m,k} \to G_k$ be a cocharacter of $G$ defined over $k$. It gives rise to a pair of opposite parabolic subgroups $(P_-(\mu),P_+(\mu))$ and a Levi subgroup $L \defeq L(\mu) = P_-(\mu) \cap P_+(\mu)$ defined by the condition that $\Lie(P_-(\mu))$ (resp.~$\Lie(P_+(\mu))$) is the some of the non-positive (resp.~non-negative) weight space of $\mu$ in $\Lie(G)$. On $k$-valued points we have
\begin{align*}
P_+(\mu) &= \sett{g \in G}{$\lim_{t\to 0}\mu(t)g\mu(t)^{-1}$ exists}, \\
P_-(\mu) &= \sett{g \in G}{$\lim_{t\to \infty}\mu(t)g\mu(t)^{-1}$ exists}, 
\end{align*}
and $L = \Cent_G(\mu)$. We set
\[
P := P_-, \qquad Q := (P_+)^{(p)}, \qquad M := L^{(p)} = \Cent_G(\varphi \circ \mu).
\]
Hence $M$ is a Levi subgroup of $Q$.

\subsubsection*{The Stack of $G$-Zips of Type $\mu$}
Denote the projections to the Levi components $P \to L$ and $Q \to M$ both by $x \sends \xbar$. The zip group $E$ is defined as
\begin{equation}\label{EqDefZipGroup}
E \defeq \set{(x,y) \in P \times Q}{\varphi(\xbar) = \ybar}.
\end{equation}
We let $G \times G$ act on $G$ by $(x,y)\cdot g \defeq xgy^{-1}$. By restriction we obtain actions of $P \times Q$ and of $E$ on $G$. 

We denote by
\[
\GZip^{\mu} := [E \bs G]
\]
the quotient stack. It is a smooth algebraic stack of dimension $0$.

Every morphism $f\colon G \to G'$ of reductive groups over $\BF_p$ yields a morphism of stacks $\GZip^{\mu} \to \GZip[G']^{f \circ \mu}$. In particular, if $\mu' = \inn(h) \circ \mu$ for some $h \in G(k)$, then conjugation with $h$ yields an isomorphism $\GZip^{\mu} \iso \GZip^{\mu'}$. Let $\kappa$ be the field of definition of the conjugation class of $\mu$. As $G$ is quasi-split, there exists an element in that conjugacy class, that is defined over $\kappa$. Therefore it is harmless to assume that $\mu$ is defined over $\kappa$, the field of definition of its conjugacy class. We do assume this from now on. Then the stack $\GZip^{\mu}$ is defined over $\kappa$ as well.


\subsection{Choosing a Frame}\label{Frame}

\begin{lemma}\label{FindFrameLemma}
Let $\kappa$ be a finite extension of $\BF_p$. Let $G$ be a reductive group defined over $\BF_p$, let $Q \subseteq G_\kappa$ be a parabolic subgroup and let $M \subseteq Q$ be a Levi subgroup that is also defined over $\kappa$. Then there exists $g \in G(\kappa)$ and a Borel pair $T \subseteq B \subseteq G$ that is already defined over $\BF_p$ with $T \subseteq {}^gM$ and $B \subseteq {}^gQ$.
\end{lemma}

\begin{proof}
As every reductive group over a finite field is quasi-split, we can choose a maximal torus $T$ and a Borel subgroup $B \supseteq T$ defined over $\BF_p$. By \cite[Exp. XXVI, Lemme 3.8]{SGA3III} there exists a parabolic subgroup $P'$ defined over $\kappa$ with the same type as $P$ such that $B \subseteq P'$. Let $L'$ be the unique Levi subgroup of $P'$ that contains $T$. By \cite[Exp. XXVI, Cor.~5.5(iv)]{SGA3III} there exist $g \in G(\kappa)$ with ${}^gP = P'$ and ${}^gL = L'$.
\end{proof}

After replacing $\mu$ by some conjugate cocharacter $\mu'$ we may (and do) assume by Lemma~\ref{FindFrameLemma} that there exists a Borel pair $T \subseteq B \subseteq G$ defined over $\BF_p$ with $B \subseteq Q$ and $T \subseteq M$. If $\mu$ is defined over some finite extension $\kappa$ of $\BF_p$, we may assume that its conjugate is also defined over $\kappa$. Then $T$ is also a maximal torus of $M$ and hence contains its center. Hence $\varphi \circ \mu$ factors through $T$. As $T$ is defined over $\BF_p$, also $\mu$ itself factors through $T$.
As $B \subseteq Q$, the cocharacter $\varphi(\mu)$ is $B$-dominant. Hence $\mu$ is also $B$-dominant because $B$ is defined over $\BF_p$.

Recall that we denote by $(W,\Sigma)$ the Coxeter system associated to $(G,B,T)$. The Frobenius $\varphi$ on $G$ induces an automorphism of the Coxeter system $(W,\Sigma)$ which is again denoted by $\varphi$ (see also Subsection~\ref{FrobAct} below). Let $I,J \subseteq \Sigma$ be the set of simple reflections corresponding to the conjugacy classes of $P$ and $Q$, respectively. 

By \cite[3.7]{PWZ1} (and its proof) we find $z \in G(k)$ with ${}^zT = T$ such that $(B,T,z)$ is a frame for $(G,P,L,Q,M,\varphi)$ in the sense of \cite[3.6]{PWZ1}, i.e., ${}^zB \subseteq P$ and $\varphi({}^zB \cap L) = B \cap M$. In fact we can and will choose $z$ as follows.

\begin{lemma}\label{DescribeFrame}
Let $z \in \Norm_G(T)(k)$ be a lift of $\zbar := w_{0,I}w_0 \in W$. Then ${}^zB \subseteq P$ and $\varphi({}^zB \cap L) = B \cap M$.
\end{lemma}

\begin{proof}
For any smooth connected algebraic subgroup $H$ of $G$ (defined over $k$) that contains $T$ we have $\Lie(H) = \Lie(T) \oplus \bigoplus_{\alpha \in \Phi_H}\Lie(G)_{\alpha}$ for a subset $\Phi_H$ of the set of roots of $(G,T)$. Moreover $H = H'$ if and only if $\Phi_H = \Phi_{H'}$ for two such groups $H$ and $H'$. Let $\Phi^+ = \Phi_B \subseteq \Phi$ be the set of $B$-positive roots. We have $\Phi_P = \set{\alpha \in \Phi}{\langle \mu,\alpha\rangle \leq 0}$, $\Phi_L = \set{\alpha \in \Phi}{\langle \mu,\alpha\rangle = 0}$, and $\Phi_M = \set{\alpha \in \Phi}{\langle \varphi(\mu),\alpha\rangle = 0}$. Hence it suffices to show the following assertions for all $\alpha \in \Phi^+$:
\begin{assertionlist}
\item
$\langle \mu, \zbar(\alpha)\rangle \leq 0$,
\item
If $\langle \mu, \zbar(\alpha) \rangle = 0$, then $\varphi(\zbar(\alpha)) \in \Phi^+$ and $\langle \varphi(\mu), \varphi(\zbar(\alpha)) \rangle = 0$.
\end{assertionlist}
We show (1). As $L$ centralizes $\mu$, we find $w_{0,I}(\mu) = \mu$ and hence $\langle \mu, \zbar(\alpha)\rangle = \langle \mu, w_0(\alpha) \rangle \leq 0$ because $\mu$ is $B$-dominant.
For (2) we first note that the second equality is clear because $\langle \varphi(\mu), \varphi(\zbar(\alpha)) \rangle = \langle \mu, \zbar(\alpha) \rangle$. As $\varphi$ preserves the set of positive roots, it remains to show that if
\[
\langle \mu, \zbar(\alpha) \rangle = \langle w_0(\mu), \alpha \rangle = 0,\tag{*}
\]
then $w_0w_{0,I}w_0(\alpha) \in w_0(\Phi^+) = \Phi^-$. Let $L^0 := \Cent_G(w_0(\mu))$. Then (*) implies that $\alpha \in \Phi_{L^0} \cap \Phi^+$ and hence $w_0w_{0,I}w_0(\alpha) \in \Phi_{L^0} \cap \Phi^- \subseteq \Phi^-$.
\end{proof}

By \cite[3.11]{PWZ1} the map
\begin{equation}\label{EqInducedIsomWeyl}
\varphi \circ \inn(z)\colon (W_I,I) \liso (W_J,J)
\end{equation}
is an isomorphism of Coxeter systems and $I$, $J$, and $\varphi \circ \inn(z)$ are independent of the choice of the frame.


\subsection{Classification of $G$-Zips}\label{ClassifyGZip}

For $w \in W$ let $G_w \subseteq G$ be the $E$-orbit of $\dot{w}z$. By \cite[7.5]{PWZ1} there is a bijection
\begin{equation}\label{EqZipOrbits}
{}^IW \bijective \{\text{$E$-orbits in $G$}\}, \qquad w \sends G_w.
\end{equation}
and $\dim(G_w) = \ell(w) + \dim(P)$ for $w \in {}^IW$. We call the corresponding locally closed algebraic substack of $\GZip^{\mu}$
\begin{equation}\label{EqDefZw}
Z_w := [E \bs G_w] \subseteq \GZip^{\mu}
\end{equation}
the \emph{zip stratum corresponding to $w \in {}^IW$}. One has $\codim_{\GZip^{\mu}}(Z_w) = \#\Phi^+ - \ell(w)$, where $\Phi^+$ is the set of positive roots of $(G,T,B)$.

Let $\overline{G_w}$ be the closure of the $E$-orbit $G_w$. We set $\overline{Z}_w := [E \bs \overline{G_w}]$. This is the unique reduced closed algebraic substack of $\GZip^{\mu}$ whose underlying topological space is the closure of the one-point topological space underlying $Z_w$. By \cite[6.2]{PWZ1} we have
\begin{equation}\label{EqClosureGZipStratum}
\overline{Z}_w = \bigcup_{w' \preceq w}Z_{w'}
\end{equation}
for a partial order $\preceq$ on ${}^IW$ defined in loc.~cit.~6.1. Here we will need only the following properties of this partial order (see \cite[\S 3]{He_GStable}).

\begin{lemma}\label{LemPartialOrder}
\begin{assertionlist}
\item\label{LemPartialOrder1}
There exists a unique minimal element in ${}^IW$, namely the neutral element $e$, and a unique maximal element in ${}^IW$, namely $w_{0,I}w_0$, where $w_0$ and $w_{0,I}$ are unique elements of maximal length in $W$ and in $W_I$, respectively.
\item\label{LemPartialOrder2}
The partial order $\preceq$ is finer than the Bruhat order.
\item\label{LemPartialOrder3}
Let $w' \preceq w$. Then $\ell(w') \leq \ell(w)$ and one has $\ell(w') = \ell(w)$ if and only if $w' = w$.
\item\label{LemPartialOrder4}
If $w' \prec w$ and there exists no $u \in {}^IW$ with $w' \prec u \prec w$, then $\ell(w') = \ell(w)-1$.
\end{assertionlist}
\end{lemma}



\subsection{The Action of Frobenius}\label{FrobAct}

Recall that we denote by $\varphi\colon G \to G$ the relative Frobenius. We also denote by $\sigma\colon k \to k$, $x \sends x^p$ the arithmetic Frobenius. As $T$ and $B$ are defined over $\BF_p$, we can identify canonically $T^{(p)}$ with $T$ and $B^{(p)}$ with $B$. Hence the relative Frobenius induces isogenies $\varphi\colon T \to T$ and $\varphi\colon B \to B$.

Set $\CW := \Norm_G(T)/T = \pi_0(\Norm_G(T))$ which is a finite \'etale group scheme over $\BF_p$. Then $W = \CW(k)$ is the absolute Weyl group. As $\Norm_G(T)$ is also defined over $\BF_p$, the relative Frobenius $\varphi$ induces an automorphisms of $\CW$ and hence an automorphism $\varphi$ of the finite group $W$. As $B$ is defined over $\BF_p$, this automorphism preserves the set $\Sigma$ of simple reflections in $W$ defined by $B$. By functoriality, $\sigma$ also defines an automorphism $w \sends {}^{\sigma}w$ of $W = \CW(k)$ and we have $\varphi(w) = {}^{\sigma^{-1}}w$ for all $w \in W$. If $T$ is a split torus, then $\varphi = \id$ on $W$.

We denote by $X^*(T)$ the group of characters of $T \otimes_{\BF_p} k$. For $\lambda \in X^*(T)$ we set $\varphi(\lambda) \defeq \lambda \circ \varphi$ which defines an endomorphism $\varphi$ on the abelian group $X^*(T)$. We denote by $\lambda \sends {}^{\sigma}\lambda$ the canonical action of $\sigma$ on $X^*(T)$, i.e., 
\[
{}^{\sigma}\!\lambda \defeq (\id_{\BG_{m,\BF_p}} \otimes \sigma) \circ \lambda \circ (\id_T \otimes \sigma^{-1}).
\]
Then one has for $\lambda \in X^*(T)$
\begin{equation}\label{EqFrobChar}
\varphi(\lambda) = p \,{}^{\sigma^{-1}}\!\lambda.
\end{equation}
If $T$ is a split torus, then ${}^{\sigma}\!\lambda = \lambda$ and $\varphi(\lambda) = p\lambda$ for all $\lambda \in X^*(T)$.

By functoriality, the action of $\varphi$ and $\sigma$ on $X^*(T)$ also induce actions on the graded $\BQ$-algebra $S = \Sym(X^*(T))_{\BQ}$ and for $f \in S$ of degree $d$ we have
\begin{equation}\label{EqFrobSym}
\varphi(f) = p^d \,{}^{\sigma^{-1}}\!f.
\end{equation}


\subsection{The Stack of Flagged $G$-zips of Type $\mu$}

We fix $I_0 \subseteq I$ a subset and let $P_0$ be the unique parabolic subgroup of $G$ of type $I_0$ with ${}^zB \subseteq P_0 \subseteq P$. We let $E$ act on $G \times P/P_0$ by
\[
(x,y)\cdot (g,aP_0) \defeq (xgy^{-1}, xaP_0)
\]
and set
\[
\GZipFlag^{\mu,I_0} \defeq [E\bs (G \times P/P_0)].
\]
If $I_0 = \emptyset$, then $P_0 = {}^zB$ and we abbreviate $\GZipFlag^{\mu} \defeq \GZipFlag^{\mu,\emptyset}$. Note that $\GZipFlag^{\mu,I} = \GZip^{\mu}$. For $I'_0 \subseteq I_0$ there are canonical projection maps
\[
\GZipFlag^{\mu,I'_0} \to \GZipFlag^{\mu,I_0}
\]
that are $P_0/P'_0$-bundles where $P'_0$ is the unique parabolic subgroup of type $I'_0$ with ${}^zB \subseteq P'_0 \subseteq P$. In particular, these maps are proper, smooth, and representable.

Let $L_0 \subset P_0$ be the unique Levi subgroup containing $T$. We set
\[
M_0 \defeq L_0^{(p)}, \qquad Q_0 \defeq M_0B.
\]
Then $Q_0$ is a parabolic subgroup containing $B$ of type
\[
J_0 := \varphi({}^zI_0)
\]
and $M_0$ is the unique Levi subgroup of $Q_0$ containing $T$. Then $(B,T,z)$ is again a frame for $(G,P_0,L_0,Q_0,M_0,\varphi)$. By \cite[2.2]{GoKo_ZipFunctoriality}, $G \times P \to G$, $(g,x) \sends \xbar g \varphi(\xbar)^{-1}$ induces a smooth representable morphism of algebraic stacks
\[
\psi^{I_0}\colon \GZipFlag^{\mu,I_0} \to \Brh^{I_0} := [P_0\bs G/Q_0]
\]
with irreducible fibers. The maps $\psi^{I_0}$ are compatible with passing to $I'_0 \subseteq I_0$. 

For $I_0 = \emptyset$ we have $P_0 = {}^zB$ and $Q_0 = B$. Therefore $g \sends z^{-1}g$ yields an isomorphism $\Brh^{\emptyset} \iso \Brh_G$ and we denote by $\psi$ the composition
\[
\psi\colon \GZipFlag^{\mu} \ltoover{\psi^{\emptyset}} \Brh^{\emptyset} \liso \Brh_G,
\]
which is a smooth representable morphism with irreducible fibers.
For $w \in W$ we denote by
\begin{equation}\label{EqZipFlagStrata}
Z_w^{\emptyset} := \psi^{-1}(\Brh_w) \subseteq \GZipFlag^{\mu}.
\end{equation}
Since $\psi$ is smooth the $Z^\emptyset_w$ form a stratification of $\GZipFlag^\mu$ whose closure relation is given by the Bruhat order on $W$:
\begin{equation*}
  \overline{Z^\emptyset_w}=\bigcup_{w' \leq w}Z^\emptyset_{w'}
\end{equation*}

\begin{proposition} \label{ZwProps}
The strata $Z_w^\emptyset$ are smooth and irreducible. Their closures $\overline{Z^\emptyset_w}$ are normal and with only rational singularities. In particular, they are Cohen-Macaulay.
\end{proposition}

\begin{proof}
As $\psi$ is smooth with irreducible fibers and $\Brh_w$ is smooth and irreducible, the first assertion holds. The smoothness of $\psi$ also implies that $\overline{Z^\emptyset_w} = \psi^{-1}(\overline{\Brh_w})$. Hence all remaining assertions follow from the analogous properties for Schubert varieties \cite[3.2.2, 3.4.3]{BrionKumar_FrobSplit}.
\end{proof}

By \cite[2.2.1]{Koskivirta_NormalEO} we have the following:

\begin{proposition} \label{ZwProj}
The projection $\pi\colon \GZipFlag^{\mu} \to \GZip^{\mu}$ induces for $w \in {}^IW$ representable finite \'etale maps
\[
\pi_w\colon Z_w^{\emptyset} \to Z_w.
\]
\end{proposition}

\begin{definition}\label{DefGammaw}
We set $\gamma(w) \defeq \deg(\pi_w)$.
\end{definition}

In the next section we give a description of $\gamma(w)$.

\begin{remark} \label{GZipFlagModuli}
  Like their name suggests, the spaces $\GZipFlag^{\mu,I_0}$ admit a moduli description as a ``flag space'' over $\GZip^\mu$. Specifically, the stack $\GZipFlag^\mu$ is canonically isomorphic to the moduli stack of pairs consisting of a $G$-zip $(I,I_+,I_-,\iota)$ of type $\mu$ as in \cite[Definition 3.1]{PWZ2} together with a $P_0$-subtorsor of the $P$-torsor $I_+$. See \cite[Section 2.1]{GoKo_ZipFunctoriality} for details on this construction.
\end{remark}

\subsection{Calculation of $\gamma(w)$} \label{Gammaw}

Fix $w \in {}^IW$. 

\subsubsection*{The Type of $w \in {}^IW$}

We recall the following construction from \cite[\S5]{PWZ1}. Fix $w \in {}^IW$. Let $I_w$ be the largest subset of $I$ such that
\[
\varphi({}^zI_w) = {}^{w^{-1}}I_w
\]
and call it the \emph{type of $w$}. In other words
\begin{equation}\label{EqIw}
I_w = \set{s \in I}{\forall\,k\geq 1:\; (\inn(w) \circ \varphi \circ \inn(z))^k \in I\ }.
\end{equation}
For instance, as $\varphi({}^zI) = J$, one has $I_e = I$ if and only if $I = J$. Let $P_w$ be the unique parabolic subgroup of type $I_w$ with ${}^zB \subseteq P_w$ and let $L_w$ be the unique Levi subgroup of $P_w$ with $L_w \supseteq T$. As for an arbitrary subset of $I$ we obtain 
\[
M_w := {}^{(zw)^{-1}}L_w = L_w^{(p)}, \qquad Q_w := M_wB.
\]
Hence $Q_w$ is the unique parabolic subgroup containing $B$ of type $J_w$, where
\[
J_w \defeq {}^{w^{-1}}I_w = \varphi({}^zI_w)
\]
and $M_w$ is the unique Levi subgroup of $Q_w$ containing $T$. Note that $M_w$ (resp.~$J_w$) is denoted by $H_w$ (resp.~$K_w$) in \cite[\S5]{PWZ1}.

\subsubsection*{Description of $\gamma(w)$ via Flag Varieties}

Set
\[
A_w := \set{x \in L_w}{{}^{zw}\varphi(x) = x}
\]
Then we have by \cite[2.2.1]{Koskivirta_NormalEO} and \cite[8.1]{PWZ1}
\begin{equation}\label{EqShapeGammaw}
\gamma(w) = \#(A_w/(A_w \cap {}^zB)).
\end{equation}

\begin{lemma}\label{StableBorel}
\begin{assertionlist}
\item\label{StableBorel1}
$(\inn{zw} \circ \varphi)(L_w) = L_w$.
\item\label{StableBorel2}
$(\inn{zw} \circ \varphi)(L_w \cap {}^zB) = L_w  \cap {}^zB$.
\end{assertionlist}
\end{lemma}

\begin{proof}
The first assertion follows from ${}^{zw}\varphi({}^zI_w) = {}^zI_w$. Let us show the second assertion. Both sides are Borel subgroups of $L_w$ which contain $T$. Hence it suffices to show that they contain the same root subgroups. Let $\Phi$ be the set of roots for $(G,T)$ and let $\Phi^+$ be the set of positive roots with respect to $B$. For a set of simple reflection $K$ let $\Phi_K$ be the set of roots of the standard Levi subgroup $L_K$ of type $K$. Then $\Phi^+_K := \Phi_K \cap \Phi^+$ is the system of positive roots given by the Borel subgroup $L_K \cap B$ of $L_K$.

As $z$ normalizes $T$ we can consider its image in $W$ which we denote again by $z$. Then the set of roots corresponding to $L_w$ is ${}^z\Phi_{I_w}$ and the set of roots corresponding to $L_w  \cap {}^zB$ is ${}^z\Phi^+_{I_w}$. Hence we have to show
\[
{}^w\varphi({}^z\Phi^+_{I_w}) = \Phi^+_{I_w}
\]
As both sides have the same cardinality, it suffices to show that the left side is contained in the right side.
By definition of a frame we have $\varphi({}^zB \cap L_w) \subseteq B \cap M_w$ and this shows
\[
{}^w\varphi({}^z\Phi^+_{I_w}) \subseteq {}^w\Phi^+_{J_w} = \Phi^+_{I_w}
\]
because ${}^wJ_w = I_w$.
\end{proof}

Hence $\inn{zw} \circ \varphi$ defines a descent datum from $k$ to $\BF_p$ for the reductive group $L_w$ together with its Borel subgroup ${}^zB \cap L_w$. We obtain a reductive group $L'_w$ and a Borel subgroup $B'_w$ defined over $\BF_p$ and its full flag variety by $F\ell_w := L'_w/B'_w$. Then we have by \eqref{EqShapeGammaw} the following description of $\gamma(w)$.

\begin{proposition}\label{CalcGammw}
For $w \in {}^IW$ one has
\begin{equation}\label{EqGammaw}
\gamma(w) = L'_w(\BF_p)/B'_w(\BF_p) = F\ell_w(\BF_p)
\end{equation}
\end{proposition}

Here the second identity follows from $H^1(\BF_p,B'_w) = 0$.

\begin{remark}\label{TipGammaw}
By definition, $L'_w$ is a form defined over $\BF_p$ of the standard Levi subgroup of $G$ corresponding to the set of simple reflections $I_w$. It is split if and only if ${}^w\varphi({}^zs) = s$ for all $s \in I_w$. If the Dynkin diagram of $L_w$ has no automorphisms (e.g., if it is connected of type $B_n$, $C_n$, $E_7$, $E_8$, $F_4$, or $G_2$), then this is automatic.

If $L'_w$ is split, one obtains from the decomposition of the flag variety $F\ell_w$ into a disjoint union of Schubert cells the following formula
\begin{equation}\label{EqGammawSplit}
\gamma(w) = \sum_{w\in W_{I_w}}p^{\ell(w)}.
\end{equation}
\end{remark}



\subsection{The Key Diagram}

The projection $E \to P$, $(x,y) \sends x$ is a surjective homomorphism of algebraic groups. We obtain a composition
\begin{equation}\label{EqDefBeta}
\beta\colon \GZip^{\mu} = [E\bs G] \lto \B{E} \lto \B{P}.
\end{equation}
Finally, we have a morphism $\gamma\colon \Brh_G \to \B{P}$ defined as the composition
\begin{equation}\label{EqDefGamma}
\gamma\colon \Brh_G = \B{B} \times_{\B{G}} \B{B} \ltoover{{\rm pr}_1} \B{B} \liso \B{{}^zB} \lto \B{P},
\end{equation}
where the second map is induced by the isomorphism $b \sends zbz^{-1}$ and where the third map is induced by the inclusion ${}^zB \to P$.

The following commutative diagram will be our key diagram
\begin{equation}\label{EqKey}
\begin{aligned}\xymatrix{
\GZipFlag^{\mu} \ar[rr]^-{\psi} \ar[rrd]^-{\alpha} \ar[d]_{\pi} & & \Brh_G \ar[d]^{\gamma} \\
\GZip^{\mu} \ar[rr]^{\beta} & & \B{P}
}\end{aligned}
\end{equation}
where $\alpha := \beta \circ \pi$. All morphisms are flat of constant relative dimension. Moreover, $\pi$ is a $P/{}^zB$-bundle. Note that $P/{}^zB = L/({}^zB \cap L)$ is the full flag variety for $L$. In particular, $\pi$ is proper, smooth, and representable.

%% file: Ch4CycleEO.tex
\section{Induced Maps of Chow Rings}\label{CHOW}

In this section we describe the maps induced by the key diagram \eqref{EqKey} on Chow rings. If $\Xscr$ is any smooth algebraic quotient stack defined over some subfield $k_0$ of $k$ we set $A^{\bullet}(\Xscr) := A^{\bullet}(\Xscr \otimes_{k_0} k)$.


\subsection{The Chow Ring of $\GZip^{\mu}$ and $\GZipFlag^{\mu}$}\label{ChowGZipFlag}
We recall Brokemper's description of the Chow ring of $A^\bullet(\GZip^\mu)$ from \cite{Brokemper_ChowZip}: 

Recall that $S := \Sym(X^*(T)_{\BQ}) = A^{\bullet}(\B{T})$. This is a graded $\BQ$-algebra carrying an action by the Weyl group $W$ by graded automorphisms. We also denote by $S_+ \defeq S_{\geq 1}$ the augmentation ideal of $S$. Let
\begin{equation}\label{DefCI}
\CI \defeq ( f - \varphi(f) \mid f \in S^W_+ ) \subseteq S^W
\end{equation}
be the ideal generated by $f - \varphi(f)$ for $f \in S^W_+$ in $S^W$. As we work with rational coefficients, there is also a simpler description of $\CI$ (see Remark~\ref{IntegralChowZip} below why the definition \eqref{DefCI} is more natural in this context).

\begin{lemma}\label{RationalCI}
One has $\CI = S^W_+$.
\end{lemma}

\begin{proof}
We have to show that $S^W_+ \subseteq \CI$. Let $f \in S^W$ be of degree $d \geq 1$ and let $s \geq 1$ be an integer such that $T \otimes_{\BF_p} \BF_{p^s}$ is split. Thus $\sigma^s$ acts trivially on $X^*(T)$ and hence on $S$. Then $\varphi^s(f) = p^{ds}f$ by \eqref{EqFrobSym} and therefore
\[
(1-p^{ds})f = f - \varphi^s(f) = \sum_{i=1}^s(\varphi^{i-1}(f)- \varphi^i(f)) \in \CI.\qedhere
\]
\end{proof}

For every set $K\subseteq \Sigma$ of simple reflections $S^{W_K}$ is a finite free $S^W$-algebra of rank $\#(W/W_K)$, hence the canonical map
\[
S^W/\CI \to S^{W_K}/\CI S^{W_K}
\]
is finite and faithfully flat and in particular injective.

We keep the notation from Subsection~\ref{GZip}. For every type $K \subseteq \Sigma$ of a parabolic subgroups we denote by $K^{\rm o}$ the opposed type. Then
\begin{equation}\label{EqIopp}
I^{\rm o} = {}^zI = \varphi^{-1}(J)
\end{equation}
is a set of simple reflections and $LB = {}^zP_I$ is the standard parabolic subgroup of type ${}^zI$.

For a subgroup $H$ of $G$, we denote by $[H {}_\varphi\!\!\bs G]$ the quotient stack for the action of $H$ on $G$ by $\phi$-conjugation $(h,g) \mapsto hg\varphi^{-1}(h)$. The following description of the Chow ring of these stacks for $H=T$ and $H=L$ is given by \cite[2.3.2]{Brokemper_ChowZip} and its proof.

\begin{proposition} \label{ChowPhiQ}
\begin{assertionlist}
\item
Consider the homomorphism
    \begin{equation} \label{ChowPhiQ0}
      S \otimes_{S^W} S \cong A^\bullet([B\bs G/B]) \cong A^\bullet([T\bs G/T]) \to A^\bullet(\PhiQ{T}{G} )
    \end{equation}
induced by pullback along the quotient morphism $\PhiQ{T}{G} \to [T\bs G/T]$ and the homomorphism
\[
S \to S\otimes_{S^W} S,\; f \mapsto f\otimes 1.
\]
The composition $S \to A^\bullet(\PhiQ{T}{G})$ of these homomorphisms factors through an isomorphism of graded $\BQ$-algebras
\begin{equation} \label{ChowPhiQ1}
  S/\CI S \cong A^\bullet(\PhiQ{T}{G}).
\end{equation}
\item
The homomorphism $S \otimes_{S^W} S \to S/\CI S$ given by \eqref{ChowPhiQ0} and \eqref{ChowPhiQ1} sends $f \otimes g$ to the class of $f\phi(g)$.
\item
The homomorphism 
  \begin{equation*}
    S^{W_{I^{\rm o}}}/\CI S^{W_{I^{\rm o}}} \to S /\CI S
  \end{equation*}
induced by the inclusion $S^{W_{I^{\rm o}}} \into S$ is injective and free of rank $|W_{I^{\rm o}}|$. 
\item
The homomorphism $A^\bullet(\PhiQ{L}{G}) \to A^\bullet(\PhiQ{T}{G})$ induced by the quotient morphism $\PhiQ{T}{G} \to \PhiQ{L}{G}$ is injective. Under \eqref{ChowPhiQ2} it gives an isomorphism of graded $\BQ$-algebras
  \begin{equation}    \label{ChowPhiQ2}
    A^\bullet(\PhiQ{L}{G}) \cong S^{W_{I^{\rm o}}}/\CI S^{W_{I^{\rm o}}}.
  \end{equation}
\end{assertionlist}
\end{proposition}

\begin{proposition}[{c.f. \cite[2.4.4]{Brokemper_ChowZip}}]\label{ChowGZip}
\begin{assertionlist}
\item\label{ChowGZip1}
The homomorphism $E \to L, \; x \mapsto \bar x$ induces a morphism
$$
\GZip^\mu=[E\bs G] \to \PhiQ{L}{G}.
$$
Using \eqref{ChowPhiQ2}, on Chow rings this morphism induces an isomorphism
    \begin{equation} \label{EqChowGZip1}
        S^{W_{I^{\rm o}}}/\CI S^{W_{I^{\rm o}}} \cong A^\bullet(\PhiQ{L}{G}) \cong A^{\bullet}(\GZip^{\mu})
    \end{equation}
of graded $\BQ$-algebras.
\item\label{ChowGZip2}
For the group scheme $E' = E \cap ({}^zB \times G)$ we have a natural identification
\begin{equation}\label{EqAltDescGZipFlag}
\GZipFlag^{\mu} = [E \bs (G \times P/{}^zB)] = [E' \bs G].
\end{equation}
Under this identification, the homomorphism $E' \to T,\; (x,y) \mapsto \bar x$ induces a morphism
\begin{equation*}
  \GZipFlag^\mu \to \PhiQ{T}{G}.
\end{equation*}
Using \eqref{ChowPhiQ1}, on Chow rings this morphism induces an isomorphism
\begin{equation} \label{EqChowGZip2}
S/\CI S \cong A^\bullet(\PhiQ{T}{G}) \cong A^{\bullet}(\GZipFlag^{\mu})
\end{equation}
of graded $\BQ$-algebras.
\item\label{ChowGZip3}
Under the isomorphisms \eqref{EqChowGZip1} and \eqref{EqChowGZip2}, the homomorphism $\pi^*\colon S^{W_{I^{\rm o}}}/\CI S^{W_{I^{\rm o}}} \to S/\CI S$ induced on Chow rings by the projection $\pi\colon \GZipFlag^\mu \to \GZip^\mu$ is the one induced by the inclusion $S^{W_{I^{\rm o}}} \into S$.
\end{assertionlist}
\end{proposition}

\begin{proof}
The kernels of the surjective homomorphisms $E \to L$ and $E' \to T$ are unipotent. Hence \ref{ChowGZip1} and \ref{ChowGZip2} follow from Proposition~\ref{IgnoreUnipotent}. Then \ref{ChowGZip3} follows from the compatibility of the various constructions.
\end{proof}

The above results allow to give a non-canonical identification of $A^{\bullet}(\GZip^{\mu})$ with the rational cohomology ring of a certain flag variety. Let $\Gbf_{\BC}$ be the reductive group over $\BC$ with the same based root datum as $G_k$, let $\Pbf$ be a parabolic subgroup of type $I$ of $\Gbf_{\BC}$, and set $\Xbf\vdual \defeq \Gbf_{\BC}/\Pbf_I$. If $(G,\mu)$ is induced by a Shimura datum $(\Gbf,\Xbf)$ (see Section~\ref{TAUT} below), then $\Xbf\vdual$ is the compact dual of $\Xbf$. This explains the notation. Denote
\[
H^{2\bullet}(\Xbf\vdual) \defeq \bigoplus_{i=0}^dH^{2i}(\Xbf\vdual(\BC),\BQ)
\]
the cohomology ring of the complex manifold $\Xbf\vdual(\BC)$ with rational coefficients. The multiplication is given by cup product. As the cohomology is concentrated in even degree, this is a commutative graded $\BQ$-algebra.

\begin{corollary}\label{ChowGZipCohomology}
There is an isomorphism of graded $\BQ$-algebras $A^{\bullet}(\GZip^{\mu}) \cong H^{2\bullet}(\Xbf\vdual)$.
\end{corollary}

\begin{proof}
We use the description of $\CI$ as in Lemma~\ref{RationalCI}. Then we have isomorphisms of graded $\BQ$-algebras
\[
A^{\bullet}(\GZip^{\mu}) \cong S^{W_{I^o}}/\CI S^{W_{I^o}} \cong S^{W_{I}}/\CI S^{W_{I}} \cong H^{2\bullet}(\Xbf\vdual),
\]
where the first isomorphism is given by \eqref{EqChowGZip1}, the second isomorphism is given by conjugation with the longest element $w_0$ in the Weyl group, and the third isomorphism holds by \cite[Theorem 26.1]{Borel_CohEspacesFibres}, identifying $\Xbf\vdual$ with a quotient of the real compact form of $\Gbf_{\BC}$.
\end{proof}

\begin{remark}\label{IntegralChowZip}
Recall that we work with $\BQ$-coefficients, hence we may use the description of $\CI$ in Lemma~\ref{RationalCI}. But from the results of Brokemper it follows that the results of Proposition~\ref{ChowGZip} even hold with $\BZ$-coefficients if one uses the description \eqref{DefCI} of $\CI$ and if the group $G$ is special, i.e., every \'etale $G$-torsor is already Zariski-locally trivial. Examples for special groups are $\GL_n$, $\SL_n$, or $\Sp_{2n}$. A non-example would be $\PGL_n$ for $n \geq 2$.
\end{remark}

Moreover, by \cite[2.4.10, 2.4.11]{Brokemper_ChowZip} we also have the following result.

\begin{proposition}\label{ChowGZipNaive}
$A^{\bullet}(\GZip^{\mu})$ is a finite $\BQ$-algebra of dimension $\#{}^IW$. A $\BQ$-basis is given by the classes $[\overline{Z}_w]$ of the closures of the $E$-orbits on $G$.
\end{proposition}

This result holds even for special groups only with rational coefficients.


\subsection{Pullback Maps for the Key Diagram}\label{PullBackOfKey}

We now apply $A^{\bullet}(-)$ as a contravariant functor to the key diagram. Recall that $\CI = S^W_+$ is the augmentation ideal of $S^W$. We have $A^{\bullet}(\B{P}) = A^{\bullet}(\B{L}) = S^{W_{I^{\rm o}}}$ by Proposition~\ref{IgnoreUnipotent} and \eqref{EqDescribeChowReductive} and
\[
A^{\bullet}(\Brh_G) = S \otimes_{S^W} S = (S \otimes_{\BQ} S)/(1 \otimes f - f \otimes 1 \mid f \in \CI)
\]
by Proposition~\ref{ChowBruhat}. Using this, \eqref{EqChowGZip1} and \eqref{EqChowGZip2} we obtain the following commutative diagram of graded $\BQ$-algebras by applying $A^{\bullet}(-)$ to \eqref{EqKey}:
\begin{equation}\label{EqKeyChow}
\begin{aligned}\xymatrix{
S/\CI S  & & S \otimes_{S^W} S \ar[ll]_{\psi^*} \\
S^{W_{I^{\rm o}}}/\CI S^{W_{I^{\rm o}}} \ar@{^(->}[u]_{\pi^*} & & S^{W_{I^{\rm o}}} \ar[ll]^{\beta^*} \ar[llu]_{\alpha^*} \ar[u]_{\gamma^*}
}\end{aligned}
\end{equation}

\begin{theorem}\label{PullbackKey}
The morphisms in \eqref{EqKeyChow} are as follows:
\begin{assertionlist}
\item
The homomorphisms $\pi^*$ and $\alpha^*$ are induced from the inclusion $S^{W_{I^{\rm o}}} \into S$.
\item
The homomorphism $\beta^*$ is the canonical projection.
\item
The homomorphism $\gamma^*$ is the composition (using \eqref{EqIopp})
\[
\gamma^*\colon S^{W_{I^{\rm o}}} = z(S^{W_I}) \ltoover{z^{-1}} S^{W_I} \vartoover{30}{f \sends f \otimes 1} S \otimes_{S^W} S.
\]
\item
The homomorphism $\psi^*$ is induced by
\[
f \otimes g \lsends z(f)\varphi(g).
\]
\end{assertionlist}
\end{theorem}

\begin{proof}
The description of $\pi^*$ is given by Proposition \ref{ChowGZip}. Since $\pi^*$ is injective by Proposition \ref{ChowPhiQ}, the descriptions of $\alpha^*$ and $\beta^*$ will follow from the description of $\psi^*$ and $\gamma^*$ since \eqref{EqKeyChow} commutes. The description of $\gamma^*$ follows from the definition of $\gamma$ and the construction of the isomorphism $A^\bullet(\Brh_G) \cong S \otimes_{S^W} S$.

To verify the description of $\psi^*$ we consider the following commutative diagram:
\begin{equation*}
  \xymatrix{
    \GZipFlag^\mu=[E' \bs G] \ar[r]^\psi \ar[d] & \Brh_G=[(B \times B)\bs G] \ar[d] \\
    \PhiQ{T}{G} \ar[r] & [(T\times T) \bs G]
  }
\end{equation*}
The morphisms in this diagram are given as follows: The morphism $\psi$ is induced from $G \to G,\; g \mapsto z^{-1} g$ and $E' \to B \times B,\; (x,y) \mapsto (\leftexp{z^{-1}}{x},x)$. Similarly the bottom horizontal morphism is induced from $G \to G,\; g \mapsto z^{-1} g$ and $T \to T\times T,\; t\mapsto (\leftexp{z^{-1}}{t},\phi(t))$. The left vertical morphism is the one from Proposition \ref{ChowGZip} and the right vertical one is induced from the identity on $G$ and the projection $B \times B \to (B \times B)/ \rad^u(B \times B) \cong T \times T$. 

The two vertical morphisms induce isomorphisms on Chow rings. Using Proposition \ref{ChowPhiQ} one checks that the bottom horizontal morphism induces the morphism $S \otimes_{S^W} S \to S/IS$ which sends $f \otimes g$ to the class of $z(f)\phi(g)$. This shows what we want.

\end{proof}


\subsection{Description of $\pi_*$}\label{DescribePi}
The morphism $\pi\colon \GZipFlag^\mu \to \GZip^\mu$, being a $P/\leftexp{z}{B}$-bundle, is proper. Hence under \eqref{EqChowGZip1} and \eqref{EqChowGZip2} it induces a push-forward morphism
\begin{equation*}
  \pi_*\colon A^\bullet(\GZipFlag^\mu)\cong S/\CI S \to A^\bullet(\GZip^\mu) \cong S^{W_{I^{\rm o}}} / \CI S^{W_{I^{\rm o}}}.
\end{equation*}

As an application of a general push-forward formula of Brion from \cite{BrionGysin} we get the following description of $\pi_*$:
\begin{theorem} \label{GysinFormula}
  The pushforward $\pi_*\colon S/\CI S \to S^{W_{I^{\rm o}}} / \CI S^{W_{I^{\rm o}}}$ sends the class of $f \in S$ to the class of
  \begin{equation*}
    \frac{\sum_{w \in W_{I^{\rm o}}}  (-1)^{\ell(w)} w(f)}{\prod_{\alpha \in \Phi^+_{I^{\rm o}}}\alpha} \in S^{W_{I^{\rm o}}}.
  \end{equation*}
\end{theorem}
\begin{proof}
  Consider the following cartesian diagram:
  \begin{equation*}
    \xymatrix{
      \GZipFlag^\mu \ar[r]^(0.3){\psi^\emptyset} \ar[d]_\pi & \Brh^\emptyset \cong [\leftexp{z}{B} \bs G / B] \cong [*/\leftexp{z}{B}] \times_{[*/G]} [*/B] \ar[r]^(0.8){\operatorname{pr}_1} & [*/\leftexp{z}{B}]  \ar[d]_{\tilde\pi} \\
    \GZip^\mu \ar[rr] & &[*/P]
    }
  \end{equation*}
On Chow rings this induces the following diagram:
\begin{equation*}
  \xymatrix{
    S/ \CI S & \ar[l] S \ar[l] \\
    S^{W_{I^{\rm o}}}/ \CI S^{W_{I^{\rm o}}} \ar[u] & S^{W_{I^{\rm o}}} \ar[u] \ar[l] \\
}
\end{equation*}
Hence it suffices to prove the corresponding formula for $\tilde\pi_*\colon A^\bullet([*/\leftexp{z}{B}]) \cong S \to A^\bullet([*/P]) \cong S^{W_{I^{\rm o}}}$. Similarly, using the cartesian diagram
\begin{equation*}
  \xymatrix{
    [*/\leftexp{z}{B} \cap L] \ar[r] \ar[d]_{\tilde{\tilde\pi}} & [*/\leftexp{z}{B}] \ar[d]_{\tilde\pi}  \\
    [*/L] \ar[r] & [*/P],
  }
\end{equation*}
whose horizontal morphisms induce isomorphisms on Chow groups, one reduces to proving the corresponding formula for $\tilde{\tilde \pi}_*$. This formula is given by \cite[Prop. 1.1]{BrionGysin}.
\end{proof}

From Proposition \ref{ZwProj} we get the following:

\begin{proposition}\label{ZwProjChow}
  For $w \in \leftexp{I}{W}$ have have $[\bar Z_w] = \gamma(w) \pi_*([\bar Z^\emptyset_w])$ in $A^\bullet(\GZip^\mu)$.
\end{proposition}


\subsection{Computing the Cycle Classes of the Ekedahl-Oort Strata on $\GZip^\mu$}\label{HowCompute}

By putting together the above results we get the following procedure for computing the classes $[\bar Z_w]$ in $A^\bullet(\GZip^\mu)$ for $w \in \leftexp{I}{W}$:

For computations, it is convenient to replace the rings appearing in the diagram \eqref{EqKeyChow} with certain simpler rings mapping surjectively onto them. For this we consider the following diagram of graded algebras, in which all rings are either polynomial rings or subrings of polynomial rings:
\begin{equation}  \label{EqKeyChowLift}
  \begin{aligned}\xymatrix{
      S  & & S \otimes_{\BQ} S \ar[ll]_{\tilde \psi^*} \\
      S^{W_{I^{\rm o}}} \ar@{^(->}[u]_{\tilde \pi^*} & & S^{W_{I^{\rm o}}} \ar[ll]^{\tilde \beta^*}  \ar[u]_{\tilde\gamma^*}
}\end{aligned}
\end{equation}
Here we define the homomorphisms as follows:
\begin{enumerate}[(i)]
\item The homomorphism $\tilde\pi^*$ is the inclusion $S^{W_{I^{\rm o}}} \into S$.
\item The homomorphism $\tilde\beta^*$ is the identity.
\item The homomorphism $\tilde\gamma^*$ is the composition
\[
\tilde\gamma^*\colon S^{W_{I^{\rm o}}} = z(S^{W_I}) \ltoover{z^{-1}} S^{W_I} \vartoover{30}{f \sends f \otimes 1} S \otimes_{S^W} S.
\]
\item The homomorphism $\tilde\psi^*$ is given by
\[
f \otimes g \lsends z(f)\varphi(g).
\]
\end{enumerate}
Using Theorem \ref{PullbackKey} one readily checks that under the canonical surjections from the objects in the diagram \eqref{EqKeyChowLift} to the corresponding objects in the diagram \eqref{EqKeyChow} these two diagrams are compatible. Similarly, using Theorem \ref{GysinFormula}, one checks that the morphism $\pi_*\colon S/\CI S \to S^{W_{I^{\rm o}}}/ \CI S^{W_{I^{\rm o}}}$ lifts to a morphism $\tilde\pi_*\colon S \to S^{W_{I^{\rm o}}}$ given by the formula from Theorem \ref{GysinFormula}.

In the following, for a class $c$ in one of the algebras of \eqref{EqKeyChow}, we will refer to a lift of $c$ to the corresponding algebra in \eqref{EqKeyChowLift} as a formula for $c$. Then, for $w \in \leftexp{I}{W}$, we can compute a formula for $[\bar Z_w]$ as follows:
\begin{enumerate}[(i)]
\item Using the results from Subsection \ref{DiagonalClass} one finds a formula for the class of the diagonal $\Brh_e$ in $S \otimes S$.
\item The operator $\delta_w$ on $S\otimes_{S^W} S$ from Subsection \ref{Deltaw} lifts to an operator on $S \otimes S$ by letting the operator $\delta_w$ on $S$ from Subsection \ref{Deltaw} act on the first factor of $S \otimes S$. Then by Corollary \ref{dwFormula}, by applying this operator $\delta_w$ to a formula for $[\Brh_e]$ one gets a formula for the class $[\overline \Brh_w]$.
\item By the definition of the subscheme $Z^\emptyset_w$ of $\GZipFlag^\mu$, the image of a formula for $[\overline \Brh_w]$ under the homomorphism $\tilde\psi^*$ gives a formula for the class $[\bar Z^\emptyset_w]$.
\item By applying $\tilde\pi_*$ to a formula for $[\bar Z^\emptyset_w]$ one gets a formula for $\pi_*([\bar Z^\emptyset_w])$.
\item Using the results from Subsection \ref{Gammaw} one computes the number $\gamma(w)$.
\item Using Proposition \ref{ZwProj} by multiplying the results of the previous two steps we get a formula for $[\bar Z_w] = \gamma(w) \pi_*([\bar Z^\emptyset_w])$.
\end{enumerate}


\subsection{Functoriality in the Zip Datum}\label{FuncZip}

To simplify notation it is often convenient for the computations in Subsection~\ref{HowCompute} to replace $G$ by some other group $\Gtilde$. Here we explain that this is harmless as long as $G$ and $\Gtilde$ have the same adjoint group.

Let $(G,\mu)$ and $(\Gtilde,\mgtilde)$ be two pairs consisting of a reductive group over $\BF_p$ and a cocharacter defined over the algebraic closure $k$ of $\BF_p$. Let
\[
f\colon G \to \Gtilde
\]
be a map of algebraic groups over $\BF_p$ with $f \circ \mu = \mgtilde$. Let $\kappa$ (resp.~$\kgtilde$) be the field of definition of the conjugacy class of $\mu$ (resp.~of $\mgtilde$). Then $\kgtilde \subseteq \kappa$.

Let $P$ and $Q$ be the parabolics and $E$ the zip group attached to $(G,\mu)$ as in Subsection~\ref{GZip}. Let $\Ptilde$, $\Qtilde$ and $\Etilde$ the parabolics and zip group attached similarly to $(\Gtilde,\mgtilde)$. Then $f$ induces maps $P \to \Ptilde$, $Q \to \Qtilde$, and $E \to \Etilde$ and hence a morphism
\[
[f]\colon \GZip^{\mu} \lto \GZip[\Gtilde]^{\mgtilde} \otimes_{\kgtilde} \kappa
\]
of smooth algebraic quotient stacks over $\kappa$. Every map $f\colon G \to \Gtilde$ of algebraic groups can be factorized into a faithfully flat map $G \to G' = G/\Ker(f)$ and a closed embedding $G' \to G$. If $G$ is reductive, then $G'$ is reductive. Therefore the following lemma in particular implies that the pullback $[f]^*$ on Chow rings exists.

\begin{lemma}\label{InducedMorphGood}
\begin{assertionlist}
\item
If $f$ is flat, then $[f]$ is flat.
\item
If $f$ is a monomorphism, then $[f]$ is representable.
\end{assertionlist}
\end{lemma}

\begin{proof}
This follows from Lemma~\ref{QuotientMapRep} because if $f$ is a monomorphism, then the induced map $E \to \Etilde$ is also a monomorphism.
\end{proof}

Using the description of $\GZipFlag^{\mu}$ given in \eqref{EqAltDescGZipFlag} one sees that $f$ also induces a map $[\ftilde]$ on stacks of flagged $G$-zips making the diagram
\[\xymatrix{
\GZipFlag^{\mu} \ar[r]^-{[\ftilde]} \ar[d] & \GZipFlag[\Gtilde]^{\mgtilde} \otimes_{\kgtilde} \kappa \ar[d] \\
\GZip^{\mu} \ar[r]^-{[f]} & \GZip[\Gtilde]^{\mgtilde} \otimes_{\kgtilde} \kappa
}\]
commutative. Moreover the same arguments as above show that the pullback $[\ftilde]^*$ on Chow rings exists.


\begin{lemma}\label{FunctorialityGZip}
Suppose that $f$ induces an isomorphism of the adjoint groups $G^{\rm ad} \iso \Gtilde^{\rm ad}$.
\begin{assertionlist}
\item\label{FunctorialityGZip1}
Let $\Ztilde$ be the radical of $\Gtilde$. Let $(T,B,z)$ be a frame as in Subsection~\ref{Frame} for $(G,\mu)$. Set $\Ttilde := \Ztilde f(T)$ and $\Btilde := \Ztilde f(B)$. Then $(\Ttilde, \Btilde, f(z))$ is a frame for $(\Gtilde,\mgtilde)$.
\item\label{FunctorialityGZip2}
The map $f$ induces an isomorphism $W \iso \Wtilde$ of the Weyl groups with their set of simple reflections attached to $(G,B,T)$ and $(\Gtilde,\Btilde,\Ttilde)$, respectively.
\item\label{FunctorialityGZip3}
The morphism $[f]$ of algebraic stacks induces a homeomorphism of the underlying topological spaces.
\end{assertionlist} 
\end{lemma}

\begin{proof}
The hypothesis on $f$ means that $\Ker(f)$ is central and that $\Cent(\Gtilde)f(G) = \Gtilde$. As $\Ztilde f(G)$ is of finite index in $\Cent(\Gtilde)f(G)$ and $\Gtilde$ is connected, this implies $\Ztilde f(G) = \Gtilde$.
As $\Ztilde$ is a torus and clearly commutes with $f(T)$, $\Ttilde := \Ztilde f(T)$ is a torus. Its dimension is the reductive rank of $\Gtilde$. Hence it is a maximal torus. By hypothesis, $f$ induces a bijection between the roots of $(G,T)$ and of $(\Gtilde,\Ttilde)$. This shows that $\Ztilde f(B)$ is a Borel subgroup and that $f$ induces an isomorphism $W \iso \Wtilde$. This implies all remaining assertions. 
\end{proof}

We continue to assume that $f$ induces an isomorphism of the adjoint groups $G^{\rm ad} \iso \Gtilde^{\rm ad}$ and use the notation of the lemma. We identify $W$ with $\Wtilde$ via the isomorphism induced by $f$.

If we define $\GZipFlag[\Gtilde]^{\mgtilde}$ and $\Brh_{\Gtilde}$ using $(\Gtilde,\Ptilde,\Qtilde,\Btilde)$ then the key diagram \eqref{EqKey} and the corresponding diagram of Chow rings \eqref{EqKeyChow} is functorial for $f$. The induced map of $\BQ$-algebras $\Stilde \defeq \Sym(X^*(\Ttilde)_{\BQ}) \to S$ is equivariant for the action of $W$. More precisely, if we choose splittings of the exact sequences of tori
\begin{gather*}
1 \lto \Ker(f)^0 \lto T \lto T/\Ker(f)^0 \lto 1,\\
1 \lto f(T) \lto \Ttilde \lto \Ttilde/f(T) \lto 1,
\end{gather*}
then $\Stilde \to S$ is of the form
\[
\Stilde \epi \Sym(X^*(f(T))_{\BQ}) \liso \Sym(X^*(T/\Ker(f)^0)_{\BQ}) \mono S,
\]
where the second map is an isomorphism of $\BQ$-algebras with $W$-action. The map $\Stilde \to S$ is also equivariant for the action of the Frobenius because $f$ is defined over $\BF_p$.

\begin{proposition}\label{FuncChow}
Let $f\colon G \to \Gtilde$ be a map of algebraic groups defined over $\BF_p$ that induces an isomorphism on adjoint groups.
\begin{assertionlist}
\item\label{FuncChow1}
One has a commutative diagram of $\BQ$-linear maps
\begin{equation}\label{EqFuncPistar}
\begin{aligned}\xymatrix{
A^{\bullet}(\GZipFlag[\Gtilde]^{\mgtilde}) \ar[rr] \ar[d]_{\pi_*} & & A^{\bullet}(\GZipFlag^{\mu}) \ar[d]^{\pi_*} \\
A^{\bullet}(\GZip[\Gtilde]^{\mgtilde}) \ar[rr]^{\sim} & &  A^{\bullet}(\GZip^{\mu})
}\end{aligned}
\end{equation}
where the horizontal maps are the maps of $\BQ$-algebras induced by $f$. The lower horizontal map is an isomorphism.
\item\label{FuncChow2}
For $w \in {}^IW$ the numbers $\gamma(w)$ defined in Definition~\ref{DefGammaw} for $(G,\mu)$ coincide with those defined for $(\Gtilde,\mgtilde)$.
\end{assertionlist}
\end{proposition}

\begin{proof}
Under the identifications \eqref{EqChowGZip1} and \eqref{EqChowGZip2} the horizontal maps are both induced by the $W$-equivariant map $\Stilde \to S$. Hence the commutativity of \eqref{EqFuncPistar} follows from the concrete description of $\pi_*$ in Theorem~\ref{GysinFormula}. From Proposition~\ref{ChowGZipNaive} and Lemma~\ref{FunctorialityGZip}~\ref{FunctorialityGZip3} we also deduce that the lower horizontal map sends a $\BQ$-basis to a $\BQ$-basis. In particular it is an isomorphism.

Let us show \ref{FuncChow2}. The upper horizontal map sends for all $w \in W$ the cycle $[\bar{Z}^{\emptyset}_w]$ defined for $(\Gtilde,\mgtilde)$ to the cycle $[\bar{Z}^{\emptyset}_w]$ defined for $(G,\mu)$ because $\psi^*$ is functorial for $f$. Hence \ref{FuncChow2} follows from \ref{FuncChow1} and Proposition~\ref{ZwProjChow}.
\end{proof}

%% file: Ch5CycleEO.tex
\section{The Tautological Ring of a Shimura Variety}\label{TAUT}

\subsection{Automorphic Bundles and the Tautological Ring in Characteristic Zero}
Let $(\Gbf,\Xbf)$ be a Shimura datum, i.e., $\Gbf$ is a connected reductive group over $\BQ$ and $\Xbf$ is a $\Gbf(\BR)$-conjugacy class of homomorphisms $h\colon \BS \to \Gbf_{\BR}$ of real algebraic groups, where $\BS \defeq \Res_{\BC/\BR}\Gm[\BC]$ is $\BC^{\times}$ viewed as a real algebraic group. The pair $(\Gbf,\Xbf)$ satisfies a list of axioms \cite[2.1.1]{Deligne_ShimuraInterpretModulaire}. 

For $h \in \Xbf$ let $\mu_h$ be the associated cocharacter of $\Gbf_{\BC}$, i.e., $\mu_h$ is the restriction of
\[
h_{\BC}\colon \BS_{\BC} = \prod_{\Gal(\BC/\BR)}\BG_{m,\BC} \lto G_{\BC}
\]
to the factor indexed by $\id \in \Gal(\BC/\BR)$. For every finite-dimensional representation $\rho\colon \Gbf_{\BR} \to \GL(V)$ the Hodge filtration induced by $\rho \circ h$ on $V$ has as stabilizer the parabolic subgroup $P_-(\rho \circ \mu_h)$ of $\GL(V)$ (here we follow the normalizations of \cite{Deligne_ShimuraInterpretModulaire}). The $\Gbf(\BC)$-conjugacy class of $\mu_h$ has as field of definition a finite extension $E$ of $\BQ$, called the reflex field.

Let $\Xbf\vdual$ be the compact dual of $\Xbf$. Then $\Xbf\vdual= \Par_{\Gbf_{\BC},\mu_h^{-1}}$ is the scheme of parabolic subgroups of type $\mu_h^{-1}$. It is a projective homogeneous $\Gbf$-space and it is defined over $E$.

For each neat open compact subgroup $K$ of $\Gbf(\BA_f)$ we denote by $\Sbf_K \defeq Sh_K(\Gbf,\Xbf)$ the canonical model of the attached Shimura variety at level $K$. This is a smooth quasi-projective scheme over $E$.

Denote by $\Gbf^c$ the quotient of $\Gbf$ by the maximal $\BQ$-anisotropic $\BR$-split torus in the center of $\Gbf$. For instance, if $(\Gbf,\Xbf)$ is of Hodge type, then $\Gbf = \Gbf^c$ but in general these groups differ, for instance if $\Gbf = \Res_{F/\BQ}\GL_{2,F}$ for a nontrivial totally real extension $F$ of $\BQ$. The action of $\Gbf_E$ on the $E$-scheme $\Xbf\vdual$ factors through $\Gbf^c_E$.

Milne constructs in \cite[III]{Milne_CanonicalModel} a diagram of schemes defined over $E$
\begin{equation}\label{EqMilneAut}
\begin{aligned}\xymatrix{
& \tilde{\Sbf}_K \ar[dl]_{\pi} \ar[dr]^{\tilde\sigma} \\
\Sbf_K & & \Xbf\vdual,
}\end{aligned}
\end{equation}
where $\pi$ is a $\Gbf^c_{E}$-torsor and $\tilde\sigma$ is $\Gbf_E$-equivariant. We set
\begin{equation}\label{EqDefHdg}
\Hdg_E  \defeq [\Gbf^c_E\bs \Xbf\vdual]
\end{equation}
which is an algebraic stack over $E$. The diagram \eqref{EqMilneAut} corresponds to a morphism of algebraic stacks
\begin{equation}\label{EqDefSigma}
\sigma\colon \Sbf_K \lto \Hdg_E
\end{equation}
making
\[\xymatrix{
\tilde{\Sbf}_K \ar[r]^{\tilde\sigma} \ar[d]_{\pi} & \Xbf\vdual \ar[d] \\
\Sbf_K \ar[r]^{\sigma} & \Hdg_E
}\]
cartesian.

Let $\Sbf^{\tor}_K$ be a smooth toroidal compactification of $\Sbf_K$. Then by \cite[V, Theorem 6.1]{Milne_CanonicalModel} the morphism $\sigma$ canonically extends to a morphism
\[
\sigma^{\tor}\colon \Sbf_K \lto \Hdg_E.
\]
Note that a vector bundle on the quotient stack $\Hdg_E = [\Gbf^c_E\bs \Xbf\vdual]$ is the same as a $\Gbf^c$-equivariant vector bundle on $\Xbf\vdual$.

\begin{definition}\label{DefAutBundle}
Let $E'$ be an extension of $E$. A vector bundle $\CE$ on $\Sbf_{K,E'}$ (resp.~on $\Sbf^{\tor}_{K,E'}$) is called \emph{automorphic bundle} if there exists a vector bundle $\Escr$ on $[\Gbf^c_E\bs \Xbf\vdual]_{E'}$ such that $\CE \cong \sigma^*(\Escr)$ (resp.~such that $\CE \cong (\sigma^{\tor})^*(\Escr)$. Moreover, $(\sigma^{\tor})^*(\Escr)$ is called the \emph{canonical extension} of $\sigma^*(\Escr)$.
\end{definition}

\begin{remark}\label{DescribeHodgeStack}
Suppose that $\Xbf\vdual(E') \ne \emptyset$, i.e., there exists a parabolic subgroup $\Pbf$ of $\Gbf_{E'}$ of type $\mu^{-1}$ that is defined over $E'$. Then the choice of $\Pbf$ yields  isomorphisms $\Xbf\vdual_{E'} \cong \Gbf_{E'}/\Pbf \cong \Gbf^c_{E'}/\Pbf^c$, where $\Pbf^c$ is the image of $\Pbf$ in $\Gbf^c_{E'}$. We obtain an isomorphism
\[
\Hdg_{E'} \cong \B{\Pbf^c}.
\]
Hence in this case a vector bundle $\Escr$ on $[\Gbf^c_E\bs \Xbf\vdual]_{E'}$ is the same as a finite-dimensional representation $(V,\eta)$ of $\Pbf^c$ over $E$ and $\sigma^*(\Escr)$ is the automorphic bundle attached to the representation $\eta$.
\end{remark}

The structure morphism $\Xbf\vdual \to \Spec E$ induces a morphism of algebraic stacks
\[
\tau\colon \Hdg_E \lto \B{\Gbf^c_E}.
\]

\begin{definition}\label{DefFlatAutBundle}
Let $E'$ be an extension of $E$. An  automorphic vector bundle on $\Sbf_{K,E'}$ (resp.~on $\Sbf^{\tor}_{K,E'}$) is called \emph{flat} if it is isomorphic to a vector bundle obtained by pullback via $\sigma \circ \tau$ (resp.~via $\sigma^{\tor} \circ \tau$) from a vector bundle on $\B{\Gbf^c_{E'}}$.
\end{definition}

In other words, flat automorphic bundles are those given by representations of $G^c$. They are endowed with a canonical integrable connection.

\begin{definition}\label{DefTautologicalRing}
Let $E'$ be a field extension of $E$. Then the images of $A^{\bullet}(\Hdg_{E'})$ in $A^{\bullet}(\Sbf_{K,E'})$ and in $A^{\bullet}(\Sbf^{\tor}_{K,E'})$ are called the \emph{tautological rings} of $\Sbf_{K,E'}$ and of $\Sbf^{\tor}_{K,E'}$, respectively. They are denoted by $\CT_{E'}$ and $\CT_{E'}^{\tor}$, respectively.
\end{definition}


\begin{remark}\label{TautScalarExt}
Let $E'$ be a field extension of $E$ and let $E''$ be a Galois extension of $E'$ with Galois group $\Gamma$. Then $\Gamma$ acts on $\CT_{E''}$ and one has $(\CT_{E''})^{\Gamma} = \CT_{E'}$ by Proposition~\ref{GaloisInvariants}.

In particular assume that the reductive group $\Gbf$ splits over $E''$. Then we can choose $\Pbf \in \Xbf\vdual(E'')$ and
\[
A^{\bullet}(\Hdg_{E''}) \cong A^{\bullet}(\B{\Pbf}) = A^{\bullet}(\B{\Lbf}) \cong \Sym(X^*(\Tbf)_{\BQ})^{W_{\Lbf}},
\]
where $\Lbf$ is the Levi quotient of $\Pbf$, $\Tbf \subseteq \Lbf$ a maximal torus and $W_{\Lbf}$ the Weyl group of $\Lbf$. Hence $\CT_{E'}$ is a quotient of $\Sym(X^*(\Tbf)_{\BQ})^{\Gamma \ltimes W_{\Lbf}}$.
\end{remark}

\begin{example}\label{TautologicalRingSiegel}
In the Siegel case, we have $\Gbf = \Gbf^c = \GSp_{2g}$ and $\Pbf$ is a Siegel parabolic subgroup, i.e., the stabilizer of some Lagrangian subspace. We identify $\Sbf_K$ with the moduli space of principally polarized abelian varieties of dimension $g$ endowed with some sufficiently fine level structure. Let $f\colon \CA \to \Sbf_K$ be the universal abelian scheme over $\Sbf_K$.

The Hodge stack $\Hdg = \B{\Pbf}$ parametrizes in this case vector bundles together with a symplectic pairing that has values in some line bundle and Lagrangian subbundles. The morphism $\sigma$ is the classifying map of the De-Rham cohomology of $\CA$ and its Hodge filtration where the pairing is induced by the principal polarizations.

The projection of $\Pbf$ onto its Levi quotient $\Lbf$ yields an isomorphism $A^{\bullet}(\Hdg) \cong A^{\bullet}(\B{\Lbf})$. Note that $\Lbf \cong \GL_g \times \BG_m$ in this case. The projection $\GL_g \times \BG_m \to \GL_g$ yields a vector bundle $\Omega^{\flat}$ on $\B{\Lbf}$ whose pullback to $\Sbf_K$ is the Hodge filtration bundle $f_*\Omega^1_{\CA/\Sbf_K}$. Moreover, the projection $\Lbf \to \BG_m$ is the restriction of the multiplier character of $G$ and hence the pullback of the corresponding line bundle to $\Sbf_K$ is trivial. Therefore in this case the tautological ring is the $\BQ$-subalgebra generated by the Chern classes of $f_*\Omega^1_{\CA/\Sbf_K}$ and our notion agrees with the one introduced in \cite{EkedahlVanDerGeer_CycleAbVar}.
\end{example}


\subsection{Stacks of Filtered Fiber Functors}

In Remark~\ref{DescribeHodgeStack} we explained that $\Hdg_{E'}$ is the classifying stack of a certain parabolic subgroup $\Pbf^c$ of $\Gbf^c$ if such a subgroup can be defined over $E'$. In this case, $\Hdg_{E'}$ simply classifies $\Pbf^c$-torsors. Let us briefly digress to give a moduli theoretic description of $\Hdg$ in general.

Hence for the moment let $k$ be any field, let $G$ be a reductive group over $k$, and let $\lambda$ be a cocharacter of $G$ defined over some extension of $k$. Suppose that that the conjugacy class of $\lambda$ is defined over $k$ or, equivalently, that the scheme $\Par_{\lambda}$ of parabolic subgroups of type $\lambda$ is defined over $k$. The reductive group scheme $G$ acts on $\Par_{\lambda}$ and we consider the quotient stack
\[
\Hdg_{G,\lambda} \defeq [G \bs \Par_{\lambda}].
\]
Clearly, $\Hdg_{G,\lambda}$ is a smooth algebraic stack over $S$.

Denote by $\Rep(G)$ the abelian $\otimes$-category of finite-dimensional representations of $G$. For any $S$-scheme $T$ we denote by $\FilLF(T)$ the exact rigid tensor category of filtered finite locally free $\Oscr_T$-modules (\cite[4C]{PWZ2}).

\begin{proposition}\label{CharHodge}
The stack $\Hdg_{G,\lambda}$ is canonically equivalent to the stack $\CF_\lambda$ which sends an $S$-scheme $T$ to the groupoid $\CF_\lambda(T)$ of exact $k$-linear $\otimes$-functors $\Rep(G) \to \FilLF(T)$ of type $\lambda$ (c.f. \cite[5.3]{PWZ2}).
\end{proposition}

\begin{proof}
  First we construct a canonical morphism $\CF_\lambda \to \Hdg_{G,\lambda}$ as follows: Let $T$ be an $S$ scheme and $\phi \colon \Rep(G) \to \FilLF(T)$ be an exact $k$-linear tensor functor $\Rep(G) \to \FilLF(T)$ of type $\lambda$. Similarly, let $\phi_\lambda \colon \Rep(G) \to \FilLF(S')$ be the exact $k$-linear tensor functor induced by the cocharacter $\lambda$. Then, by definition, the fact that $\phi$ is of type $\lambda$ means that there exists an fpqc covering $T'$ of $T \times_S S'$ over which the functors $\phi$ and $\phi_\lambda$ become isomorphic. The group of automorphisms of $\phi_{\lambda,T'}$ is $P_+(\lambda)_{T'}$. Hence the sheaf $\UIsom^\otimes(\phi_{\lambda,T'},\phi_{T'})$ of tensor isomorphisms $\phi_{\lambda,T'} \to \phi_{T'}$ is a right $P_+(\lambda)_{T'}$-torsor over $T'$. Thus under the canonical isomorphism $[*/P_+(\lambda)_{T'}]\cong (\Hdg_{G,\lambda})_{T'}$ noted above we obtain an object $\CP_\phi$ of $\Hdg_{G,\lambda}(T')$. Since $\phi$ is defined over $T$, there is a canonical descent datum for $T'/T$ on $\phi_{T'}$. This descent datum induces an analogous descent datum on $\CP_\phi$, so that $\CP_\phi$ descends canonically to an object of $\Hdg_{G,\lambda}(T)$. Finally one checks that the assignment $\phi \mapsto \CP_\phi$ naturally extends to a morphism of groupoids $\CF_\lambda(T) \to \Hdg_{G,\lambda}(T)$ and that for varying $T$ these morphisms are compatible with base change.

To check that the morphism $\CF_\lambda \to \Hdg_{G,\lambda}$ is an isomorphism of stacks we may work fpqc-locally on $S$. Hence we may assume that $\lambda$ is defined over $S$. Then, under the above isomorphism $\Hdg_{G,\lambda} \cong [*/P_+(\lambda)]$, the claim is given by \cite[Theorem 5.6]{PWZ2}.
\end{proof}


\subsection{The Tautological Ring in Positive Characteristic}\label{TAUTPOS}

From now on we assume that the Shimura datum $(\Gbf,\Xbf)$ is of Hodge type. Then $\Gbf = \Gbf^c$. Let $p$ be a prime of good reduction, i.e., there exists a reductive group scheme $\Gscr$ over $\BZ_p$ such that $\Gscr_{\BQ_p} = \Gbf_{\BQ_p}$. We fix a neat level structure $K = K^pK_p \subseteq \Gbf(\BA_{f})$ with $K^p \subseteq \Gbf(\BA^p_{f})$ compact open and $K_p = \Gscr(\BZ_p) \subseteq \Gbf(\BQ_p)$ hyperspecial. 

\subsubsection*{Integral models}

We fix a place $v$ of the reflex field $E$ over $p$ and denote by $E_v$ the $v$-adic completion of $E$. As $\Gbf_{\BQ_p}$ has a reductive model over $\BZ_p$, $E_v$ is an unramified extension $\BQ_p$. Let $\Sscr_K$ be its canonical smooth integral model over the ring of integers $O_{E_v}$ defined by Kisin (\cite{Kisin_IntegralModels}) and Vasiu (\cite{Vasiu_IntegralModels}) for $p > 2$ and by Kim and Madapusi-Pera (\cite{KimMadapusi_2AdicIntegralModels}) for $p = 2$. Let $\QQbreve_p$ be the completion of a maximal unramified extension of $E_v$ and let $k$ be the residue field of the ring of integers of $\QQbreve_p$. Let $\kappa$ be the residue field of $O_{E_v}$. Then $k$ is an algebraic closure of $\kappa$. Let $S_K$ be the special fiber of $\Sscr_K$ over $\kappa$ and let $G$ be the special fiber of $\Gscr$. Hence $G$ is a reductive group over $\BF_p$. 

By definition, $E$ is the field of definition of the conjugacy class of $\mu_h$. As conjugacy classes of cocharacters depend only on the root datum of the reductive group, we can view the $\Gbf(\BC)$-conjugacy class of $\mu_h$ also as a $\Gbf(\QQbreve_p)$-conjugacy $[\mu_h]_{\QQbreve_p}$ of cocharacters of $\Gbf_{\QQbreve_p}$ because $\Gbf_{\BQ_p}$ splits over an unramified extension. We may also view it as a $G(k)$-conjugacy class $[\mu_h]_k$ of cocharacters of $G_k$. The field of definition of $[\mu_h]_{\QQbreve_p}$ is $E_v$ and the field of definition of $[\mu_h]_k$ is $\kappa$.

As $\Gscr$ and $G$ are quasi-split, we may choose an element in $[\mu_h]_{\QQbreve_p}$ that extends to a cocharacter $\mu$ of $\Gscr$ defined over $O_{E_v}$. We also denote by $\mu$ its reduction modulo $p$, a cocharacter of $G_{\kappa}$. As $\mu_h$ is minuscule, $\mu$ is minuscule. 

\subsubsection*{Arithmetic Compactifications}

We recall some results on integral compactifications by Madapusa-Pera \cite{Madapusi_ToroidalHodge}. Let $\Sscr_K^{\tor}$ be some smooth proper toroidal compactification of the integral model $\Sscr_K$. It depends on the choice of a smooth, finite, admissible rational polyhedral cone decomposition. Moreover let $\Sscr^{\min}_K$ be the minimal compactification of $\Sscr_K$ and let
\[
\pi\colon \Sscr^{\tor}_K \lto \Sscr^{\min}_K
\]
be the canonical morphism. It is constructed as the Stein factorization of a certain proper morphism \cite[5.2.1]{Madapusi_ToroidalHodge}. In particular $\pi$ is proper and has geometrically connected fibers and for every line bundle $\Lscr$ on $\Sscr^{\min}_K$ one has a canonical isomorphism
\begin{equation}\label{EqPushPullMin}
\Lscr \liso \pi_*\pi^*\Lscr.
\end{equation}
We denote the special fibers over $\kappa$ of $\Sscr^{\tor}_K$ and of $\Sscr^{\min}_K$ by $S_K^{\tor}$ and $S_{K}^{\min}$, respectively. The restriction of $\pi$ to the special fibers is again denoted by $\pi$.

Recall that a morphism $f\colon X \to Y$ of finite type between noetherian schemes is called \emph{normal} if it is flat and has geometrically normal fibers. This notion is stable under base change $Y' \to Y$ and if $Y$ is normal and $X \to Y$ is normal, then $X$ is normal.

\begin{lemma}\label{MinCompNormal}
The minimal compactification $\Sscr^{\min}_K$ is normal over $O_{E_v}$.
\end{lemma}

\begin{proof}
For Shimura varieties of PEL type this is shown in \cite[7.2.4.3]{Lan_CompactifyPEL} using the description of completed local rings of $\Sscr^{\min}_K$ in geometric points (\cite[7.2.3.17]{Lan_CompactifyPEL}). But the same description also holds for the minimal compactification for Shimura varieties of Hodge type by \cite[5.2.8]{Madapusi_ToroidalHodge}.
\end{proof}

\subsubsection*{The Tautological Ring in Characteristic $p$}

The cocharacter $\mu\colon \Gm[O_{E_v}] \to \Gscr_{O_{E_v}}$ defines a parabolic subgroup $\Pscr \defeq P_-(\mu)$ of $\Gscr_{O_{E_v}}$ and we set
\[
\Hdg_{O_{E_v}} \defeq \B{\Pscr}.
\]
Then $\Hdg_{O_{E_v}} \otimes_{O_{E_v}} E_v = \Hdg_{E_v}$ by Remark~\ref{DescribeHodgeStack}. We denote by $P \defeq \Pscr_{\kappa}$ the special fiber of $\Pscr$ which is a parabolic subgroup of $G_{\kappa}$. Then we have
\[
\Hdg_{\kappa} \defeq \Hdg_{O_{E_v}} \otimes_{O_{E_v}} \kappa = \B{P}.
\]
By \cite[5.3]{Madapusi_ToroidalHodge} the morphism $\sigma$ and $\sigma^{\tor}$ extend to a morphism
\begin{equation}\label{EqToHdg}
\sigma\colon \Sscr_K \to \Hdg_{O_{E_v}}, \qquad \sigma^{\tor}\colon \Sscr^{\tor}_K \to \Hdg_{O_{E_v}}
\end{equation}
of smooth algebraic stacks over $O_{E_v}$. Let $O'$ be a local finite flat extension of $O_{E_v}$. As $\Hdg_{O_{E_v}} = \B{\Pscr}$, a vector bundle on $\Hdg_{O'}$ is given by an algebraic representation $\rho$ of the group scheme $\Pscr_{O'}$ on some finite free $O'$-module. The pullback of such a vector bundle to $\Sscr_{K,O'}$ via $\sigma$ (resp.~to $\Sscr^{\tor}_{K,O'}$ via $\sigma^{\tor}$) is denoted by
\[
\Vscr(\rho) \qquad (\text{resp.\ }\Vscr(\rho)^{\tor}).
\]
Again we define vector bundles on $\Sscr_K$ of this form to be \emph{automorphic vector bundles} and $\Vscr(\rho)^{\tor}$ is the canonical extension of $\Vscr(\rho)$ to the toroidal compactification $\Sscr_K^{\tor}$.

The morphisms $\sigma$ and $\sigma^{\tor}$ induce on special fibers morphisms
\begin{equation}\label{EqHdgMorph}
\sigma\colon S_K \to \Hdg_{\kappa}, \qquad \sigma^{\tor}\colon S^{\tor}_K \to \Hdg_{\kappa}
\end{equation}
of smooth algebraic stacks over $\kappa$. Again we have the notion of an automorphic bundle on $S_K$ and its canonical extension to $S^{\tor}_K$.

We now define the tautological rings in positive characteristic as in characteristic $0$.

\begin{definition}\label{DefTautologicalRingCharp}
Let $\kappa'$ be a field extension of $\kappa$, then the images of $A^{\bullet}(\Hdg_{\kappa'})$ in $A^{\bullet}(S_{K,\kappa'})$ and in $A^{\bullet}(S^{\tor}_{K,\kappa'})$ are called the \emph{tautological rings} of $S_{K,\kappa'}$ and of $S^{\tor}_{K,\kappa'}$, respectively. They are denoted by $\CT_{\kappa'}$ and $\CT_{\kappa'}^{\tor}$, respectively.
\end{definition}

%


\section{Cycle Classes of Ekedahl-Oort Strata}\label{EOTAUT}

We continue to use the notation of Subsection~\ref{TAUTPOS}, i.e., $\Sbf_K/E$ denotes the Shimura variety attached to a Shimura datum of Hodge type $(\Gbf,\Xbf)$ and a neat open compact subgroup $K \subset \Gbf(\BA_f)$, $\Sscr_K/O_{E_v}$ denotes its smooth integral model at a prime $p$ of good reduction, $S_K/\kappa$ its special fiber. We denote by $\Sscr_K^{\tor}$ a fixed smooth proper toroidal compactification of $\Sscr_K$ and by $S^{\tor}_K$ its special fiber. Moreover, $\Gscr$ denotes the reductive model of $\Gbf_{\BQ_p}$ which is endowed with a cocharacter $\mu$ defined over $O_{E_v}$. We denote by $(G,\mu)$ the special fiber of $(\Gscr,\mu)$.

From now on we assume that $p > 2$. This hypothesis is only needed for the existence of the smooth morphism $\zeta$ and the morphism $\zeta^{\tor}$ defined below in \eqref{EqZeta} and \eqref{EqZetaTor}. It seems probable that these morphisms also exist with the stated properties for $p = 2$ using ideas from \cite{KimMadapusi_2AdicIntegralModels}.

\subsection{Ekedahl-Oort Strata}

From the reductive group $G$ over $\BF_p$ and the cocharacter $\mu\colon \Gm[\kappa] \to G_{\kappa}$ we obtain the stack $\GZip^{\mu}$ recalled in Section~\ref{GZIP}. We use all notations introduced in Section~\ref{GZIP} for this pair $(G,\mu)$. In particular, we define $P \defeq P_-(\mu)$, a parabolic subgroup of $G$ of type $I$ which is defined over $\kappa$. The choice of $P$ yields an isomorphism $\Hdg_{\kappa} \cong \B{P}$. The morphism $\sigma$ \eqref{EqHdgMorph} is a morphism $\sigma\colon S_K \to \B{P}$.

In a series of papers (\cite{VW_PEL}, \cite{Zhang_EOHodge}, \cite{Wortmann_MuOrd}) Viehmann, the first author, Zhang, and Wortmann defined (for $p > 2$) a smooth morphism
\begin{equation}\label{EqZeta}
\zeta\colon S_K \lto \GZip^{\mu},
\end{equation}
which has also been extended to toroidal compactification
\begin{equation}\label{EqZetaTor}
\zeta^{\tor}\colon S^{\tor}_K \lto \GZip^{\mu}
\end{equation}
by Goldring and Koskivirta \cite[6.2.1]{GoKo_HasseHeckeGalois}. Moreover, one has by construction
\begin{equation}\label{EqVBComp}
\beta \circ \zeta = \sigma, \qquad \beta \circ \zeta^{\tor} = \sigma^{\tor}
\end{equation}
where $\beta$ is the morphism defined in \eqref{EqDefBeta}.

\begin{remark}\label{ZetaTorSmooth}
We conjecture the morphism $\zeta^{\tor}$ to be smooth as well. Beno\^it Stroh has sketched an argument (private communication) for a proof based on his joint work with Kai-Wen Lan \cite{LanStroh_Compactify} and we will give a detailed proof elsewhere. For now we add this as an additional assumption if necessary. This assumption is trivially satisfied if $S_K$ is proper over $\kappa$, i.e., if $\Gbf^{\rm ad}$ is $\BQ$-anisotropic.
\end{remark}

Recall the definition of the zip strata $Z_w \subseteq \GZip^{\mu}$ \eqref{EqDefZw}. The \emph{Ekedahl-Oort strata of $S_K$} (resp.~of $S^{\tor}_K$) are defined for $w \in {}^IW$ as
\[
S_{K,w} \defeq \zeta^{-1}(Z_w), \qquad S^{\tor}_{K,w} \defeq (\zeta^{\tor})^{-1}(Z_w).
\]
The smoothness of $\zeta$ implies the following properties of the Ekedahl-Oort strata.
\begin{assertionlist}
\item
For all $w \in {}^IW$ the $S_{K,w}$ are locally closed smooth subschemes of $S_K$. They are equi-dimensional of dimension $\ell(w)$ by \cite[5.11]{PWZ1}.
\item
By \eqref{EqClosureGZipStratum} one has
\begin{equation}\label{EqClosureEO}
\zeta^{-1}(\overline{Z}_w) = \overline{S}_{K,w} = \bigcup_{w' \preceq w}S_{K,w'}.
\end{equation}
\item
The map $\zeta^*\colon A^{\bullet}(\GZip^{\mu}) \to A^{\bullet}(S_K)$ of $\BQ$-algebras sends $[\overline{Z}_w]$ to $[\overline{S}_{K,w}]$.
\end{assertionlist}
If $\zeta^{\tor}$ is smooth, then the same properties hold for the subschemes $S^{\tor}_{K,w}$ of $S_K^{\tor}$ as well, in particular we find in this case
\begin{equation}\label{EqClosureEOtor}
(\zeta^{\tor})^{-1}(\overline{Z}_w) = \overline{S^{\tor}_{K,w}} = \bigcup_{w' \preceq w}S^{\tor}_{K,w'}.
\end{equation}

\begin{proposition}\label{AboutTautological}
The tautological rings $\CT$ and $\CT^{\tor}$ are finite-dimensional $\BQ$-algebras that are generated as a $\BQ$-vector space by $[\overline{S}_{K,w}]$ and $\zeta^{\tor,*}([\overline{Z}_w])$ for $w \in {}^IW$, respectively.
\end{proposition}

Note that if $\zeta^{\tor}$ was not flat, then it could happen that $\zeta^{\tor,*}([\overline{Z}_w]) \ne [\overline{S^{\tor}_{K,w}}]$ in $A^{\bullet}(S^{\tor}_K)$. Of course, that cannot happen if $\zeta^{\tor}$ is smooth.

\begin{proof}
By \eqref{EqVBComp} the tautological rings are the images of $A^{\bullet}(\GZip^{\mu})$. Hence the claim follows from Proposition~\ref{ChowGZipNaive}.
\end{proof}

We also recall and complement some results from \cite{GoKo_HasseHeckeGalois} about the Ekedahl-Oort strata in the minimal compactification. Let $\pi\colon S^{\tor}_K \to S_K^{\min}$ be the projection and set
\begin{equation}\label{EqDefEOMin}
S^{\min}_{K,w} \defeq \pi(S^{\tor}_{K,w}), \qquad w \in {}^IW.
\end{equation}
Then the $S^{\min}_{K,w}$ are pairwise disjoint, in other words $\pi^{-1}(S^{\min}_{K,w}) = S^{\tor}_{K,w}$, and locally closed in $S^{\min}_K$ by \cite[6.3.1]{GoKo_HasseHeckeGalois}. We endow $S^{\min}_{K,w}$ with the reduced scheme structure.

We will also use the following result of Goldring and Koskivirta.

\begin{theorem}[\cite{GoKo_HasseHeckeGalois}]\label{AffineEO}
The Ekedahl-Oort strata $S^{\min}_{K,w}$ in the minimal compactification are affine for all $w \in {}^IW$. In particular $S_{K,w}$ is quasi-affine for all $w$.
\end{theorem}

If $\zeta^{\tor}$ is smooth, then \eqref{EqClosureEOtor} implies
\begin{equation}\label{EqClosureEOmin}
\overline{S^{\min}_{K,w}} = \bigcup_{w' \preceq w}S^{\min}_{K,w'}
\end{equation}
because $\pi$ is proper. Below (Corollary~\ref{EOminDim}) we will also show that $S^{\min}_{K,w}$ is equi-dimensional of dimension $\ell(w)$ if $\zeta^{\tor}$ is smooth.


\subsection{Connectedness of unions of Ekedahl-Oort strata}\label{CONNEO}

In this subsection we show that the smoothness of $\zeta^{\tor}$ allows to deduce from \cite{GoKo_HasseHeckeGalois} certain results of connectedness of Ekedahl-Oort strata. These are new even in the Siegel case.
Hence until the end of Subsection~\ref{CONNEO} we assume that $\zeta^{\tor}$ is smooth.

\subsubsection*{Two Lemmas on Connectedness}

For lack of a reference we collect two probably well known lemmas.

For a topological space $X$ we denote by $\pi_0(X)$ the space of connected components of $X$. If $X$ is a noetherian scheme, then $\pi_0(X)$ is a finite discrete space.

\begin{lemma}\label{ConnectedLemma}
Let $f\colon X \to Y$ be a continuous map between topological spaces with connected (and hence non-empty) fibers. Suppose that the topology on $Y$ is the quotient topology of the topology on $X$ (e.g., if $f$ is closed or open). Then $Z \sends f^{-1}(Z)$ yields a bijection $\pi_0(Y) \iso \pi_0(X)$.
\end{lemma}

\begin{proof}
We may assume that $Y$ is connected. Assume $X$ could be covered by disjoint open non-empty subsets $U$ and $V$. As all fibers are connected, each fiber is either contained in $U$ or in $V$. This shows $f^{-1}(f(U)) = U$ and $f^{-1}(f(V)) = V$. Hence $f(U)$ and $f(V)$ are open, disjoint and cover $Y$. Contradiction.
\end{proof}

Let $l \geq 0$ be an integer. Recall that a noetherian scheme $X$ is called \emph{connected in dimension $l$} if $X \setminus Z$ is connected for every closed subset $Z \subseteq X$ of dimension $< l$. Hence $X$ is connected if and only if $X$ is connected in dimension $0$. A scheme $X$ of finite type over a field $k$ is called \emph{geometrically connected in dimension $l$} if $X \otimes_k k'$ is connected in dimension $l$ for all field extensions $k'$ of $k$.

We recall the following variant of a theorem of Grothendieck.

\begin{proposition}\label{GrothendieckConnected}
Let $k$ be a field, let $X$ be a proper $k$-scheme and let $D \subseteq X$ be an effective ample divisor. Let $l \geq 1$ be an integer. Suppose that the irreducible components of $X$ have dimension $\geq l+1$ and that $X$ is geometrically connected in dimension $\geq l$. Then the irreducible components of $D$ have dimension $\geq l$ and $D$ is geometrically connected in dimension $\geq l-1$.
\end{proposition}

\begin{proof}
Let $D$ be the vanishing locus of a section $s$ of an ample line bundle $\Lscr$. Replacing $\Lscr$ and $s$ by some power, we may assume that $\Lscr$ is very ample and hence that $X$ is a closed subscheme of projective space $\BP^N_k$ and that $D = X \cap H$ for some hyperplane $H$. Then the result follows from \cite[Exp.~XIII, 2.3]{SGA2}.
\end{proof}

\subsubsection*{Inheritance of Connectedness}

For any subset $B \subseteq {}^IW$ we set $Z_B := \bigcup_{w\in B}Z_w$, considered as a subspace of the underlying topological space of $\GZip^{\mu}$. We also set $S_{K,B} := \bigcup_{w \in B}S_{K,w}$ and define similarly subsets $S^{\tor}_{K,B}$ and $S^{\min}_{K,B}$ of $S^{\tor}_K$ and $S^{\min}_K$, respectively. Then $\zeta^{-1}(Z_B) = S_{K,B}$ and $(\zeta^{\tor})^{-1}(Z_B) = S^{\tor}_{K,B}$.

Now let $A \subseteq {}^IW$ be a closed subset, i.e., if $w \in A$ and $w' \in {}^IW$ with $w' \preceq w$, then $w' \in A$. Let $A^0$ be the set of maximal elements in $A$ with respect to $\preceq$ and set $\partial A := A \setminus A^0$. Then $Z_A$ is closed in $\GZip^{\mu}$ and $Z_{A^0}$ is open and dense in $Z_A$. We consider $Z_A$, $Z_{A^{0}}$, and $Z_{\partial A}$ as reduced locally closed algebraic substacks of $\GZip^{\mu}$.

Moreover, $S_{K,A}$ is closed in $S_K$, and $S_{K,A^0}$ is open and dense in $S_{K,A}$ by \eqref{EqClosureEO}. Analogous assertions also hold for unions of Ekedahl-Oort strata in $S_K^{\tor}$ and $S_K^{\min}$ by \eqref{EqClosureEOtor} and \eqref{EqClosureEOmin}.

\begin{lemma}\label{ConnectedEOStrata}
Let $Y \subseteq S^{\min}_K$ be a closed subscheme, let $A \subseteq {}^IW$ be closed as above, and let $l \geq 1$ be an integer. Suppose that the irreducible components of $Y_A := Y \cap S^{\min}_{K,A}$ have dimension $\geq l+1$ and that $Y_A$ is geometrically connected in dimension $\geq l$. Then the irreducible components of $Y_{\partial A} := Y \cap S^{\min}_{K,\partial A}$ have dimension $\geq l$ and $Y_{\partial A}$ is geometrically connected in dimension $\geq l-1$.
\end{lemma}

The proof relies heavily on results from \cite{GoKo_HasseHeckeGalois} using Proposition~\ref{GrothendieckConnected} as an additional ingredient.

\begin{proof}
For brevity we say that a scheme $X$ of finite type over a field is \emph{$l$-gc} if all irreducible components of $X$ have dimension $\geq l+1$ and $X$ is geometrically connected in dimension $\geq l$.

Let $Y^{\tor}_A \defeq \pi^{-1}(Y_A)$ and let
\[
Y^{\tor}_A \ltoover{\pi'} Y'_A \ltoover{f} Y_A
\]
be the Stein factorization of $\pi\colon Y^{\tor}_A \to Y_A$. As $\pi$ has geometrically connected fibers, the same holds for the finite morphism $f$. Hence $f$ is a universal homeomorphism. Therefore $Y'_A$ is $l$-gc and it suffices to show that $Y'_{\partial A} := f^{-1}(Y_{\partial A})$ is $(l-1)$-gc.

Let $\omega^{\tor}$ be the Hodge line on $S^{\tor}_K$ obtained from some Siegel embedding of the Shimura datum. Let $\omega^{\min} := \pi_*\omega^{\tor}$. By \cite[5.2.11]{Madapusi_ToroidalHodge} $\omega^{\min}$ extends the Hodge line bundle on $S_K$, it is ample, and $\pi^*(\omega^{\min}) \cong \omega^{\tor}$. The restrictions of $\omega^{\tor}$ and $\omega^{\min}$ to $Y^{\tor}_A$ and $Y_A$, respectively, are denoted by $\omega_{Y_A}^{\tor}$ and $\omega_{Y_A}^{\min}$. Set $\omega'_{Y_A} := f^*\omega_{Y_A}^{\min}$. Then for all $N > 1$ one has
\[
\pi'_*\omega_{Y_A}^{\tor,\otimes N} = \pi'_*\pi^{\prime*}(\omega^{\prime\otimes N}_{Y_A}) = \omega^{\prime\otimes N}_{Y_A},\tag{*}
\]
where the second equality follows from $\pi'_*(\Oscr_{Y^{\tor}_A}) = \Oscr_{Y'_A}$.

In the special case $Y = S^{\min}_K$ and $A = {}^IW$ one has $\pi' = \pi$ and we also see
\begin{equation}\label{EqPushPowerHodge}
\pi_*\omega^{\tor,\otimes N} \cong \omega^{\min,\otimes N}.
\end{equation}
By \cite[6.2.2]{GoKo_HasseHeckeGalois} there exists an $N \geq 1$ and for all $w \in {}^IW$ sections $h_w \in \Gamma(\overline{S^{\tor}_{K,w}},\omega^{\tor,\otimes N})$ whose non-vanishing locus is $S^{\tor}_{K,w}$. For $w,w' \in A^0$ with $w \ne w'$ one has
\[
{h_w}|_{\overline{S^{\tor}_{K,w}} \cap \overline{S^{\tor}_{K,w'}}} = 0 = {h_w}|_{\overline{S^{\tor}_{K,w}} \cap \overline{S^{\tor}_{K,w'}}}
\]
hence after passing to some power of $N$ and of $h_w$ we can glue the sections $h_w$ with $w \in A^0$ to a section $h_A \in \Gamma(S^{\tor}_{K,A},\omega^{\otimes N})$ whose non-vanishing locus is $S^{\tor}_{K,A^0}$ (\cite[5.2.1]{GoKo_HasseHeckeGalois}). We denote the restriction of $h_A$ to $Y^{\tor}_A$ again by $h_A$. Using (*) we obtain a section
\[
h_A \in \Gamma(Y^{\tor}_A,\omega_{Y_A}^{\tor,\otimes N}) = \Gamma(Y'_A,\omega^{\prime\otimes N}_{Y_A})
\]
whose vanishing locus in $Y'_A$ is precisely $Y'_{\partial A}$. As $\omega^{\prime\otimes N}_{Y_A}$ is the pullback of $\omega^{\min,\otimes N}$ under a finite morphism $Y'_A \to S^{\min}_K$, it is ample and we conclude by Proposition~\ref{GrothendieckConnected}.
\end{proof}

\subsubsection*{Connectedness of the length strata}

Let $d := \dim S_K = \dim S^{\tor}_K = \langle 2\rho,\mu \rangle$, where $\rho$ denotes as usual the half of the sum all positive roots on the root datum of $G$. For $j = 0,\dots,d$ we set
\[
S^?_{K,\leq j} \defeq \bigcup_{\ell(w) \leq j}S^?_{K,w}, \qquad S^?_{K,j} \defeq \bigcup_{\ell(w) = j}S^?_{K,w}
\]
for $? \in \{\emptyset, \tor,\min\}$. Then $S^?_{K,\leq j}$ is closed in $S^{?}_K$ and $S^?_{K,j}$ is open and dense in $S^{?}_{K,\leq j}$ by \eqref{EqClosureEO}, \eqref{EqClosureEOtor}, and \eqref{EqClosureEOmin}. We endow them with the reduced subscheme structure. The closed subschemes $S^?_{K,\leq j}$ of $S^{?}_K$ are called \emph{closed length strata}.

\begin{lemma}\label{LengthSmooth}
The schemes $S_{K,j}$ and $S^{\tor}_{K,j}$ are smooth.
\end{lemma}

\begin{proof}
By Lemma~\ref{LemPartialOrder} no two elements of ${}^IW$ of the same length are comparable with respect to $\preceq$. Hence $S^?_{K,j}$ is the topological sum of the $S^{?}_{K,w}$ for $w \in {}^IW$ with $\ell(w) = j$. This shows the lemma because $S_{K,w}$ and $S^{\tor}_{K,w}$ are smooth for all $w \in {}^IW$.
\end{proof}

Let $X$ be a scheme of finite type over a field $k$. By a \emph{geometric connected component of $X$} we mean a connected component $Y$ of $X_{\kbar}$, where $\kbar$ is an algebraic closure of $k$. Then $Y$ is already defined over some finite extension of $k$.
%

%

\begin{theorem}\label{ConnectedDimStrata}
Let $Y$ be a geometric connected component of $S^{\min}_K$ and let $Y^{\tor} \defeq \pi^{-1}(Y)$ be the corresponding (Lemma~\ref{ConnectedLemma}) geometric connected component of $S^{\tor}_K$. Then for all $j = 1,\dots, d$ the length strata $S^{\tor}_{K,\leq j} \cap Y^{\tor}$ and $S^{\min}_{K,\leq j} \cap Y$ are geometrically connected and equi-dimensional of dimension $j$. In particular they are non-empty.
\end{theorem}

\begin{proof}
We already know that $S^{\tor}_{K,\leq j}$ is equi-dimensional of dimension $j$. This shows that $S^{\tor}_{K,\leq j} \cap Y^{\tor}$ is either empty or equi-dimensional of dimension $j$.

Next we show that $S^{\min}_{K,\leq j} \cap Y$ is geometrically connected in dimension $\geq j-1$ and equi-dimensional of dimension $j$ by descending induction on $j$. We have $S^{\min}_{K,\leq d} = S^{\min}_K$. As $S_K^{\min}$ is (geometrically) normal (Lemma~\ref{MinCompNormal}), $Y$ is irreducible of dimension $d$ and in particular it is geometrically connected in dimension $\geq d - 1$.

Now let $A_j := \set{w \in {}^IW}{\ell(w) \leq j}$. Then $A_j^0 = \set{w \in {}^IW}{\ell(w) = j}$ and $\partial A = A_{j-1}$ by Lemma~\ref{LemPartialOrder}. Hence by induction we deduce from Lemma~\ref{ConnectedEOStrata} that $S^{\min}_{K,\leq j} \cap Y$ is geometrically connected in dimension $\geq j-1$ and that every irreducible component $Y$ of $S^{\min}_{K,\leq j} \cap Y$ has dimension $\geq j$. On the other hand we have $\dim(Y) \leq \dim S^{\min}_{K,j} = \pi(S^{\tor}_{K,j}) = j$.

This shows in particular that $S^{\min}_{K,\leq j} \cap Y$ is non-empty which implies that
\[
S^{\tor}_{K,\leq j} \cap Y^{\tor} = \pi^{-1}(S^{\min}_{K,\leq j} \cap Y)
\]
is non-empty. Moreover, $S^{\tor}_{K,\leq j} \cap Y^{\tor} = $ is geometrically connected by Lemma~\ref{ConnectedLemma}.
\end{proof}

\begin{corollary}\label{EOminDim}
The Ekedahl-Oort strata $S^{\min}_{K,w}$ in the minimal compactification is equi-dimensional of dimension $\ell(w)$.
\end{corollary}

\begin{proof}
Let $Y \subseteq S^{\min}_{K,w}$ be an irreducible component and let $\bar{Y}$ be its closure in $S^{\min}_K$. By \eqref{EqClosureEOmin}, $\bar{Y}$ is an irreducible component of $S^{\min}_{K,\leq \ell(w)}$. Hence $\dim(Y) = \dim(\bar{Y}) = \ell(w)$ by Theorem~\ref{ConnectedDimStrata}.
\end{proof}

\begin{corollary}\label{Length1Connected}
Let $S^{\tor}_{K,e}$ be the $0$-dimensional Ekedahl-Oort stratum in $S^{\tor}_K$. Suppose that $S^{\tor}_{K,e}$ is already contained in $S_K$. Let $Y$ be a geometric connected component of $S^{\tor}_K$. Then the length $1$ stratum $S_{K,\leq 1} \cap Y$ in $S_K \cap Y$ is geometrically connected.
\end{corollary}

The condition that $S^{\tor}_{K,e}$ is contained in $S_K$ is satisfied for all Shimura varieties of PEL type (\cite[6.4.1]{GoKo_HasseHeckeGalois}) and we expect it to hold in general.

\begin{proof}
Let $Y$ and $Y'$ be irreducible components of $S^{\tor}_{K,\leq 1}$. As $S^{\tor}_{K,\leq 1} \cap Y$ is geometrically connected by Theorem~\ref{ConnectedDimStrata}, it suffices to show that $Y \cap Y' \subseteq S^{\tor}_{K,e} = S_{K,e}$. But this is clear because $S^{\tor}_{K,1}$ is smooth (Lemma~\ref{LengthSmooth}) and hence cannot contain intersection points of irreducible components.
\end{proof}


\subsection{The flag space over $S_K^{\tor}$}\label{FlagSpace}
Let $\pi_K\colon \CF_K^{\tor} \to S_K^{\tor}$ be defined by the following fiber product:
\begin{equation*}
  \xymatrix{
    \CF_K^{\tor} \ar[r] \ar[d]_{\pi_K} & \GZipFlag^\mu \ar[d]^\pi \\
    S_K^{\tor} \ar[r]^{\zeta^{\tor}} & \GZip^\mu
}
\end{equation*}
Similarly, we let $\CF_K$ be the restriction of $\CF_K^{\tor}$ to $S_K$. As $\pi$ is representable by schemes, smooth, and proper, $\CF_K^{\tor}$ and $\CF_K$ are schemes, and $\pi_K$ is smooth and proper.

By pulling back the stratification $\GZipFlag^\mu=\bigcup_{w\in W}Z^\emptyset_w$ \eqref{EqZipFlagStrata} to $\CF_K$ and $\CF_K^{\tor}$ we obtain stratifications
\begin{equation*}
  \CF_K=\bigcup_{w\in W}\CF_{K,w}, \qquad \text{and}\qquad 
  \CF_K^{\tor}=\bigcup_{w\in W}\CF_{K,w}^{\tor}.
\end{equation*}

\begin{proposition} \label{FwProps}
The strata $\CF_{K,w}$ are smooth and equi-dimensional of dimension $\ell(w)$. Their closures $\overline{\CF_{K,w}}$ are normal, Cohen-Macaulay, with only rational singularities. If $\zeta^{\tor}$ is smooth, then an analogous result holds for $\CF^{\tor}_{K,w}$.
\end{proposition}

\begin{proof}
Since $\zeta$ is smooth, this follows from Proposition \ref{ZwProps}.
\end{proof}

The following result generalizes \cite[Proposition 6.1]{EkedahlVanDerGeer_CycleAbVar}:

\begin{proposition} \label{EOMinBoundary}
Suppose that $\zeta^{\tor}$ is smooth. Let $w \in \leftexp{I}{W}$ and $Y$ an irreducible component of $\overline{S^{\min}_{K,w}}$.
  \begin{enumerate}[(i)]
  \item The variety $Y$ has dimension $\ell(w)$ and is geometrically connected in dimension $\geq \ell(w)-1$.
  \item The intersection $Y \cap S^{\min}_{K,e}$ is non-empty.
  \end{enumerate}
\end{proposition}
\begin{proof}
  (i) As in the proof of Corollary \ref{EOminDim}, the variety $Y$ is an irreducible component of $S^{\min}_{K,\leq \ell(w)}$. Hence (i) follows from Theorem \ref{ConnectedDimStrata}.

(ii) Let $Y^\circ \defeq Y \cap S^{\min}_{K,w}$. Since $Y^\circ$ is an irreducible component of $S^{\min}_{K,w}$, it is affine by Theorem \ref{AffineEO}. Thus if $Y^\circ$ is closed in $S^{\min}_K$ it has dimension zero, which using (i) implies $\ell(w)=0$ and hence $w=e$. Otherwise $Y \setminus Y^\circ$ is non-empty. Let $Y'$ be an irreducible component of $Y\setminus Y^\circ = Y \cap \bigcup_{w' \prec w} S^{\min}_{K,w'}$. Using (i), Proposition \ref{ConnectedEOStrata} yields $\dim Y' \geq \ell(w)-1$. On the other hand, the inclusion $Y' \subset \bigcup_{w' \prec w} S^{\min}_{K,w'}$ and Corollary \ref{EOminDim} yield $\dim Y' \leq \ell(w)-1$. Thus $\dim Y'=\ell(w)-1$, which again by Corollary \ref{EOminDim} implies that $Y'$ must be an irreducible component of $\overline{S^{\min}_{K,w'}}$ for some $w' \prec w$ with $\ell(w')=\ell(w)-1$. Now we may conclude by induction on $\ell(w)$. 
\end{proof}

%% file: Ch7CycleEO.tex
\section{Applications}\label{APP}

\subsection{Triviality of Chern classes of flat automorphic bundles}

Let $E'$ be an extension of $E$. By definition, an automorphic bundle over $E'$ is a vector bundle on $Sh_K(\Gbf,\Xbf)_{E'}$ that arises by pullback of a vector bundle on $\Hdg_{E'}$ via the map $\sigma$. Recall that such an automorphic bundle is called \emph{flat} if it comes from a vector bundle on $\B{\Gbf_{E'}}$ by pullback via the composition
\[
Sh_K(\Gbf,\Xbf)_{E'} \ltoover{\sigma} \Hdg_{E'} \lto \B{\Gbf_{E'}},
\]
i.e., it is an automorphic bundle associated with a finite-dimensional representation of $\Gbf_{E'}$. Similarly, we define what it means for an automorphic bundle (or their canonical extensions to the toroidal compactification) on the integral model or its special fiber to be flat.

Now the theory of Chow rings of $G$-zips allows us easily to show the following result for flat automorphic bundles on the special fiber.

\begin{theorem}\label{TrivialFlatBundle}
Let $\kappa'$ be an extension of $\kappa$ and let $\Vscr$ be a flat automorphic bundle on $S_{K,\kappa'}$ and let $\Vscr^{\tor}$ be its canonical extension to $S_{K,\kappa'}^{\tor}$. Then for all $i \geq 1$ the $i$-th Chern class of $\Vscr$ in $A^{\bullet}(S_{K,\kappa'})$ and of $\Vscr^{\tor}$ in $A^{\bullet}(S_{K,\kappa'}^{\tor})$ is zero.
\end{theorem}

\begin{proof}
As all automorphic bundles are defined over some finite extension of $\kappa$, we may assume that $\kappa'$ is an algebraic extension of $\kappa$. By Proposition~\ref{GaloisInvariants} we may assume that $\kappa' = k$ is an algebraic closure of $\kappa$. As $\sigma$ and $\sigma^{\tor}$ both factor through $\GZip^{\mu}$, it suffices to show that under pullback via the composition
\[
\GZip^{\mu} \ltoover{\beta} \Hdg_{k} \ltoover{\nu} \B{G_{k}}
\]
all elements of degree $>0$ in $A^{\bullet}(\B{G_k})$ are send to $0$. Here $\nu$ is the canonical projection $\Hdg_K = \B{P_k} \to \B{G_k}$ which induces via pullback on Chow rings the inclusion $A^{\bullet}(\B{G_k}) = S^W \mono S^{W_{I^o}}= A^{\bullet}(\Hdg_k)$. Hence the description of $\CI$ in Lemma~\ref{RationalCI} and of $\beta^*$ in Theorem~\ref{PullbackKey} implies the claim. 
\end{proof}

Using proper smooth base change we obtain the triviality of \'etale Chern classes in characteristic $0$ as follows. For a scheme of finite type over a field $k$ we denote by $H^i(X,\BQ_{\ell}(d))$ be the $i$-th continuous $\ell$-adic cohomology with Tate twist defined by Jannsen (\cite{Jannsen_ContEtaleCoh}) or, equivalently, the pro-\'etale cohomology defined by Bhatt and Scholze (\cite{BhattScholze_Proet}). Here $\ell$ is a prime different from the characteristic of $k$. Recall that $\Sbf_K$ denotes a Shimura variety of Hodge type in characteristic $0$ and that $\Sbf^{\tor}_K$ denotes a smooth proper toroidal compactification of $\Sbf_K$.

\begin{corollary}\label{TrivialEtaleChern}
Let $E'$ be a finite extension of the reflex field $E$ contained in the algebraic closure $\Ebar$ of $E$ in $\BC$. Let $\Vscr$ be a flat automorphic bundle over $\Sbf_{K,E'}$ and let $\Vscr^{\tor}$ be its canonical extension to $\Sbf^{\tor}_{K,E'}$. Then for all $i \geq 1$ the $i$-th \'etale Chern classes $c_i(\Vscr) \in H^{2i}(\Sbf_{K,E'},\BQ_{\ell}(i))$ and $c_i(\Vscr^{\tor}) \in H^{2i}(\Sbf^{\tor}_{K,E'},\BQ_{\ell}(i))$ are zero.
\end{corollary}

This had been proved by Esnault and Harris for compact Shimura varieties in \cite{EsnaultHarris_AutomorphicVB}. 

\begin{proof}
By passing to a finite extension of $E'$ we may assume that $\Gbf_{E'}$ is split. Let $\Vscr$ be associated to a representation $\rho$ of $\Gbf_{E'}$. As $\Gbf_{E'}$ is reductive and finite-dimensional representations of reductive groups are semi-simple in characteristic $0$, we may assume that $\rho$ is irreducible. It suffices to show the triviality for $c_i(\Vscr^{\tor})$. Let $p \ne \ell$ be a prime number of good reduction for the Shimura datum $(\Gbf,\Xbf)$ such that exists an unramified place $v'$ of $E'$ over $p$. Let $v$ be the restriction of $v'$ to $E$ and let $\Sscr^{\tor}_K$ be a smooth proper toroidal compactification of $\Sscr_K$ over $O_{E_v}$ with generic fiber $\Sbf^{\tor}_{K,E_v}$. Let $\kappa'$ be the residue field of $O' \defeq O_{E'_{v'}}$. We also use the notation of Section~\ref{EOTAUT}. In particular we denote by $\Gscr$ a reductive model of $\Gbf_{\BQ_p}$ over $\BZ_p$.

We extend the representation $\rho$ to a representation $\tilde\rho$ of $\Gscr_{O'}$ on a free $O'$-module of finite rank as follows. Let $\Tscr \subseteq \Bscr \subseteq \Gscr$ be a Borel pair and let $\lambda \in X^*(\Tscr)$ be the highest weight of $\rho$. Considering $\lambda$ as a representation of $\Bscr_{O'}$ on a free $O'$-module of rank $1$ we obtain a line bundle on $\B{\Bscr_{O'}} = [\Gscr \bs \Gscr/\Bscr]_{O'}$, or equivalently a $\Gscr$-equivariant line bundle $\Lscr_{\lambda}$ on $\Gscr/\Bscr$. Then $\rho_{E'_{v'}}$ is given by the global sections of $\Lscr_{\lambda}$ over $(\Gscr/\Bscr)_{E'_{v'}}$. By Kempf's vanishing theorem we have $H^1((\Gscr/\Bscr)_s,(\Lscr_{\lambda})_s) = 0$ for all $s \in \Spec(O')$. Hence $H^0(\Gscr/\Bscr_{O'}, \Lscr_{\lambda})$ is a (locally) free $O'$-module whose formation is compatible with arbitrary base change $S' \to \Spec O'$. This is the desired extension.

Its special fiber is then a representation of the split reductive group $G_{\kappa'}$. Let $\Vscr^{\tor}_{\kappa'}$ be the corresponding flat automorphic bundle on $S^{\tor}_{K,\kappa'}$. It lifts to a flat automorphic bundle over $\Sscr^{\tor}_{K,O'}$ whose generic fiber is $\Vscr^{\tor}$. By Theorem~\ref{TrivialFlatBundle} the $i$-th Chern class of $\Vscr^{\tor}_{\kappa'}$ in $A^i(S_{K,\kappa'}^{\tor})$ is zero for $i \geq 1$. In particular its \'etale cycle class vanishes in
\[
H^{2i}(S^{\tor}_{\kappa'},\BQ_{\ell}(i)) = H^{2i}(\Sbf^{\tor}_{E'_{v'}},\BQ_{\ell}(i)) = H^{2i}(\Sbf^{\tor}_{E'},\BQ_{\ell}(i)),
\]
where the equalities hold by smooth and proper base change. But this cycle class in the last space is the \'etale cycle class of $\Vscr^{\tor}$ because the \'etale cycle class map from Chow groups to \'etale cohomology is compatible with specialization.
\end{proof}


\subsection{The Hodge half line}\label{HodgeHalf}

%

As the Shimura datum is of Hodge type there exists a Siegel embedding of $G$, i.e., an embedding $\iota\colon G \mono \GSp(V)$ of algebraic groups over $\BF_p$ such that $\tilde\mu\defeq \iota \circ \mu$ is minuscule and the parabolic $P_+(\tilde\mu)$ is the stabilizer of a Lagrangian subspace $U \subseteq V$. Consider the character
\begin{equation}\label{EqHodgeChar}
\chi(\iota) \defeq \det(V/U)\vdual
\end{equation}
of $P$, which is defined over $\kappa$. It corresponds to a line bundle on the Hodge stack $\Hdg$ over $\kappa$. We denote its pullback to $\GZip^{\mu}$ by $\omega^{\flat}(\iota)$. We call a class in $A^1(\GZip^{\mu})$ a \emph{Hodge line bundle class} if it is the first Chern class of the line bundle $\omega^{\flat}(\iota)$ given by a symplectic embedding.

Such a Hodge line bundle class is essentially independent of the choice of the embedding by combining Theorem~\ref{TrivialFlatBundle} with a result of Goldring and Koskivirta.

\begin{proposition}\label{HodgeLineBundlesIndep}
Suppose that $\Gbf^{\rm ad}$ is $\BQ$-simple. If $\iota$ and $\iota'$ are two Siegel embeddings, then there exists $\rho \in \BQ_{>0}$ such that
\[
c_1(\omega^{\flat}(\iota)) = \rho c_1(\omega^{\flat}(\iota')) \in A^1(\GZip^{\mu}).
\]
\end{proposition}

\begin{proof}
By Theorem~\ref{TrivialFlatBundle} it suffices to show that there exists a character $\lambda$ of $G$ and $m,n \in \BZ_{>0}$ such that $m\chi(\iota) = \lambda + n\chi(\iota')$ as characters of $P$ or, equivalently, of the Levi subgroup $L$. Let $\tilde{L}$ be the connected component of the preimage of $L$ in the simply connected cover of the derived group of $G$ and let $\tilde{\chi}$ and $\tilde{\chi}'$ be the characters obtained from $\chi(\iota)$ and $\chi(\iota')$, respectively, by composition with $\tilde{L} \to L$. Then it suffices to show that there exist $m,n \in \BZ_{>0}$ such that $m\tilde{\chi} = n\tilde{\chi}'$. But this is shown in \cite[1.4.5]{GoKo_Quasiconstant}.
\end{proof}

It is easy, as was explained to us by W.~Goldring, to give examples that the assertions fails without the assumption that $\Gbf^{\rm ad}$ is $\BQ$-simple. Indeed if $Sh_K(\Gbf,\Xbf)$ is the Shimura variety that classifies pairs of elliptic curves (with some level structure), then it is easy to see that the embeddings into the moduli space of principally polarized abelian threefolds given by
\[
(E_1,E_2) \sends E_1^2 \times E_2, \qquad\text{and}\qquad (E_1,E_2) \sends E_1 \times E_2^2
\]
yield classes in $A^1(\GZip^{\mu})$ that are not multiples of each other.

Let $\CT$ and $\CT^{\tor}$ be the tautological ring of $S_K$ and $S^{\tor}_K$, respectively.

\begin{definition}\label{DefHodgeHalfLine}
Suppose that $\Gbf^{\rm ad}$ is $\BQ$-simple. We call the $\BQ_{>0}$ half line in $A^1(\GZip^{\mu})$ generated by $c_1(\omega^{\flat}(\iota))$ the \emph{Hodge half line of $\GZip^{\mu}$}. Its image in the tautological rings $\CT_{\kappa}$ and $\CT_{\kappa}^{\tor}$ is also called the \emph{Hodge half line}.
\end{definition}

By \cite[5]{Madapusi_ToroidalHodge} we find that the pullback of a Hodge line bundle class to $\CT_{\kappa}$ (resp.~to $\CT_{\kappa}^{\tor}$) is generated by the determinant of the sheaf of invariant differentials of the abelian scheme (resp.~semi-abelian scheme) that is obtained via pullback from the universal abelian (resp.~semi-abelian) scheme over the Siegel Shimura variety (resp.~over a suitable toroidal compactification of the Siegel Shimura variety). In particular we find the following result.

\begin{proposition}\label{HodgeAmple}
The pullback of a Hodge line bundle class to $\CT_{\kappa}$ is ample.
\end{proposition}


\subsection{Powers of Hodge line bundle classes}

By Proposition~\ref{ChowGZip} and Proposition~\ref{ChowGZipNaive} the Chow ring $A^{\bullet}(\GZip^{\mu})$ is a graded finite-dimensional $\BQ$-algebra of dimension $\#{}^IW$. For $j = 0, \dots, d \defeq \langle 2\rho,\mu\rangle$ the $\BQ$-vector space $A^j(\GZip^{\mu})$ has a basis the cycle classes $[\overline{Z}_w]$ with $w \in {}^IW$ with $\ell(w) = d - j$. In particular its top degree $A^d(\GZip^{\mu})$ is $1$-dimensional and generated by the unique closed zip stratum $[Z_e]$ which we call the \emph{superspecial stratum}.


\begin{proposition}\label{PowerHodge}
Let $\lambda^{\flat} \in A^1(\GZip^{\mu})$ be a Hodge line bundle class. Then for all $j = 0, \dots, d$ one has
\[
(\lambda^{\flat})^{d-j} = \sum_{\substack{w \in {}^IW \\ \ell(w) = j}}\alpha_w[\overline{Z}_w],
\]
with $\alpha_w \in \BQ_{>0}$. In particular there exists $\alpha_e \in \BQ_{>0}$ such that
\begin{equation}\label{EqHodgePower}
(\lambda^{\flat})^d = \alpha_e[Z_e].
\end{equation}
\end{proposition}

\begin{proof}
This follows by an easy induction from \cite[5.2.2]{GoKo_HasseHeckeGalois}.
\end{proof}

\begin{remark}\label{EqualityCoeff}
Calculations of examples suggest that the coefficients $\alpha_w$ should be equal for $w \in {}^IW$ with $\ell(w) = j$ if $\Gbf^{\rm ad}$ is $\BQ$-simple. We cannot prove this.
\end{remark}

By pullback we obtain:

\begin{corollary}\label{ClassSuperspecial}
Let $\lambda^{\flat} \in A^1(\GZip^{\mu})$ be a Hodge line bundle class. Let $\lambda \in \CT$ be its pullback. Then for all $j = 0, \dots, d$ one has
\[
\lambda^{d-j} = \sum_{\substack{w \in {}^IW \\ \ell(w) = j}}\alpha_w[\overline{S_{K,w}}]
\]
with $\alpha_w \in \BQ_{>0}$. In particular there exists $\alpha_e \in \BQ_{>0}$ such that $\lambda^d = \alpha_e[S_e]$.
\end{corollary}


\subsection{Description of the tautological ring}

We will now show that the pullback map $\zeta^{\tor,*}\colon A^{\bullet}(\GZip^{\mu}_k) \to A^{\bullet}(S_{K,k}^{\tor})$ is always injective. By Proposition~\ref{GaloisInvariants} this also implies the injectivity of $\zeta^{\tor,*}\colon A^{\bullet}(\GZip^{\mu}_{\kappa'}) \to A^{\bullet}(S_{K,\kappa'}^{\tor})$ for every algebraic extension $\kappa'$ of $\kappa$. In particular we obtain an isomorphism of the tautological ring $\CT_{\kappa'}^{\tor}$ with $A^{\bullet}(\GZip^{\mu}_{\kappa'})$.  

The tool for showing injectivity is the following lemma.

\begin{lemma}\label{CritInjectivity}
Let $\alpha\colon A^{\bullet}(\GZip^{\mu}_k) \to T$ be a map of graded $\BQ$-algebras. Then $\alpha$ is injective if and only if $\alpha([Z_e]) \ne 0$.
\end{lemma}

\begin{proof}
It suffices to show that any graded non-zero ideal of $A^{\bullet}(\GZip^{\mu})$ contains $[Z_e]$. By Corollary \ref{ChowGZipCohomology}, $A^{\bullet}(\GZip^{\mu})$ is isomorphic to the rational cohomology ring of the flag space $\Xbf\vdual$ over $\BC$. In particular multiplication yields for all $j = 0,\dots,d = \dim \Sbf_K$ a perfect pairing
\[
A^{j}(\GZip^{\mu}) \times A^{d-j}(\GZip^{\mu}) \to A^{d}(\GZip^{\mu}) = \BQ[Z_e].
\]
This implies our claim. 
\end{proof}

\begin{theorem}\label{ZetaTorInj}
The map $\zeta^{\tor,*}$ is injective. One has
\begin{equation}\label{EqDescribeTaut}
\CT^{\tor}_k \cong A^{\bullet}(\GZip^{\mu}_k) \cong H^{2\bullet}(\Xbf\vdual).
\end{equation}
\end{theorem}

\begin{proof}
Let $\lambda^{\flat} \in A^1(\GZip^{\mu})$ be a Hodge line bundle class, say the first Chern class of a line bundle $\omega^{\flat}$ on $\GZip^{\mu}$. Let $\omega^{\tor} \defeq \zeta^{\tor,*}(\omega^{\flat})$. Let $\pi\colon S_K^{\tor} \to S^{\min}_K$ be the canonical proper birational map to the minimal compactification. By \cite[5.2.11]{Madapusi_ToroidalHodge} there exists an ample line bundle $\omega^{\min}$ on $S_K^{\min}$ such that $\pi^{*}(\omega^{\min}) \cong \omega^{\tor}$.

By Lemma~\ref{CritInjectivity} and \eqref{EqHodgePower} we have to show that
\[
\zeta^{\tor,*}(c_1(\omega^{\flat})^d \cap [\GZip^{\mu}_k]) = c_1(\omega^{\tor})^d \cap S_K^{\tor} \ne 0,\tag{*}
\]
where the equality holds by \cite[6.6]{Fulton_Intersection} because $\zeta^{\tor}$ is an l.c.i. morphism (note that any morphism $f\colon X \to \Yscr$ from a smooth algebraic space to a smooth algebraic stack is l.c.i. in the sense used in \cite{Fulton_Intersection}: factorize it into its graph $\Gamma_f\colon X \to X \times \Yscr$ which is representable and a regular immersion by \cite[17.12.1]{EGA4.4} and the representable smooth projection $X \times \Yscr \to \Yscr$). 

The projection formula shows
\[
\pi_*(c_1(\omega^{\tor})^d \cap S_K^{\tor}) = c_1(\omega^{\min})^d \cap [S^{\min}_K]
\]
which is non-zero because $\omega^{\min}$ is ample and $S_K^{\min}$ is proper and of pure dimension $d$ over $\kappa$. Hence the left hand side of (*) is non-zero.

The isomorphisms in \eqref{EqDescribeTaut} are then a consequence by using Corollary~\ref{ChowGZipCohomology}.
\end{proof}

It is conjectured that a similar description does also hold for the tautological ring of a smooth toroidal compactification of the Shimura variety in characteristic $0$. Let $E'$ be an algebraic extension of $E_v$ and let $\kappa'$ be the residue field of the ring of integers of $E'$. There is a commutative diagram
\begin{equation}\label{EqSpecialize}
\begin{aligned}\xymatrix{
A^{\bullet}(\Hdg_{E'}) \ar[r]^{\sigma^{\tor,*}} \ar[d]_{\cong} & A^{\bullet}(\Sbf^{\tor}_{K,E'}) \ar[d] \\
A^{\bullet}(\Hdg_{\kappa'}) \ar[r] & A^{\bullet}(S^{\tor}_{K,\kappa'})
}\end{aligned}
\end{equation}
where the vertical arrows are the specialization maps. For the Hodge stacks one can show that the specialization map is an isomorphism. In particular the right specialization map induces a surjective map of $\BQ$-algebras
\begin{equation}\label{EqSpecialTaut}
{\rm sp}^{\tor}\colon \CT^{\tor}_{E'} \to \CT^{\tor}_{\kappa'}.
\end{equation}

A similar diagram as \eqref{EqSpecialize} with the specialization of $A^{\bullet}(\Sbf_{K,E'}) \to A^{\bullet}(S_{K,\kappa'})$ as the right vertical arrow yields also a surjective map ${\rm sp}\colon \CT_{E'} \to \CT_{\kappa'}$.

\begin{proposition}\label{DescribeTautChar0}
Suppose that $E'$ is chosen such that $\kappa' = k$ is algebraically closed. Then the following assertions are equivalent.
\begin{equivlist}
\item
The map ${\rm sp}^{\tor}\colon \CT^{\tor}_{E'} \to \CT^{\tor}_{k}$ is injective (and hence an isomorphism and $\CT^{\tor}_{E'} \cong H^{2\bullet}(\Xbf\vdual)$ by \eqref{EqDescribeTaut}).
\item
The composition $A^{\bullet}(\B{G_{E'}}) \to A^{\bullet}(\Hdg_{E'}) \to A^{\bullet}(\Sbf^{\tor}_K)$ is zero in degree $> 0$.
\end{equivlist}
\end{proposition}

\begin{proof}
The commutative diagram \eqref{EqSpecialize} can be extended to a commutative diagram
\begin{equation}\label{EqKeySpecialize}
\begin{aligned}\xymatrix{
A^{\bullet}(\B{G_{E'}}) \ar@{^{(}->}[r] \ar[d]_{\cong} & A^{\bullet}(\Hdg_{E'}) \ar@{->>}[rr]^{\sigma^{\tor,*}} \ar[d]_{\cong} & & \CT^{\tor}_{E'} \ar[d] \\
A^{\bullet}(\B{G_{k}}) \ar@{^{(}->}[r] & A^{\bullet}(\Hdg_{\kappa'}) \ar@{->>}[r]_{\beta^*} & A^{\bullet}(\GZip^{\mu}) \ar[r]^-{\sim}_-{\zeta^{\tor,*}} & \CT^{\tor}_k.
}\end{aligned}
\end{equation}
Hence the equivalence follows as the kernel of $\beta^*$ is generated by the image of $A^{>0}(\B{G_{k}})$ by Theorem~\ref{PullbackKey}.
\end{proof}

Altough we cannot prove this description of the tautological ring in characteristic $0$, we can prove that analogous statement for cohomology:
\begin{theorem}
  For the $\mathbf{Q}_\ell$-algebra $H^{2 \bullet}(\Sbf^{\tor}_K) \defeq \oplus_{i}H^{2i}(\Sbf^{\tor}_K,\mathbf{Q}_\ell(i))$, the composition $$A^{\bullet}(\Hdg_{E'}) \to A^{\bullet}(\Sbf^{\tor}_K) \to H^{2\bullet}(\Sbf^{\tor}_K)$$ induces an injection $H^{2\bullet}(\Xbf\vdual) \into H^{2\bullet}(\Sbf^{\tor}_K)$.
\end{theorem}
\begin{proof}
 The existence of the factorization $A^\bullet(\Hdg_{E'}) \to H^{2\bullet}(\Xbf\vdual) \to H^{2\bullet}(\Sbf^{\tor}_K)$ is given by Corollary \ref{TrivialEtaleChern}. To prove injectivity, as in the proof of Corollary \ref{TrivialEtaleChern} one reduces to proving that the morphism $H^{2\bullet}(\Xbf\vdual) \cong A^\bullet(\GZip^\mu) \to A^\bullet(\Sbf^{\tor}_K) \to H^{2\bullet}(\Sbf^{\tor}_K)$ is injective in characteristic $p$. This follows from the argument used in the proof of Theorem \ref{ZetaTorInj}.
\end{proof}
%


\subsection{Hirzebruch-Mumford proportionality}

The above results immediately imply a very strong form of Hirzebruch-Mumford proportionality in positive characteristic and the usual form of Hirzebruch-Mumford proportionality in characteristic $0$.

Recall that an automorphic bundle on $S_{K,k}^{\tor}$ is by definition a vector bundle of the form $\sigma^{\tor,*}(\Escr)$ for some vector bundle $\Escr$ on $\Hdg_{k} = [G_{k}\bs G_{k}/P_{k}]$. Let $X\vdual \defeq G_{k}/P$ be the characteristic-$p$-version of the compact dual $\Xbf\vdual$ and let 
\[
\rho\colon X\vdual_{k} \to \Hdg_{k}
\]
be the projection. If we consider $\Escr$ as a $G_{k}$-equivariant vector bundle on $X\vdual_{k}$, then $\rho^*(\Escr)$ is the underlying vector bundle.

\begin{theorem}\label{HMCharp}
There is an isomorphism $u\colon A^{\bullet}(X\vdual) \liso \CT^{\tor}_k$ of graded $\BQ$-algebras such that for every $G_{k}$-equivariant vector bundle $\Escr$ on $X\vdual_{k}$ the $i$-th Chern class of the underlying vector bundle on $X\vdual_{k}$ is sent by $u$ to the $i$-th Chern class of the automorphic bundle $\sigma^{\tor,*}(\Escr)$.
\end{theorem}

\begin{proof}
The surjective map $\rho^*\colon A^{\bullet}(\Hdg_k) \to A^{\bullet}(X\vdual_k)$ has the same kernel as the surjective map $\beta^* \colon A^{\bullet}(\Hdg_k) \to A^{\bullet}(\GZip^{\mu}_k)$ by Lemma~\ref{RationalCI}. Hence we obtain some isomorphism of graded $\BQ$-algebras $A^{\bullet}(X\vdual) \iso A^{\bullet}(\GZip^{\mu}_k)$. Composing it with $\zeta^{\tor,*}\colon A^{\bullet}(\GZip^{\mu}_k) \to \CT^{\tor}_k$, which is an isomorphism by Theorem~\ref{ZetaTorInj}, we obtain the desired isomorphism $u$.
\end{proof}

For a smooth proper equi-dimensional scheme $X$ over $k$ we denote by $\int_X\colon A^{\dim X}(X) \to \BQ$ the degree map. Let $\BQ[c_1,\dots,c_d]$ be the graded polynomial ring with $\deg(c_i) = i$. 

The isomorphism $u$ from Theorem~\ref{HMCharp} induces in particular an isomorphism of the $1$-dimensional top degree parts $A^d(X\vdual_k)$ and $\CT^{\tor,d}_k$, where $d \defeq \dim(X\vdual) = \dim(S^{\tor}_K)$ we obtain:

\begin{corollary}\label{HMStandardCharp}
There exists a rational number $R \in \BQ^{\times}$ such that for all classes $\alpha \in A^{d}(\Hdg_k)$ one has
\[
\int_{X\vdual_k} \rho^*(\alpha) = R \int_{S_{K,k}^{\tor}} \sigma^{\tor,*}(\alpha).
\]
\end{corollary}

%

As specialization of cycles commutes with taking degrees we obtain a new and purely algebraic proof of Hirzebruch-Mumford proportionality in characteristic $0$.

\begin{corollary}\label{HMChar0}
There exists a rational number $R \in \BQ^{\times}$ such that for all homogenous $f \in \BQ[c_1,\dots,c_d]$ of degree $d$ and all $\Gbf_{\BC}$-equivariant vector bundles $\Escr$ on $\Xbf\vdual$ one has
\[
\int_{\Xbf\vdual} f(c_1(\rho^*(\Escr)),\dots,c_d(\rho^*(\Escr))) = R \int_{\Sbf_{K,\BC}^{\tor}} f(c_1(\sigma^{\tor,*}(\Escr)),\dots,c_d(\sigma^{\tor,*}(\Escr))).
\]
\end{corollary}

\begin{proof}
All $\Gbf_{\BC}$-equivariant vector bundles $\Escr$ on $\Xbf\vdual$ are already defined over some splitting field $E'$ of $G$ that we may assume to be a finite extension of the reflex field. We now choose $p$ and $v'$ as in the proof of Corollary~\ref{TrivialEtaleChern}: Let $p$ be a prime number of good reduction for the Shimura datum $(\Gbf,\Xbf)$ such that exists an unramified place $v'$ of $E'$ over $p$. Let $v$ be the restriction of $v'$ to $E$ and let $\Sscr^{\tor}_K$ be a smooth proper toroidal compactification of $\Sscr_K$ over $O_{E_v}$ with generic fiber $\Sbf^{\tor}_{K,E_v}$. Consider the commutative diagram
\begin{equation}\label{EqSpecialize}
\begin{aligned}\xymatrix{
A^{\bullet}(\Xbf\vdual_{E'}) \ar[d]^{\rm sp} & A^{\bullet}(\Hdg_{E'}) \ar[l]_{\rho^*} \ar[r]^{\sigma^{\tor,*}} \ar[d]^{{\rm sp}} & A^{\bullet}(\Sbf^{\tor}_{K,E'}) \ar[d]^{\rm sp} \\
A^{\bullet}(X\vdual_{\kappa'}) & A^{\bullet}(\Hdg_{\kappa'}) \ar[l]_{\rho^*} \ar[r]^{\sigma^{\tor,*}} & A^{\bullet}(S^{\tor}_{K,\kappa'}),
}\end{aligned}
\end{equation}
where the vertical maps are given by specialization.
Then we have
\begin{align*}
\int_{\Xbf\vdual} f\bigl(c_1(\rho^*(\Escr)),\dots,\rho^*(c_d(\Escr))\bigr) 
&= \int_{X\vdual_{\kappa'}}{\rm sp}\bigl(f\bigl(c_1(\rho^*(\Escr)),\dots,c_d(\rho^*(\Escr))\bigr)\bigr) \\
&= \int_{X\vdual_{\kappa'}}\rho^*\bigl({\rm sp}\bigl(f(c_1(\Escr),\dots,c_d(\Escr))\bigr)\bigr) \\
&= R\int_{S^{\tor}_{K,\kappa'}}\sigma^{\tor,*}\bigl({\rm sp}\bigl(f(c_1(\Escr),\dots,c_d(\Escr))\bigr)\bigr) \\
&= R\int_{S^{\tor}_{K,\kappa'}}{\rm sp}\bigl(\sigma^{\tor,*}\bigl(f(c_1(\Escr),\dots,c_d(\Escr))\bigr)\bigr) \\
&= R\int_{\Sbf_{K,\BC}^{\tor}} f\bigl(c_1(\sigma^{\tor,*}(\Escr)),\dots,c_d(\sigma^{\tor,*}(\Escr))\bigr).
\end{align*}
Here the first and the last equality hold because taking the degree commutes with specialization (\cite[20.3(a)]{Fulton_Intersection}), and the third equality is a special case of Corollary~\ref{HMStandardCharp}.
%
%
%
\end{proof}

The proof shows that the numbers $R$ of Corollary~\ref{HMStandardCharp} and Corollary~\ref{HMChar0} coincide.


%% file: Ch8CycleEO.tex

\section{Examples} \label{EXAMPLE}

For a permutation $\pi \in S_n$ we also write $\pi = [\pi(1),\pi(2),\dots,\pi(n)]$. We will always denote by $\tau_{i,j}$ the transposition of $i$ and $j$. For any permutation $\sigma$ one has $\sigma\tau_{i,j}\sigma^{-1} = \tau_{\sigma(i),\sigma(j)}$.

\subsection{Siegel Case}
Fix $g \geq 1$. We consider the vector space $\BF_p^{2g}$ with the symplectic pairing
\begin{equation*}
  (( a_i)_i, (b_i)_i) \mapsto \sum_{1\leq i \leq g}a_i b_{2g+1-i} - \sum_{g+1 \leq i \leq 2g} a_i b_{2g+1-i}.
\end{equation*}
We take $G$ to be the resulting group of symplectic similitudes and let $\mu$ be the cocharacter of $G$ with weights $(1,\hdots,1,0,\hdots ,0)$ (with both weights having multiplicity $g$). 

Let $T$ be the group of diagonal matrices in $G$. We use the description of the Weyl group $W$ of $(G,T)$ given in \cite[Section A7]{VW_PEL}, i.e.,
\[
W = \set{w \in S_{2g}}{w(i) + w(i^{\bot}) = 2g+1},
\]
where $i^{\bot} \defeq 2g+1-i$. Its simple reflections are $s_i = \tau_{i,i+1}\tau_{2g-i,2g+1-i}$ for $i = 1,\dots,g-1$ and $s_g = \tau_{g,g+1}$. Every element $w \in W$ is uniquely determined by $w(1),\dots,w(g)$. As $G$ is split over $\BF_p$, the Frobenius $\varphi$ acts trivially on $W$.

We get a frame for the resulting zip datum by taking $T$ to be the above torus, $B$ the group of upper triangular matrices in $G$ and $z$ the canonical representative of the element $[1+g,\hdots ,2g, 1,\hdots g]$ of $W$. The types $I$ and $J$ of $P$ and $Q$ are both equal to $\{s_1,\hdots,s_{g-1}\}$ and
\[
{}^IW = \set{w \in W}{w^{-1}(1) < \dots < w^{-1}(g)} = \set{w \in W}{w^{-1}(g+1) < \dots < w^{-1}(2g)}.
\]
Elements of $T$ have diagonal entries $(t_1,\hdots ,t_g,t_g^{-1},\hdots, t_1^{-1})$. Hence if for $1\leq i \leq g$ we let $x_i \in X^*(T)$ be the character sending such an element to $t_i$ we obtain a basis $(x_1,\hdots, x_g)$ of $X^*(T)$ which induces an isomorphism $S \cong \BQ[x_1,\hdots, x_g]$. The element $z \in W$ is given by $z(i) = g+i$ for all $i = 1,\dots,g$. It acts on $X^*(T)$ via $x_i \mapsto -x_{g+1-i}$. We have
\begin{equation}\label{EqSiegelz}
zs_iz^{-1} = s_{g-i} \qquad \text{for all} \quad i = 1,\dots,g-1.
\end{equation} 

\subsubsection*{Computation of $\gamma(w)$}

For $w \in {}^IW$ set $\sigma_w \defeq \inn(wz)$. Then we find
\[
I_w = \bigcap_{m\geq1} I_w^{(m)}, \qquad I_w^{(m)} \defeq \set{s \in I}{\sigma_w^k(s) \in I\ \forall\,k = 1,\dots,m}
\]
by \eqref{EqIw}. For instance $s_i \in I_w^{(1)}$ if and only if $w(g-i)$ and $w(g+1-i)$ are both $\leq g$ or both $\geq g+1$. In this case $w(g+1-i) = w(g-i) + 1$ and $\sigma_w(s)$ is $s_{w(g-i)}$ if $w(g-i) \leq g$ and it is $s_{w(g+1-i)^{\bot}}$ if $w(g-i) \geq g+1$.

We can consider $I_w$ as a subset of vertices of the Dynkin diagram of $G$ and get a subgraph with those edges in the Dynkin diagram of $G$ that have vertices in $I_w$. As $\sigma_w$ preserves angles between roots, it is an automorphism of the Dynkin diagram $I_w$. In particular it permutes all connected components of $I_w$. Let $\mathfrak{c}$ be a $\alpha_w^{\BZ}$-orbit of such connected components. Choose some connected component $C$ in $\mathfrak{c}$. Let $m(\mathfrak{c})$ be the number of vertices in $C$ and let $l(\mathfrak{c})$ be the minimal integer $n \geq 1$ such that $\sigma_w^n(C) = C$. We say that $\mathfrak{c}$ is of linear type if $\sigma^{l(\mathfrak{c})}(w) = w$ for all $w \in C$. Otherwise it is called of unitary type. Then $m(\mathfrak{c})$, $l(\mathfrak{c})$, and the type do not depend on the choice of $C$.

Now we can calculate $\gamma(w)$ by \eqref{EqGammaw} as follows. If $\mathfrak{c}$ is of linear type, then we let $\gamma_{\mathfrak{c}}(w)$ be the number of $\BF_p$-valued points in the full flag variety of the scalar restriction of $\GL_{m(\mathfrak{c})+1}$ over $\BF_{p^{l(\mathfrak{c})}}$, i.e.,
\[
\gamma_{\mathfrak{c}}(w) = \sum_{\pi \in S_{m(\mathfrak{c})+1}}p^{l(\mathfrak{c})\ell(\pi)} = \prod_{1 \leq j \leq m(\mathfrak{c})} \frac{q^{j+1} - 1}{q-1}.
\]
where $q \defeq p^{l(\mathfrak{c})}$.
If $\mathfrak{c}$ is of unitary type, then we let $\gamma_{\mathfrak{c}}(w)$ be the number of $\BF_p$-valued points in the full flag variety of the scalar restriction of a unitary group in $m+1$ variables over $\BF_{p^{l(\mathfrak{c})}}$. To describe this concretely, we let $\tau$ be the conjugation with the longest element in the symmetric group $S_{m(\mathfrak{c})+1}$, an automorphism of Coxeter groups order $2$ (except if $m(\mathfrak{c}) = 1$). For $\pi \in S_{m(\mathfrak{c})+1}$ set $\delta(\pi) \defeq p^{2l(\mathfrak{c})\ell(\pi)}$ if $\pi \ne \tau(\pi)$ and $\delta(\pi) \defeq p^{l(\mathfrak{c})\ell(\pi)}$ if $\tau(\pi) = \pi$. Then
\[
\gamma_{\mathfrak{c}}(w) = \sum_{\pi \in S_{m(\mathfrak{c})+1}/\tau}\delta(\pi),
\]
Altogether we obtain
\[
\gamma(w) = \prod_{\mathfrak{c}} \gamma_{\mathfrak{c}}(w),
\]
where $\mathfrak{c}$ runs trough all orbits of connected components of $I_w$.

For instance fix $0 \leq f \leq g$ and let $u_f$ be the permutation
\[
u_f \defeq [g+1,g+2,\dots,g+f,1,g+f+1,\dots, 2g-1, 2, \dots, g-f, 2g, g-f+1,\dots,g].
\]
Then $u_g = z$ and $\overline{Z_{u_0}}$ is the locus where the $p$-rank is $0$. We have
\[
I_{u_f} = I \setminus \{s_1,s_2,\dots,s_{g-f}\}
\]
in particular $I_{u_1} = I_{u_0} = \emptyset$. Moreover, $I_{u_f}$ has only one connected component and this is of linear type. Therefore
\[
\gamma(u_f) = \prod_{1 \leq j \leq f-1}\frac{p^{j+1} - 1}{p-1}.
\]

\subsubsection*{Cycle classes}

By Example~\ref{ExampleClassDiag} we find that
\begin{equation}\label{EqDiagSp}
[\Brh_e] = \prod_{1\leq i< j \leq g}(x_i \otimes 1 - 1 \otimes x_j)\Gamma_g(c_1,\dots,c_g),
\end{equation}
where $c_i = \sigma_i(x_1,\dots,x_g) \otimes 1 + 1 \otimes \sigma_i(x_1,\dots,x_g)$ for $i = 1,\dots,g$ and we set $c_0 = 2$ and $c_i = 0$ for all $i \notin \{0,\dots,g\}$. 

The operators $\delta_{s_i}$ act on $ S$ by
\begin{equation*}
  \delta_{s_i}(f)=\frac{f(x_1,\hdots,x_g)-f(x_1,\hdots,x_{i+1},x_i,\hdots x_g)}{x_i-x_{i+1}}
\end{equation*}
for $i=1,\hdots ,g-1$ and by
\begin{equation*}
  \delta_{s_g}(f)=\frac{f(x_1,\hdots,x_g)-f(x_1,\hdots,x_{g-1},-x_g)}{2x_g}.
\end{equation*}
for $i=g$. 

The element $z$ acts on $S$ by $x_i \mapsto -x_{g+1-i}$. Since the torus $T$ is split over $\BF_p$, the Frobenius $\phi$ acts on $S$ by $x_i \mapsto px_i$. Hence $\psi^*$ sends $x_i \otimes 1$ to $-x_{g+1-i}$ and $1\otimes x_i$ to $px_i$. Thus for $w \in W$ we find
\begin{equation*}
  [\overline{\Brh}_w]=\delta_w(\prod_{1\leq i< j \leq g}(x_i \otimes 1 - 1 \otimes x_j)\Gamma_g(c_1,\dots,c_g))
\end{equation*}
and
\begin{equation} \label{EqZEmptySp}
  [\overline{Z}^\emptyset_w] = \psi^*([\overline{\Brh}_w]).
\end{equation}

Such a formula is already given in \cite[Theorem 12.1]{EkedahlVanDerGeer_CycleAbVar}. The formula in \emph{loc.cit.} agrees with \eqref{EqZEmptySpin} if one takes the following into account: We believe that in \emph{loc.cit.} there is a typo and the polynomial should be evaluated at $y_j=p\ell_{g+1-j}$ instead of $y_j=p\ell_j$. Then the formulas agree under the substitution $x_i=\ell_{g+1-i}$.

\subsubsection*{The case $g=2$}
As an example, let us consider the case $g=2$. We let
\begin{align*}
  \Phi&\defeq x_1 \otimes 1 - 1\otimes x_2, \\
  \Gamma &\defeq c_1 c_2 = ( (x_1+x_2) \otimes 1 + 1\otimes (x_1+x_2))(x_1x_2\otimes 1 + 1\otimes x_1x_2)
\end{align*}
so that 
\begin{equation*}
  [\Brh_e]=\Phi\Gamma.
\end{equation*}
The set $\leftexp{I}{W}$ consists of the elements $\{e,s_2,s_2s_1,s_2s_1s_2\}$. By applying the operators $\delta_w$ we find
\begin{align*}
  [\Brh_{s_2}]&=  \Phi(x_1\otimes 1 + 1\otimes x_1)(x_1\otimes 1 + 1 \otimes x_2),\\{}
  [\Brh_{s_2s_1}]&= (x_1 \otimes 1 + 1\otimes x_1)(x_1\otimes 1 + 1\otimes x_2),\\{}
  [\Brh_{s_2s_1s_2}] &= x_1 \otimes 1 + 1\otimes x_1.
\end{align*}

Applying $\psi^*$ yields
\begin{align} \label{EqZEmptySp2}
  [\overline{Z}^\emptyset_{e}] &= -(p^4 -1)(x_1+x_2)x_1 x_2^2 ,\\{}
  [\overline{Z}^\emptyset_{s_2}] &= -(p^2-1)(px_1 - x_2)x_2^2, \nonumber \\{}
  [\overline{Z}^\emptyset_{s_2 s_1}] &= (p-1)(px_1 -x_2)x_2,\nonumber \\{}
  [\overline{Z}^\emptyset_{s_2 s_1 s_2}] &= px_1 - x_2. \nonumber
\end{align}

We have $I=I^{\rm o}=\{s_1\}$. Hence by Theorem $\pi_*=\delta_{s_1}$. Since $I$ has only a single element, we see that for $w \in \leftexp{I}{W}$ either $\sigma_w(s_1)=s_1$ and hence $I_w=I$ or $I_w =\emptyset$. In the first case we find $\gamma(w)=p+1$, in the second $\gamma(w)=1$. Using this we find:
\begin{equation}\label{EqGammawSp}
\gamma(e) = \gamma(s_2 s_1 s_2) = p+1, \qquad \gamma(s_2) = \gamma(s_2 s_1) = 1 
\end{equation}

Altogether we obtain the following formulas for the classes of the Ekedahl-Oort strata:
\begin{align} \label{EqZSp}
  [\overline{Z}_{e}] &= (p+1)(p^4-1)(x_1+x_2)x_1x_2,\\{}
  [\overline{Z}_{s_2}] &= (p^2-1)((p-1)x_1x_2-x_1^2-x_2^2), \nonumber \\{}
  [\overline{Z}_{s_2 s_1}] &= (p-1)(x_1+x_2),\nonumber \\{}
  [\overline{Z}_{s_2 s_1 s_2}] &= (p+1)^2,\nonumber
\end{align}
These formulas agree with the ones given in \cite[12.2]{EkedahlVanDerGeer_CycleAbVar} with $x_i$ corresponding to $\ell_{g+1-i}$, except that it appears that in \emph{loc.cit.} the rows for $s_1 s_2$ and $s_2 s_1$ should be switched and the entry for $\pi_*([\overline{U}_{s_2}])$ is incorrect.

\subsection{Hilbert-Blumenthal Case}

Fix $d \geq 1$ and let $\Gtilde \defeq \Res_{\BF_{p^d}/\BF_p}\GL_2$. Define $G$ by the cartesian diagram
\[\xymatrix{
G \ar[r] \ar[d] & \BG_{m,\BF_p} \ar[d] \\
\Gtilde \ar[r]^-{\det} & \Res_{\BF_{p^d}/\BF_p}\BG_m,
}\]
where the right vertical map is the canonical embedding. Let $\Sigma$ be the set of embeddings $\BF_{p^d} \mono k$. We fix an embedding $\iota_0$ and identify the set $\BZ/d\BZ$ with $\Sigma$  via $i \sends \sigma^{-i} \iota_0$, where $\sigma\colon x \to x^p$ is the arithmetic Frobenius. Let $\mgtilde$ be the cocharacter of $\Gtilde_k = \prod_{\Sigma}\GL_2$ given by $t \sends \twosmallmatrix{t}{0}{0}{1}$ in each component. Then $\mgtilde$ factors through a cocharacter $\mu$ of $G$. Let $\Ttilde$ be the standard torus of $\Gtilde$, i.e., $\Ttilde_k$ is the product of the diagonal tori. For $i \in \BZ/d\BZ$ and $j = 1,2$ let $x^{(i)}_j$ be the character
\[
\bigl(\twomatrix{t_1^{(i)}}{0}{0}{t_2^{(i)}}\bigr)_{i\in \BZ/d\BZ} \sends t_j^{(i)}
\]
of $\Ttilde_k$. Then $\Stilde = \Sym(X^*(\Ttilde)_{\BQ}) = \BQ[x^{(i)}_1,x_2^{(i)}; i \in \BZ/d\BZ]$. The intersection $T = \Ttilde \cap G$ is a maximal torus of $G$ and $S = \Sym(X^*(T)_{\BQ})$ identifies with the quotient of $\Stilde$ by the ideal generated by $(x^{(i)}_1 + x^{(i)}_2) - (x^{(i+1)}_1 + x^{(i+1)}_2))$ for $i \in \BZ/d\BZ$. We will compute all cycle classes of Ekedahl-Oort strata for $(\Gtilde,\mgtilde)$. This yields then also the corresponding cycle classes for $(G,\mu)$ by Subsection~\ref{FuncZip}.

Let $\Btilde$ be the Borel subgroup of $\Gtilde$ such that $\Btilde_k$ is the product of groups of upper triangular matrices in $\GL_2$. The Weyl group is $W = \{\pm 1\}^{\BZ/d\BZ}$ and we have $I = J = \emptyset$. Hence ${}^IW = W$. As frame for $(\Gtilde,\mgtilde)$ we choose $(\Ttilde,\Btilde,z)$ with $z$ a representative of $(-1,\dots,-1) \in W$.

By \eqref{DiagonalA} we have
\[
[\Brh_e] = \prod_{i \in \BZ/d\BZ}(x^{(i)}_1 \otimes 1 - 1 \otimes x^{(i)}_2) \in A^{\bullet}(\Brh_{\Gtilde}).
\]
Let $w = (\epsilon_i)_{i \in \BZ/d\BZ} \in W$. Then $\ell(w) = \#\set{i \in \BZ/d\BZ}{\epsilon_i = -1}$. We have
\[
[\overline{\Brh}_w] = \prod_{\substack{i \in \BZ/d\BZ \\ \epsilon_i = 1}}(x^{(i)}_1 \otimes 1 - 1 \otimes x^{(i)}_2)
\]
and hence
\[
[\overline{Z}_w^{\emptyset}] = \psi^*([\overline{\Brh}_w]) = \prod_{\substack{i \in \BZ/d\BZ \\ \epsilon_i = 1}}(x^{(i)}_2 - px^{(i+1)}_2).
\]
With the notation of Subsection~\ref{Gammaw} we find $I_w = \emptyset$ and $L_w = T$. Hence $\gamma(w) = 1$ for all $w \in W$. Moreover $\pi$ is an isomorphism. Therefore we have isomorphisms
\[
A^{\bullet}(\GZipFlag[\Gtilde]^{\mgtilde}) \cong A^{\bullet}(\GZip[\Gtilde]^{\mgtilde}) \cong A^{\bullet}(\GZip^{\mu}) \cong A^{\bullet}(\GZipFlag^{\mu})
\]
in this case. From the description of $A^{\bullet}(\GZipFlag^{\mu})$ in \eqref{ChowGZip2} one deduces easily that $x^{(i)}_2 \sends z_i$ yields an isomorphism of graded $\BQ$-algebras
\[
A^{\bullet}(\GZip^{\mu}) \cong \BQ[z_0,\dots,z_{d-1}]/(z_0^2, \dots ,z_{d-1}^2).
\]
Via this isomorphism we get for the cycle classes of the $\overline{Z}_w$
\[
[\overline{Z}_w] = \prod_{\substack{i \in \BZ/d\BZ \\ \epsilon_i = 1}}(z_i - pz_{i+1}) \in A^{\bullet}(\GZip^{\mu}).
\]
To describe the Hodge half line in $A^{\bullet}(\GZip^{\mu})$ we use the notation from \ref{HodgeHalf}. The restriction of the standard embedding $\iota$ of $G$ into $\GSp_{2d}$ to the maximal torus is given by
\[
\bigl(\twomatrix{t_1^{(i)}}{0}{0}{t_2^{(i)}}\bigr)_{i\in \BZ/d\BZ} \sends {\rm diag}(t^{(0)}_1, t^{(1)}_1, \dots, t^{(d-1)}_1, t^{(d-1)}_2, \dots, t^{(0)}_2).
\]
Therefore the character $\chi(\iota)$ (see \eqref{EqHodgeChar}) is given by
\[
\bigl(\twomatrix{t_1^{(i)}}{0}{0}{t_2^{(i)}}\bigr)_{i\in \BZ/d\BZ} \sends \prod_{i\in \BZ/d\BZ}(t_2^{(i)})^{-1}
\]
and the Hodge half line consists of all $\BQ_{>0}$-multiples of the class of
\[
\lambda \defeq -\sum_{i\in \BZ/d\BZ}x^{(i)}_2 = -(z_0 + \dots + z_{d-1}) \in A^{\bullet}(\GZip^{\mu}).
\]
Hence (as an illustration of Proposition~\ref{PowerHodge}) we see that
\begin{gather*}
[Z_{\leq d-1}] = \sum_{i\in \BZ/d\BZ}(z_i - pz_{i+1}) = (p-1)\lambda, \\
[Z_e] = (1+(-1)^dp^d) \prod_{i\in \BZ/d\BZ}z_i = \frac{p^d+(-1)^d}{d!}\lambda^d.
\end{gather*}


\subsection{Spin Case}

We assume that $p > 2$. Let $(V,Q)$ be a quadratic space over $\BF_p$ of odd dimension $2n+1 \geq 3$. Let $C(V) = C^+(V) \oplus C^-(V)$ its Clifford algebra. It is a $\BZ/2\BZ$-graded (non-commutative) $\BF_p$-algebra of dimension $2^{2n+1}$ generated as an algebra by the image of the canonical injective $\BF_p$-linear map $V \mono C(V)$. It is endowed with an involution ${}^*$ uniquely determined by $(v_1\cdots v_r)^* = v_r \cdots v_1$ for $v_1,\dots,v_r \in V$.

The \emph{spinor similitude group} is the reductive group $G = \GSpin(V)$ over $\BF_p$ defined by
\[
G(R) = \set{g \in C^+(V_R)^{\times}}{gV_Rg^{-1} = V_R, g^*g \in R^{\times}}.
\]
Then $g \sends (v \sends g\bullet v \defeq gvg^{-1})$ defines a surjective map of reductive groups $G \to \Gtilde \defeq \SO(V)$ with kernel $\BG_m$. The groups $G$ and $\Gtilde$ are both of Dynkin type $B_n$.

We now assume that we can find a $\BF_p$-basis $(v_0,v_1,\dots,v_{2n})$ such that the matrix of the bilinear form attached to $Q$ with respect to this basis is given by
\[
\begin{pmatrix}
1 \\
 & & & & & 1\\
 & & & & 1 \\
 & & & \addots \\
 & & 1 \\
 & 1
\end{pmatrix}.
\]
Although there are two isomorphism classes of quadratic spaces over $\BF_p$ of dimension $2n+1$, the associated groups $\GSpin$ and $\SO$ are isomorphic. Hence our assumption is harmless.

We define the cocharacter $\mu\colon \BG_m \to G$ by
\[
\mu(t) = tv_1v_{2n} + v_{2n}v_1.
\]
The composition of $\mu$ with $G \to \SO(V)$ yields the cocharacter
\[
\mgtilde\colon \BG_m \to \SO(V), \qquad t \sends {\rm diag}(1,t,1,\dots,1,t^{-1}).
\]
We will compute the cycle classes of Ekedahl Oort strata for $(\Gtilde,\mgtilde)$. Again by Subsection~\ref{FuncZip} this yields then also the corresponding cycle classes for $(G,\mu)$.

As a maximal torus $\Ttilde$ for $\Gtilde = \SO(V)$ we choose
\[
\Ttilde = \set{{\rm diag}(1,t_1,\dots,t_n,t_n^{-1},\dots,t_1^{-1})}{t_i \in \BG_m}.
\]
For $i = 1,\dots,n$ let $x_i$ be the character ${\rm diag}(1,t_1,\dots,t_n,t_n^{-1},\dots,t_1^{-1}) \sends t_i$. Then $\Stilde = \BQ[x_1,\dots,x_n]$. The Weyl group $W$ is the group
\[
W = \sett{w \in S_{2n}}{$w(i) + w(2n+1-i) = 2n+1$ for all $i$}
\]
acting on $T$ in the standard way via the last $2n$ coordinates.
The roots of $(\Gtilde,\Ttilde)$ are given by $\pm x_i \pm x_j$ for $1 \leq i \ne j \leq n$ and $\pm x_i$ for $i = 1,\dots,n$. Let $\Btilde$ be the Borel subgroup such that the corresponding simple roots are given by $x_1 - x_2, \dots, x_{n-1} - x_n, x_n$. Then the set $\Sigma$ of simple reflections in $W$ corresponding to $B$ is given by $s_1,\dots,s_{n-1},s_n$, where $s_i$ is the transposition $\tau_{i,i+1}\tau_{2n-i,2n+1-i}$ for $i = 1,\dots,n-1$ and $s_n = \tau_{n,n+1}$. 

Let $\ztilde \in \Gtilde$ be a lift of the element in $W$ with $1 \sends 2n$, $2n \sends 1$ and $i \sends i$ for all $i = 2,\dots,2n-1$. Then $(\Btilde,\Ttilde,\ztilde)$ is a frame by Lemma~\ref{DescribeFrame}.

As $I$ is the set of simple reflections corresponding to simple roots $\alpha$ with $\langle \mu,\alpha\rangle = 0$, we find $I = \{s_2,\dots,s_n\}$. Hence we find a bijection
\begin{equation}\label{EqIWSpin}
{}^IW = \set{w^{-1} \in W}{w(2) < w(3) < \dots < w(2n-1)} \iso \{1,\dots,2n\}, \qquad w \sends w^{-1}(1).
\end{equation}
Moreover $\ell(w) = w^{-1}(1) - 1$ for $w \in {}^IW$. By Lemma~\ref{LemPartialOrder}~\ref{LemPartialOrder2} and~\ref{LemPartialOrder4} this implies that the order $\preceq$ coincides with the Bruhat order on ${}^IW$ and that \eqref{EqIWSpin} is an isomorphism of ordered sets. There is the following concrete reduced expression of $w \in {}^IW$ as product of simple reflections
\begin{equation}\label{EqSpinRed}
w = \begin{cases}
s_1s_2\cdots s_{\ell(w)},&\text{if $\ell(w) \leq n$}; \\
s_1s_2\cdots s_ns_{n-1}\cdots s_{2n-\ell(w)},&\text{if $\ell(w) > n$.}
\end{cases}
\end{equation}
For all $s \in I$ one has ${}^zs = s$. As $\Gtilde$ is split over $\BF_p$, the Frobenius $\varphi$ acts trivial on $W$. Hence for $w \in {}^IW$ set subset $I_w \subseteq I$ defined in Subsection~\ref{Gammaw} is the largest subset such that ${}^wI_w = I_w$. Hence
\begin{equation}\label{EqIwSpin}
\begin{gathered}
I_e = I, \qquad I_{s_1} = I \setminus \{s_2\}, \qquad \dots \qquad I_{s_1\dots s_{n-1}} = \emptyset, \qquad I_{s_1\dots s_n} = \emptyset, \\
 I_{s_1\dots s_ns_{n-1}} = \emptyset, \qquad I_{s_1\dots s_ns_{n-1}s_{n-2}} = \{s_n\}, \qquad \dots \qquad I_{s_1\dots s_n\dots s_1} = I \setminus \{s_2\}.
\end{gathered}
\end{equation}
Hence $F\ell_w$ is the flag variety of a split group over $\BF_p$ of the following Dynkin type $B_k$, where
\[
k = \begin{cases}
n - 1 - \ell(w),&\text{if $\ell(w) \leq n-1$}; \\
0,&\text{if $\ell(w) = n$}; \\
\ell(w)-n-1,&\text{if $\ell(w) \geq n+1$}.
\end{cases}
\]
By Example~\ref{ExampleClassDiag} we find that
\begin{equation}\label{EqDiagSpin}
[\Brh_e] = \prod_{1\leq i< j \leq n}(x_i \otimes 1 - 1 \otimes x_j)\Gamma_n(c_1,\dots,c_n),
\end{equation}
where $c_i = \frac{1}{2}(\sigma_i(x_1,\dots,x_n) \otimes 1 + 1 \otimes \sigma_i(x_1,\dots,x_n))$ for $i = 1,\dots,n$ and we set $c_0 = 1$ and $c_i = 0$ for all $i \notin \{0,\dots,n\}$. For instance, if $n = 2,3$ we find
\begin{equation}\label{EqDiagSpinSmall}
[\Brh_e] = \begin{cases}
(x_1 \otimes 1 - 1 \otimes x_2)c_1c_2,&\text{if $n = 2$}; \\
\prod_{1 \leq i < j \leq 3}(x_i \otimes 1 - 1 \otimes x_j)c_3(c_1c_2 - c_3),&\text{if $n=3$}.
\end{cases}
\end{equation}

\subsubsection*{The case $n=2$}
From now on we assume that $n = 2$. Then the operators $\delta_{s_1}$ and $\delta_{s_2}$ from Subsection~\ref{Deltaw} act on $\Stilde = \BQ[x_1,x_2]$ by
\begin{align*}
\delta_1 \defeq \delta_{s_1}\colon f(x_1,x_2) &\sends \frac{f(x_1,x_2) - f(x_2,x_1)}{x_1-x_2}, \\
\delta_2 \defeq \delta_{s_2}\colon f(x_1,x_2) &\sends \frac{f(x_1,x_2) - f(x_1,-x_2)}{x_2}.
\end{align*}
Recall that we extend this action on $\Stilde \otimes \Stilde$ by letting $\delta_{s_i}$ on the first component.

We set
\begin{align*}
\Phi &\defeq x_1 \otimes 1 - 1 \otimes x_2, \\
\Gamma &\defeq c_1c_2 = \frac{1}{4}\bigl((x_1 + x_2) \otimes 1 + 1 \otimes (x_1 + x_2)\bigr)\bigl(x_1x_2 \otimes 1 + 1 \otimes x_1x_2\bigr).
\end{align*}
Note that
\begin{align*}
\delta_1(\Gamma) &= 0, \qquad \delta_2(\Phi) = 0, \qquad \delta_1(\Phi) = 1, \\
\delta_2(\Gamma) &= 2c_2 + s_2(c_1)\delta_2(c_2) = 2c_2 + 2(x_1^2 \otimes 1 - x_1x_2 \otimes 1 + x_1 \otimes (x_1 + x_2))
\end{align*}
since $\Gamma$ is symmetric in $x_1$ and $x_2$.
For $i = 1,\dots, 4$ let $w_i \in {}^IW$ be the element of length $i-1$. Then
\begin{align*}
[\overline{\Brh}_{w_1}] &= [\Brh_e] = \Phi\Gamma, \\
[\overline{\Brh}_{w_2}] = \delta_1([\Brh_e]) &= \delta_1(\Phi)\Gamma = \Gamma, \\
[\overline{\Brh}_{w_3}] = \delta_1\delta_2([\Brh_e]) &= \delta_1(\Phi \delta_2(\Gamma)) = \delta_2(\Gamma) + s_1(\Phi)\delta_1\delta_2(\Gamma), \\
[\overline{\Brh}_{w_4}] = \delta_1\delta_2\delta_1([\Brh_e]) &= \delta_1\delta_2(\Gamma) = \underbrace{\delta_1(s_2(c_1))}_{=1}\underbrace{\delta_2(c_2)}_{=x_1 \otimes 1} + s_1s_2(c_1)\underbrace{\delta_1\delta_2(c_2)}_{=\delta_1(x_1 \otimes 1) = 1}, \\
&= x_1 \otimes 1 + \frac{1}{2}\bigl((x_2-x_1)\otimes 1 + 1 \otimes (x_1 + x_2)\bigr).
\end{align*}
The element $\ztilde$ acts on $\Stilde$ by $x_1 \sends -x_1$ and $x_2 \sends x_2$. As $\Gtilde$ is split over $\BF_p$, the Frobenius $\varphi$ acts by $x_i \sends px_i$ for $i = 1,2$. We find
\begin{align*}
\psi^*(\Phi) &= -x_1 - px_2, \\
\psi^*(\Gamma) &= \psi^*(c_1)\psi^*(c_2) = \frac{1}{4}((p-1)x_1 + (p+1)x_2)(p^2 -1)x_1 x_2,\\
\psi^*(\delta_2(\Gamma))&= 2(1-p)x_1^2 + (1-p)x_1x_2
\end{align*}
and hence

\begin{align}\label{EqZEmptySpin}
[\overline{Z}^{\emptyset}_{w_1}] &= \psi^*(\Phi)\psi^*(\Gamma)  =-\frac{1}{4}(p^2-1)((p-1)x_1+(p+1)x_2)(x_1+px_2)x_1x_2, \\{}
[\overline{Z}^{\emptyset}_{w_2}] &= \psi^*(\Gamma)=\frac{1}{4}(p^2-1)((p-1)x_1+(p+1)x_2)x_1x_2, \nonumber \\{}
[\overline{Z}^{\emptyset}_{w_3}] &= 2(1-p)x_1^2 + (1-p)x_1x_2 - \psi^*(\Phi)[Z^{\emptyset}_{w_4}]=\frac{1-p}{2}((p+1)x_2^2+x_1^2-x_1x_2), \nonumber \\{}
[\overline{Z}^{\emptyset}_{w_4}] &= \frac{p-1}{2}x_1 + \frac{p+1}{2}x_2. \nonumber
\end{align}`

By \eqref{EqIwSpin} we find that
\begin{equation}\label{EqGammawSpin}
\gamma(w_1) = \#\BP^1(\BF_p) = p+1, \qquad \gamma(w_i) = 1 \quad\text{for $i = 2,3,4$}.
\end{equation}

We have $I^\circ=\leftexp{z}{I}=\{s_2\}$ and hence $\pi_*=\delta_2$ by Theorem \ref{GysinFormula}. Using this we find the following formulas for the classes of the closures of the Ekedahl-Oort strata:
\begin{align} \label{EqZSpin}
  [\overline{Z}_{w_1}] &= (p+1)\frac{1-p^2}{2} ((p^2+p)x_2^2+(p-1)x_1^2)x_1,\\{}
  [\overline{Z}_{w_2}] &= \frac{1}{2}(p^2-1)(p-1)x_1^2, \nonumber\\{}
  [\overline{Z}_{w_3}] &= (p-1)x_1, \nonumber\\{}
  [\overline{Z}_{w_4}] &= p+1. \nonumber
\end{align}